\documentclass[11pt]{article}

 \usepackage{amscd,amsmath,amsfonts,amssymb,amsthm,cite,mathrsfs,color} 
\usepackage{dsfont} \usepackage{leftidx}
\usepackage{tikz} \usepackage{enumerate}  \usepackage{hyperref}
 \usepackage{appendix}

\usepackage{geometry}

\usepackage{graphicx} \usepackage{xcolor}

\numberwithin{equation}{section} 

\usepackage{amsthm}

\theoremstyle{plain}
\newtheorem{theorem}{Theorem}[section]
\newtheorem{lemma}[theorem]{Lemma}
\newtheorem{corollary}[theorem]{Corollary}
\newtheorem{proposition}[theorem]{Proposition}

\theoremstyle{definition}
\newtheorem{definition}[theorem]{Definition}

\newtheorem{remark}[theorem]{Remark}

\usepackage{hyperref}
 
\begin{document}

\title{\bf{ Phase transition of eigenvalues in deformed Ginibre ensembles
}}

\author{
Dang-Zheng Liu\footnotemark[1] ~ and     Lu Zhang\footnotemark[2]}
\renewcommand{\thefootnote}{\fnsymbol{footnote}}
\footnotetext[1]{CAS Key Laboratory of Wu Wen-Tsun Mathematics, School of Mathematical Sciences, University of Science and Technology of China, Hefei 230026, P.R.~China. E-mail: dzliu@ustc.edu.cn}
\footnotetext[2]{School of Mathematical Sciences, University of Science and Technology of China, Hefei 230026, P.R.~China. E-mail: zl123456@mail.ustc.edu.cn}

 \maketitle
 \begin{abstract}

Consider a  random matrix  of size $N$ as  an additive deformation of  the complex  Ginibre ensemble   under  a deterministic    matrix $X_0$ with a  finite rank, independent of $N$. When  some eigenvalues  of  $X_0$   separate  from  the unit disk,  outlier eigenvalues  may appear  asymptotically in the  same locations,  and   their fluctuations   exhibit  surprising phenomena that highly  depend on  the Jordan canonical form  of $X_0$. These  findings   are  largely due to  Benaych-Georges and Rochet \cite{BR},  Bordenave and Capitaine \cite{BC16}, and Tao \cite{Ta13}.      When all  eigenvalues of $X_0$ lie inside  the unit disk, we prove that local eigenvalue statistics at the spectral edge form  a  new class  of   
determinantal point processes, for which   correlation  kernels are  characterized in terms of the repeated erfc integrals.  This thus  completes   a non-Hermitian analogue    of the BBP  phase  transition  in Random Matrix Theory.  Similar results  hold for the deformed real quaternion Ginibre ensemble.  
 \end{abstract}



\tableofcontents

\section{Introduction}

\subsection{Deformed non-Hermitian random matrices}
 Unlike eigenvalues  of Hermitian matrices,  those of non-Hermitian matrices  may  demonstrate instability both in theory and practice:  small perturbations in a large matrix can  result in  large fluctuations of eigenvalues.     As a canonical example, 
consider  the standard nilpotent matrix of size $N$ and    a modification to the $(N,1)$ entry 
\begin{equation*}\label{J0}
J_0=\begin{bmatrix}
0 & \mathbb{I}_{N-1}   \\
0   &   0
\end{bmatrix}, \quad J_{\varepsilon}=\begin{bmatrix}
0 & \mathbb{I}_{N-1}   \\
\varepsilon   &   0
\end{bmatrix},
\end{equation*} 
where  $\mathbb{I}_{N}$ denotes  an identity matrix  of size $N$ and   $\varepsilon>0$.  Obviously,  $J_0$   has only zero  eigenvalue with multiplicity $N$, however, 
$J_{\varepsilon}$    has  $N$ eigenvalues at the $N^{th}$ roots  $\varepsilon^{1/N} e^{\mathrm{i}j/N}$, $j=0,1,\ldots, N-1$ (see e.g.  \cite[Chapter 11]{BS10} or \cite[Chapter 2.8]{Tao12}). Thus,  even for exponentially small  values of $\varepsilon$, say  $2^{-N}$,   the eigenvalues of $J_{\varepsilon}$   can be far from those of  $J_{0}$ and     are indeed distributed uniformly  on the circle  of radius $1/2$.   
In practice, 
   letting  $U_N$ be a Haar-distributed orthogonal 
   matrix and 
asking a computer  to calculate the eigenvalues of  a unitarily equivalent matrix $U_{N} J_{0} U_{N}^{*}$,  as observed by Bordenave and Capitaine,  
 a really amazing phenomenon occurs: for large $N$, say $N=2000$, a lot of   eigenvalues wander over the unit circle;   see  \cite [Figure 1.1]{BC16}.
 
  One explanation  for  this  phenomenon can be  through the notion of 
    {\it pseudospectrum},  see a  comprehensive treatise    \cite{TE} by Trefethen and Embree. An alternative explanation,   as suggested by Bordenave and Capitaine \cite{BC16}, 
 in the spirit of von Neumann and Goldstine \cite{VG}, Spielman and Teng \cite{ST} and Edelman and Rao \cite{ER},  is to approximate numerical rounding errors by introducing randomness and study  spectra of a  randomly perturbed matrix 
 \begin{equation} \label{defmat}X:=X_0 +\frac{1}{\sqrt{N}}G. \end{equation}
 Here  $X_0$   is a deterministic  matrix, and    $G=[g_{ij}]_{i,j=1}^N$ is a random  matrix  whose entries   are i.i.d.  
 complex (real or quaternion) random variables with mean 0 and variance 1.  When  $G$  takes a  matrix-valued Brownian motion, \eqref{defmat}  becomes a dynamical random matrix  and remains something of a mystery; see  \cite{BD,BGN14,Sn}.

 In the non-perturbative  case,  that is,  $X_0=0$ in \eqref{defmat},  the study of  non-Hermitian random matrices  was first initiated  by  Ginibre  \cite{Gi}  for those  matrices with i.i.d.  real, complex and quaternion Gaussian entries (real, complex and quaternion Ginibre ensembles in the literature), and then was extended to  i.i.d.  entries.   At a macroscopic  level,     the limit spectral   measure was proven     under certain  moment assumptions  to be    the famous \textit{circular law}, which  is   a uniform measure on the unit disk,   after  a long list of works including Bai \cite{Ba},  Girko \cite{Gir},   G\"{o}tze and Tikhomirov \cite{GT}, Pan and Zhou  \cite{PZ},  and Tao and Vu \cite{TV10}. In particular, Tao and Vu \cite{TV10} established the law with the minimum assumption that  entries have finite variance.  See a recent survey \cite{BC12} and references therein for more details. 
  However, at  a   microscopic level,     local  eigenvalue statistics  in the bulk and at the  soft edge  of the spectrum were  revealed    first  for three  Ginibre ensembles   \cite{BS, FH,Ka},  with the help of  exact eigenvalue  correlations.     These Ginibre statistics  are
  conjectured to hold true even  for    i.i.d.   random matrices,   although  the proofs  seems much more difficult  than  the Hermitian analogy.  Later  in \cite{TV15},  Tao and Vu  established  a four moment match theorem  in the bulk and at the  soft edge, 
  both for  real and complex cases.  Very recently,  
    Cipolloni,  Erd\H{o}s and  Schr\H{o}der  remove the four moment
matching condition from \cite{TV15} and prove   the edge universality for any   random matrix  with  i.i.d.  entries  under the assumption of higher moments \cite{CES},   both in real and complex cases. But their  method   doesn't  seem to work in the bulk.  For more details about local universality  of  non-Hermitian random matrices,  see \cite{CES,TV15} and references  therein.

 For the deformed model  \eqref{defmat},    the  circular law still holds  under a  small low rank perturbation from $X_0$, see   \cite[Corollary 1.12]{TV10} as the culmination of work by many authors and an  the excellent survey [26] for further details. 
 Particularly   when  $X_0$ has bounded rank and bounded operator norm,    however,  such a perturbation can  create outliers  outside the unit disk  \cite{Ta13}.  Indeed,  
 when  the entries of $G$  have finite fourth moment, Tao  proved that outlier eigenvalues outside the unit disk are stable in the sense that outliers of  the perturbed matrix $X$   and  the deterministic $X_0$   coincide asymptotically \cite[Theorem 1.7]{Ta13}.  Since then,  this outlier phenomenon has been  extensively  studied  in \cite{BC16,BZ,BR,COW,OR}.  Furthermore, 
 the fluctuations of  outlier eigenvalues are  investigated respectively by 
 Benaych-Georges and Rochet   for   deformed invariantly matrix ensembles \cite{BR},  and   
 by Bordenave and Capitaine for deformed i.i.d. random matrices \cite{BC16}.   These fluctuations,  due to non-Hermitian structure,  become much more complicated, and indeed  highly depend on the shape of the Jordan canonical form of the perturbation.

 As mentioned previously, when  the strength of the added perturbation is above a certain threshold (supercritical regime), extreme eigenvalues will stay away from the bulk   and have been understood  well \cite{Ta13,BR,BC16}.  Otherwise,   they stick to the bulk, but  what would happen?      
 Much less  is known about  extreme eigenvalues of non-Hermitian random matrices  in later case.     This phenomenon has been well  studied for deformed  Hermitian random matrices, and   is now called  as  BBP phase transition after the work of  Baik, Ben~Arous and P{\'e}ch{\'e}  \cite{BBP}.  Our  main goal  in this paper is to understand  what effect  the finite rank  perturbation has on extreme eigenvalues  of the deformed model \eqref{defmat} in the subcritical and critical regimes,  and together with outlier fluctuations,   further to     complete a non-Hermitian analogue of  the  BBP phase transition.

 .


\subsection{Main results}

We will begin our study by specifying two  of the three deformed Ginibre ensembles- random matrices with all entries complex, or real quaternion. The independent elements are taken to be distributed as independent Gaussians with the same variance,  but possibly with the different  means.
 
 \begin{definition}  \label{GinU} 
A random  complex  
$N\times N$ matrix  $X$,    is said to belong to the deformed complex  Ginibre ensemble with mean matrix $X_0={\rm diag}\left( A_0,0_{N-r} \right)$,  where $A_0$ is a $r\times r$ constant  matrix,  denoted by GinUE$_{N}(A_0)$,  if  the joint probability density function for  matrix entries  is given by   
\begin{equation}\label{model}
P_{N}(A_0;X)=    \Big(\frac{N}{\pi}\Big)^{N^2}\
e^{ -N {\rm Tr} (X-X_0)(X-X_0)^*}.
\end{equation}
\end{definition}



Throughout the present  paper,  a  real quaternion  matrix  of size $N$ will  always be identified as its complex representation, that is,  a complex $2N\times 2N$ matrix  $X$   that   satisfies the relation
$X\mathbb{J}_N=\mathbb{J}_N\overline{X}$, where 
\begin{equation}\label{JN}
\mathbb{J}_N=\begin{bmatrix}
0_N & \mathbb{I}_N   \\
-\mathbb{I}_N    &   0_N
\end{bmatrix}. 
\end{equation} 
In this case it has the form 

\begin{equation}\label{42rep} X=\begin{bmatrix}
X^{(1)}  &  X^{(2)}  \\  -\overline{X^{(2)}}  &  \overline{X^{(1)}}
\end{bmatrix}, \end{equation}
where both $X^{(1)}$ and  $X^{(2)}$ are   $N\times N$ complex matrices; see  e.g.  \cite{Me,Wi}. 

 \begin{definition}  \label{GinS} 
A  random complex  $2N\times 2N$ matrix X of the form \eqref{42rep} is said to belong to the deformed  quaternion Ginibre ensemble with mean matrix $X_0$, 
 denoted by GinSE$_{N}(A_0)$,  if  the joint probability density function for  matrix entries  is given by   
\begin{equation}\label{4model}
P_{N,4}(A_0;X)=    \Big(\frac{2N}{\pi}\Big)^{2N^2}\
e^{ -N {\rm Tr} (X-X_0)(X-X_0)^*},
\end{equation}
where 
\begin{equation}\label{A04deno}
X_0=\begin{bmatrix} X_0^{(1)} & X_0^{(2)} \\ -\overline{X_0^{(2)}} & \overline{X_0^{(1)}}  \end{bmatrix}, \quad A_0=\begin{bmatrix} A_{0,1} & A_{0,2} \\ -\overline{A_{0,2}} & \overline{A_{0,1}}  \end{bmatrix},
\end{equation}
and $X_0^{(i)}={\rm diag}\left( A_{0,i},0_{N-r} \right)$ with $A_{0,i}$  a $r\times r$ constant matrix for $i=1,2$.
\end{definition}

 
 The following two propositions not only provide  joint density functions for eigenvalues of the GinUE$_{N}(A_0)$ and GinSE$_{N}(A_0)$ ensembles, but also relate eigenvalue correlation functions (see e.g. \cite{Me} for definition) to   auto-correlation functions  of characteristic polynomials.   They play a crucial   role  in  proving    local eigenvalue statistics for the  deformed Ginibre ensembles,   and  are believed to 
  have great potential applications which will be embodied in  subsequent papers in this series.
 
\begin{proposition}\label{intrep}
The eigenvalue density    of  the $\mathrm{GinUE}_{N}(A_0)$  matrix  $X$ 
  equals 
 \begin{equation}\label{densityeigen0}
\begin{aligned}
&f_{N}\left( A_0; z_1,\cdots,z_{N}\right)=\frac{1}{Z_{N,r}}
e^{-N\sum_{k=1}^{N}\left| z_k \right|^2} \prod_{1\leq i<j\leq N}\left| z_i-z_j \right|^2 \int \delta\left( V_1^*V_1-\mathbb{I}_r \right)  \\
&\times 
\exp\bigg\{ 
N  
\sum_{i=1}^N\left( z_i\overline{(V_{1}A_0 V_{1}^* )_{i,i}}+\overline{z_i}(V_{1}A_0 V_{1}^* )_{i,i}\right)
-N\sum_{1\leq j\leq i\leq N}|(V_{1}A_0 V_{1}^* )_{i,j}|^2
\bigg\}{\rm d}V_1,
\end{aligned}
\end{equation}
where $z_1, \ldots,z_N\in \mathbb{C}$,  $V_1$ is an $N\times r$ matrix and  the normalization constant 
\begin{equation}\label{correZNr}
Z_{N,r}= 
\pi^{N+\frac{1}{2}r(2N-r+1)}  N^{-\frac{1}{2}N(N+1)}  N!  \prod_{k=1}^{N-r-1}  k !.
\end{equation}
Moreover, when  $r+n\leq N$, the $n$-point correlation function    is  given by  
\begin{small}
\begin{equation}\label{algebraequa}
\begin{aligned}
&R_{N}^{(n)}\left( A_0; z_1,\cdots,z_{n} \right)=\frac{1}{C_N}
 e^{-N
\sum_{k=1}^{n}\left| z_k \right|^2} 
\prod_{1\leq i<j\leq n}\left| z_i-z_j \right|^2
\\
&\times  
 \int_{\mathcal{M}_{n,r}} \big(\det(\mathbb{I}_r-Q^*Q) \big)^{N-n-r}
e^{Nh(Q)}\,
{ {\mathbb{E}}_{\mathrm{ GinUE}_{N-n}( \widetilde{A}_0)}}\Big[\prod_{i=1}^n
\Big| \det\Big( \sqrt{\frac{N}{N-n}}z_i-\widetilde{X} \Big) \Big|^2\Big] {\rm d}Q,
\end{aligned}
\end{equation}
\end{small}
where  
 \begin{equation*}\label{Qintegraldomain}
\mathcal{M}_{n,r}=\left\{ Q\in \mathbb{C}^{n\times r} \big| Q^*Q\leq\mathbb{I}_r\right\},
\end{equation*}
\begin{equation} 
C_N= \pi^{n(r+1)} N^{-\frac{1}{2}n(n+1)} (N-n)^{-n(N-n)}     \prod_{k=N-n-r}^{N-r-1}  k !,
\end{equation}
\begin{equation}\label{OmegaQdenotion}
h(Q)
= 
\sum_{i=1}^n\left( z_i\overline{(QA_0 Q^* )_{i,i}}+\overline{z_i}(QA_0 Q^* )_{i,i}\right)+
\sum_{1\leq i<j\leq n}|(QA_0 Q^* )_{i,j}|^2 -{\rm Tr}\big(Q^* Q A_{0}^* A_0\big),
\end{equation} 
and
\begin{equation}\label{DeltaA0}
\widetilde{A}_0=\sqrt{N/(N-n)}\sqrt{\mathbb{I}_r-Q^*Q}A_0\sqrt{\mathbb{I}_r-Q^*Q}.
\end{equation}
\end{proposition}

\begin{proposition}\label{4intrep}
The eigenvalue density    of  the $\mathrm{GinSE}_{N}(A_0)$  matrix  $X$ 
  equals 
\begin{small}
 \begin{equation}\label{4densityeigen0}
\begin{aligned}
f_{N,4}( A_0; &z_1,\cdots,z_{N})=\frac{1}{Z_{N,r;4}}
e^{-2N\sum_{k=1}^{N}\left| z_k \right|^2} \prod_{i=1}^N\left| z_i-\overline{z_i} \right|^2
\prod_{1\leq i<j\leq N}\left| z_i-z_j \right|^2 \left| z_i-\overline{z_j} \right|^2        \\
\times &
 \int \delta\left( V_1^*V_1-\mathbb{I}_{2r} \right)  
\exp\bigg\{ 
2N\sum_{i=1}^N\left( z_i\overline{(V_{1}A_0 V_{1}^* )_{i,i}}+\overline{z_i}(V_{1}A_0 V_{1}^* )_{i,i}\right)
\bigg\}                  \\
&\times \exp\bigg\{-2N\sum_{1\leq j \leq i\leq N}\left( \left|(V_{1}A_0 V_{1}^* )_{i,j}\right|^2+
\left|(V_{1}A_0 V_{1}^* )_{i,j+N}\right|^2
\right)
\bigg\}{\rm d}V_{11}{\rm d}V_{21},
\end{aligned}
\end{equation}
\end{small}
where $\Im z_1,\cdots,\Im z_N\geq 0$, the normalization constant 
\begin{equation}\label{4correZNr}
Z_{N,r;4}= (2N)^{-2N(N+1)}
\pi^{N(2r+1)+r-r^2} N!  \prod_{k=1}^{N-r} (2k-1)!, 
\end{equation}
and  for  two  $N\times r$  complex matrices $V_{11}$ and  $V_{21}$,
 \begin{equation*} V_1=\begin{bmatrix} V_{11} & V_{21} \\ -\overline{V_{21}} & \overline{V_{11}}  \end{bmatrix}.
 \end{equation*} 
Moreover, when  $r+n\leq N$, the $n$-point correlation function is given by  
\begin{small}
\begin{equation}\label{4algebraequa}
\begin{aligned}
&R_{N,4}^{(n)}\left( A_0; z_1,\cdots,z_{n} \right)=\frac{1}{C_{N,4}}
 e^{-2N\sum_{k=1}^{n}\left| z_k \right|^2} \prod_{i=1}^N\left| z_i-\overline{z_i} \right|^2
\prod_{1\leq i<j\leq N}\left| z_i-z_j \right|^2 \left| z_i-\overline{z_j} \right|^2
\\
&\times 
 \int_{\mathcal{M}_{n,r;4}} \big(\det(\mathbb{I}_{2r}-Q^*Q) \big)^{N-n-r+\frac{1}{2}}
e^{Nh_4(Q)}\
{ {\mathbb{E}}_{\mathrm{ GinSE}_{N-n}( \widetilde{A}_0)}}\Big[\prod_{i=1}^n
\Big| \det\Big( \sqrt{\frac{N}{N-n}}z_i-\widetilde{X} \Big) \Big|^2\Big] {\rm d}Q,
\end{aligned}
\end{equation}
\end{small}
where  $dQ=dQ_{1}dQ_{2}$,  
 \begin{equation*}\label{4Qintegraldomain}
\mathcal{M}_{n,r;4}=\left\{ Q=\begin{bmatrix} Q_1 & Q_2 \\ -\overline{Q_2} & \overline{Q_1}  \end{bmatrix}
  \big| Q_1,\ Q_2\in \mathbb{C}^{n\times r} \,\&\,  Q^*Q\leq\mathbb{I}_{2r}\right\},
\end{equation*}
\begin{equation} \label{CN4cor}
C_{N,4}= \pi^{n(2r+1)} N^{-n(n+1)} (N-n)^{-2n(N-n)} 2^{-2nN+n^2-n} \prod_{k=N-n-r+1}^{N-r} (2k-1)!,
\end{equation}
and 
\begin{equation}\label{4OmegaQdenotion}
\begin{aligned}
h_4(Q)= 
&2\sum_{i=1}^n\left( z_i\overline{(QA_0 Q^* )_{i,i}}+\overline{z_i}(QA_0 Q^* )_{i,i}\right)-{\rm Tr}\big(Q^* Q A_{0}^* A_0\big)
\\
&+2\sum_{1\leq i<j\leq n}\left( |(QA_0 Q^* )_{i,j}|^2 
+|(QA_0 Q^* )_{i,j+n}|^2
\right),
\end{aligned}
\end{equation}
\begin{equation}\label{4DeltaA0}
\widetilde{A}_0=\sqrt{N/(N-n)}\sqrt{\mathbb{I}_{2r}-Q^*Q}A_0\sqrt{\mathbb{I}_{2r}-Q^*Q}.
\end{equation}
\end{proposition}

To state  our main results, we need to introduce  the  Jordan canonical  form of $A_0$. Given  a complex number $\theta$,  
a Jordan block  $R_{p}( \theta)$    is  a $p\times p$  upper triangular matrix of the form
\begin{small}
\begin{equation}\label{1Rpirealedge}
R_{p}\left( \theta\right)= 
\begin{bmatrix}
\theta & 1 & \cdots & 0 \\  & \ddots & \ddots & \vdots \\
&& \ddots & 1 \\ &&& \theta
\end{bmatrix}. 
\end{equation}
\end{small}For  a nonnegative integer  $\gamma$,  let  $\theta_1, \ldots, \theta_\gamma$  be (distinct) eigenvalues of $A_0$,  introduce a positive integer $\alpha_i$   and some   positive integers  
\begin{equation} \label{porder} p_{i,1}<p_{i,2}<\cdots <p_{i,\alpha_i}, \end{equation}
which correspond to the distinct sizes of the blocks associated with $\theta_i$, such that  $R_{p_{i,j}}( \theta_i)$ appears  $\beta_{i,j}$ times for $j=1, \ldots, \alpha_i$.   Hence, for a certain complex   invertible matrix $P$,     the Jordan canonical  form as
  a direct sum of block diagonal matrices exists for $\beta=2$
\begin{equation}\label{J2detail}
A_0=PJ_2P^{-1},\quad
J_2=\bigoplus_{i=1}^\gamma  \bigoplus_{j=1}^{\alpha_{i}}  
\overbrace{ R_{p_{i,j}}\left(\theta_i \right)   \bigoplus \cdots \bigoplus  R_{p_{i,j}}\left(\theta_i \right)
}^{\beta_{i,j}\ \mathrm{blocks}},
 \end{equation} where   $\theta_1, \ldots, \theta_{\gamma} \in \mathbb{C}$,   
 while    for $\beta=4$ 
\begin{equation}\label{J4detail} A_0=P\begin{bmatrix}
J_2 & \\ & \overline{J}_2
\end{bmatrix} P^{-1}, \quad  P=\begin{bmatrix}
P_1  &  P_2  \\  -\overline{P_2}  &  \overline{P_1}
\end{bmatrix},  \end{equation} where  $\Im \theta_1,\cdots,\Im \theta_\gamma\geq 0$; see  e.g. \cite{Wi} for the quaternion  Jordan decomposition.   To find out precisely what affect  some eigenvalues of $A_0$ have  on those of the deformed ensembles, we will suppose that  given an integer $0\leq m\leq \gamma$ and  $z_0 \in \mathbb{C}$, 
\begin{equation}\label{correbeta2edgethetaexpression}
 \theta_1=\cdots = \theta_m=z_0 \ \mbox{and}    \                                      
 \theta_i  \neq z_0,    \forall
i=m+1,\cdots,\gamma.
\end{equation}

We also need to define  two kinds of  functions,   depending on a real parameter  $n>-1$.   One is defined as 
\begin{equation} \label{IEF}
\mathrm{IE}_{n}(z)= \frac{1}{\sqrt{2\pi} \Gamma(n+1)}
\int_0^{\infty}v^n 
e^{  
-\frac{1}{2}(
v+z)^2}{\rm d}v,
\end{equation}
while  as a limit  of    $n>-1$ from  above 
\begin{equation} 
\mathrm{IE}_{-1}(z)= \frac{1}{\sqrt{2\pi}  }
e^{  -\frac{1}{2}z^2}.
\end{equation}
With integers  $n\geq -1$, these are  indeed  related to  repeated integrals of the complementary error function,   denoted by $ \mathrm{i^{n}}\mathrm{erfc}(z)$,   as follows 
  $$\mathrm{IE}_{n}(z)=2^{\frac{1}{2}n-1}  \mathrm{i^{n}}\mathrm{erfc}\Big(\frac{z}{\sqrt{2}} \Big).$$
Thus,  we have   recursive  relations  and series expansions  
\begin{equation*} 
\mathrm{IE}_{n}(z)=  
\int_0^{\infty}\mathrm{IE}_{n}(z+v){\rm d}v,
\end{equation*} and 
\begin{equation} 
\mathrm{IE}_{n}(z) =2^{-1-\frac{1}{2}n}\sum_{k=0}^{\infty} \frac{(-\sqrt{2}z)^k}{k! \Gamma(1+\frac{1}{2}(n-k))}, 
\end{equation}
see  \cite[(7.18.2) \&  (7.18.6) ]{Ol}.   The other is   a double integral given by 
\begin{equation} \label{fn2}
f_n(z,w)=\frac{1}{2\pi}e^{ \frac{1}{2}( z+w)^2}
\int_0^{\infty}
e^{-\frac{1}{2} v^2}\sinh (v(z-{w}))\int_{\frac{1}{\sqrt{2}}v}^{\infty}(u^2-v^2)^n
e^{-\frac{1}{2}(u+z+w)^2}{\rm d}u{\rm d}v.
 \end{equation}
 

At the  complex  and  real edges  of the spectrum,   two new kernels emerge and   are defined respectively as 
\begin{equation} \label{correkernel}
K_{n}\left( \hat{z},\hat{w} \right)=\sqrt{\frac{2}{\pi}}\Gamma(n+1) e^{ \frac{1}{2}( \hat{z} + \overline{\hat{w}})^2} \sqrt{
\mathrm{IE}_{n-1}( -\hat{z}- \overline{\hat{z}})
\mathrm{IE}_{n-1}( -\hat{w}- \overline{\hat{w}})
 }  \,\mathrm{IE}_{n}\big(\hat{z}+\overline{\hat{w}}\big),
\end{equation}
and \begin{equation} \label{correkernel4R}
K^{(\rm re)}_{n}\left( \hat{z},\hat{w} \right)=
 \sqrt{
\mathrm{IE}_{2n-1}( -\hat{z}- \overline{\hat{z}})
\mathrm{IE}_{2n-1}( -\hat{w}- \overline{\hat{w}})
 }  f_n(\hat{z}, \hat{w}). 
\end{equation}
Now, we are ready to  formulate  our main results concerning  scaling limits of  correlation functions at the edge  in the subcritical and critical regimes. 

\begin{theorem}\label{2-complex-correlation}
For the $ {\mathrm{GinUE}}_{N}(A_0)$ ensemble with the assumptions  \eqref{J2detail} and  \eqref{correbeta2edgethetaexpression}
on $A_0$, if $|z_0|=1$, then 
as $N\to \infty$ scaled eigenvalue correlations  hold uniformly for   all 
$\hat{z}_{1}, \ldots, \hat{z}_{n} $
in a compact subset of $\mathbb{C}$
\begin{equation} \label{corre2edgecomplex}
\frac{1}{N^n}R_N^{(n)}\Big( A_0;z_0+ \frac{\hat{z}_1}{\sqrt{N}},\cdots, z_0+ \frac{\hat{z}_n}{\sqrt{N}} \Big)=
\det\left(
\big[
K_{t}\left(z_0^{-1} \hat{z}_i, z_0^{-1}\hat{z}_j \big)
\right]_{i,j=1}^n
\right)
+O\big( N^{-\frac{1}{4}} \big),
\end{equation}
where $t=\sum_{i=1}^m\sum_{j=1}^{\alpha_i}\beta_{i,j}$ .

\end{theorem}

\begin{theorem}\label{4-complex-correlation}
For the $ {\mathrm{GinSE}}_{N}(A_0)$ ensemble with the assumptions  \eqref{J4detail} and  \eqref{correbeta2edgethetaexpression}
on $A_0$,  
as $N\to \infty$  scaled eigenvalue correlations  hold uniformly for   all 
$\hat{z}_{1}, \ldots, \hat{z}_{n} $
in a compact subset of $\mathbb{C}$.
For  $z_0=\pm 1$,  
\begin{equation} \label{corre4edgeR}
\begin{aligned}
&\frac{1}{N^n}R_{N,4}^{(n)}
\Big( A_0;z_0+ \frac{\hat{z}_1}{\sqrt{2N}},\cdots, z_0+ \frac{\hat{z}_n}{\sqrt{2N}} \Big)= \prod_{k=1}^n\big(z_0^{-1} \overline{\hat{z}_k}-z_0^{-1}\hat{z}_k \big)
\\
&\times \,{\rm Pf}\bigg(
 \begin{bmatrix}
K^{(\rm re)}_{t}(z_0^{-1}\hat{z}_i,z_0^{-1}\hat{z}_j) & K^{(\rm re)}_{t}(
z_0^{-1}\hat{z}_i,z_0^{-1}\overline{\hat{z}_j})
 \\
K^{(\rm re)}_{t}(z_0^{-1}\overline{\hat{z}_i},z_0^{-1}\hat{z}_j) 
&
K^{(\rm re)}_{t}(z_0^{-1}\overline{\hat{z}_i},z_0^{-1}\overline{\hat{z}_j})
\end{bmatrix}_{i,j=1}^n
\bigg)
+O\big( N^{-\frac{1}{4}} \big),
\end{aligned}
\end{equation}
while  for   $\Im z_0>0$ and $|z_0|=1$,     
\begin{equation} \label{corre4edgecomplex}
\frac{1}{N^n}R_{N,4}^{(n)}\Big( A_0;z_0+ \frac{\hat{z}_1}{\sqrt{2N}},\cdots, z_0+ \frac{\hat{z}_n}{\sqrt{2N}} \Big)=
\det\big(
\big[
K_{t}\left( z_0^{-1}\hat{z}_i,z_0^{-1}\hat{z}_j \big)
\right]_{i,j=1}^n
\big)
+O\big( N^{-\frac{1}{4}} \big),
\end{equation}
where $t=\sum_{i=1}^m\sum_{j=1}^{\alpha_i}\beta_{i,j}$.
\end{theorem}

Theorem \ref{2-complex-correlation} and  Theorem \ref{4-complex-correlation} show that local statistics near the edge $z_0$ only depends on the geometric multiplicity of the  eigenvalue $z_0$ of the added perturbation matrix $A_0$, that is,  the number of blocks associated with it. Together with fluctuations of outlier eigenvalues due to    Benaych-Georges\,\&\,Rochet \cite{BR}  and Bordenave\,\&\,Capitaine \cite{BC16} in the complex case,   this thus  completes  a phase transition for  complex eigenvalues in the deformed complex Ginibre ensemble.   It can be regarded as   a non-Hermitian analogue    of the BBP  phase  transition at the soft edge  of the spectrum 
   in Hermitian  Random Matrix Theory (RMT for short), first observed  by   Baik, Ben~Arous and P{\'e}ch{\'e}  for non-null complex sample covariance matrices \cite{BBP}.  
More precisely, extreme eigenvalues display distinct universal  patterns in  three  spectral regimes  created  by the perturbation: (I) In the subcritical regime where   all eigenvalues of $A_0$  lie in the open unit disk,  local statistics  is governed by the Ginibre edge  statistics; (II) In the critical regime where  some eigenvalues  are  on the unit circle and all others lie   inside  the open unit disk,  local statistics  is 
 completely characterized by  an infinite sequence of new kernels; (III) In the supercritical regime where  some  eigenvalues  stay away from the  unit disk,   outlier eigenvalues converge to roots of the eigenvalues of certain  finite random matrices, see  \cite[Theorem 2.10]{BR} and   \cite[Theorem 1.6]{BC16} for a more detailed description.

The remainder of  this paper is  organized as follows. In Section \ref{duality}, we  prove that  the ensemble  averages of certain observables admit   duality  identities   between different  Gaussian matrix ensembles.  In Section \ref{autocorrealtionI} and Section  \ref{autocorrealtionII},  we study auto-correlation functions of characteristic polynomials  in the deformed  complex and quaternion 
 Ginibre ensembles,  and further analyze   their scaling limits.  In Section \ref{proofs}, we provide   proofs of the main results,    Proposition \ref{intrep}, Proposition \ref{4intrep},  Theorem  \ref{2-complex-correlation}, and  Theorem  \ref{4-complex-correlation}. Finally, Section \ref{conclusion} contains some interpolating point processes.

\section{Duality in non-Hermitian RMT} \label{duality}

\subsection{Notation} \label{sectnotation}

By convention,  the conjugate, transpose,  and  conjugate transpose of a complex matrix $A$  are denoted by  $\overline{A}$,  $A^t$ and $A^*$,  respectively.   Besides, the following symbols may be used in this and subsequent sections.   

\begin{itemize}
\item   {\bf Pfaffian}
For a $2m\times 2m$  complex anti-symmetric matrix  $A=[a_{i,j}]$,  the  Pfaffian  
\begin{equation*}
{\rm Pf}(A)=\frac{1}{2^m m!}\sum_{\sigma\in S_{2m}}{\rm sgn}(\sigma) \prod_{i=1}^m a_{\sigma(2i-1),\sigma(2i)}.
\end{equation*}
 Property 1: 
\begin{equation*}\label{Pfaffian1}
{\rm Pf}(BAB^t)=\det(B)\,{\rm Pf}(A),
\end{equation*}
where $B$ is a $2m\times 2m$ complex matrix.  
Property 2:  
$$   
{\rm Pf}(A\otimes C)=(-1)^{\frac{1}{2}mq(q-1)}
\left( {\rm Pf}(A) \right)^q\left( \det(C) \right)^m,
  $$
  where   $C$ is any $q\times q$ complex symmetric matrix; see  
\cite[Theorem 4.2]{LC85}.

\item   {\bf Tensor}
 The tensor product, or Kronecker product,  of an $m\times n$  matrix  $A=[a_{ij}]$  and  a $p\times q$ matrix  $B$ is a block matrix $$
A \otimes B=\left[\begin{smallmatrix}
a_{11}B   &\cdots  &  a_{1n} B   \\ \vdots  & \ddots &   \vdots   \\
a_{m1}B   &\cdots  &  a_{mn} B   
\end{smallmatrix}\right].$$
When $m=n$ and $p=q$,  $A \otimes B$ is similar to  $B \otimes A$ by a permutation matrix.  
  
  \item   {\bf Trace}
  The trace of a square matrix  with real, complex or  real quaternion elements (corresponding to Dyson index $\beta=1, 2, 4$)  is redefined  as 
  $${\rm Tr_{\beta}}:=\begin{cases}
{\rm Tr}, & \beta=1,2; \\  
\frac{1}{2}{\rm Tr}, & \beta=4.
\end{cases}$$
  
  \item   {\bf  Norm}
     For a complex matrix $M$,  define the Hilbert-Schmidt norm  as $\|M\|:=\sqrt{{\rm{Tr}}(MM^*)}$.     $M=O(A_N)$ for some (matrix) sequence $A_N$  means that each  element  has the same order as  that     of $ A_N$. For two matrices $M_1$ and $M_2$, $\|M_1,M_2\|=\sqrt{{\rm Tr}(M_{1}M_{1}^*)+{\rm Tr}(M_{2}M_{2}^*)}$.

\item   {\bf  Index \& Submatrix}
Denote by   $I_{1}$ ($I_{2}$)  an   index set consisting of  numbers corresponding to   first (last)  columns of all 
 blocks  
 $R_{p_{i,j}}\left(\theta_i \right)$    from the Jordan form  $J_2$  in  \eqref{J2detail}  where  $ 1\leq i \leq m, 1\leq j \leq \alpha_i$, unless otherwise specified.  Both  have  the same  cardinality $\sum_{i=1}^m\sum_{j=1}^{\alpha_i}\beta_{i,j}$.
 For  a matrix $U$,     put   $U_{I_{1}}$ as a sub-matrix  by  keeping   the same rows and columns  with  index  in $I_1$.
 
\end{itemize}
\subsection{Duality formulae}

Our objects in considerations in this section are additive and multiplicative  deformations of  the Ginibre ensembles.    
\begin{definition}  \label{Ginb} 
A random real, complex or  real quaternion   
$N\times N$ matrix  $X$,  denoted by Dyson  index $\beta=1, 2, 4$ respectively,   is said to belong to the deformed Ginibre ensemble,   if  the joint density function for  matrix entries  is given by   
\begin{equation}\label{model}
P_{N,\beta}(\tau;X,X_0)=\frac{1}{Z_{N,\beta}}
\exp\!\left\{ -\frac{N}{\tau} {\rm Tr_{\beta}} \Big(\Sigma^{-1}(X-X_0)\Gamma^{-1}(X-X_0)^* \Big)\right\},
\end{equation}
with the normalization constant  
\begin{equation}\label{normalisation}
Z_{N,\beta}=\begin{cases}
\left(\det(\Sigma\Gamma)\right)^{\beta N/2} \left(\pi\tau/N\right)^{\beta N^2/2}, & \beta=1,2; \\  
\left(\det (\Sigma\Gamma)\right)^{N} \left( \pi\tau/N \right)^{2N^2}, & \beta=4.
\end{cases}
\end{equation}
Here $\tau>0$, correspondingly $\Sigma$,  $\Gamma$ and $X_0$  are   real, complex or real quaternion $N\times N $   matrices, and both $\Sigma$ and   $\Gamma$  are positive definite. 
\end{definition}

In RMT  there exist certain dual relations between $\beta$ and $4/\beta$  matrix  ensembles, say,  the Gaussian ensembles;  see  \cite{bh, bh01,De} and references therein.  In fact, there exist also   analogous  dual  relationships in non-Hermitian RMT.   Some special cases have been studied by Forrester and Rains  \cite{FR09}, 
Burda et al.  \cite{BGN14,BGN15}, 
Grela \cite{Gr}.  For this, we need to introduce some dual models associated with  the Ginibre ensembles.   Related  to  those are  three types of matrix spaces with complex entries:  symmetric matrices of size $K\times K$,    rectangular matrices of size $K_2\times K_1$,  antisymmetric  matrices  of size $K\times K$. For convenience, let's  endow them with Dyson index $\beta=1,2,4$ respectively,   since the associated  bounded    domains of complex matrices     correspond to   three  of six types of Cartan's    classical    domains  with Dyson circular $\beta$-ensembles as  boundaries; see \cite{Hua}.  

\begin{definition}  \label{dualdef}  A random    complex symmetric $K\times K$ matrix,  complex   $K_2\times K_1$ matrix or  
complex antisymmetric $K\times K$ matrix $Y$,  denoted by Dyson  index $\beta=1, 2, 4$ respectively,     is said to be the dual Ginibre ensemble with mean $Y_0$, if   the joint density function is given  by 
\begin{equation}\label{dualdensity}
\widehat{P}_{K,\beta}(\tau;Y,Y_0)=\frac{1}{\widehat{Z}_{K,\beta}}\exp\left\{-\frac{N}{\tau} {\rm Tr_{4/\beta}}(Y-Y_0)(Y-Y_0)^*\right\},
\end{equation}
with  the normalization constant
\begin{equation}\label{Zdual}
\widehat{Z}_{K,\beta}=\begin{cases}
2^{K}\left( \pi\tau/N\right)^{K(K+1)/2}, & \beta=1; \\ 
 \left( \pi\tau/N\right)^{K_1K_2}, & \beta=2;   \\
\left( \pi\tau/(2N)\right)^{K(K-1)/2}, & \beta=4.
\end{cases}
\end{equation}


\end{definition}




Also, we need to introduce certain  observables.  \begin{definition}  \label{Ob}  
For a  real, complex or  real quaternion   
$N\times N$ matrix  $X$ ($\beta=1, 2, 4$) ,  and  accordingly    for  a  complex antisymmetric $K\times K$, complex   $K_2\times K_1$ ($K_1\leq K_2$), 
or complex symmetric $K\times K$ matrix  $Y$ but with a dual index $4/\beta$,  define observables  $Q_\beta(A;X,Y)$ as follows. 
\begin{equation}\label{Q1}
Q_1(A;X,Y):=(-1)^{KN(KN-1)/2}\times    \\
{\rm Pf}
\begin{bmatrix}
Y\otimes \Sigma & A-\mathbb{I}_{K}\otimes X     \\
-A^t+\mathbb{I}_{K}\otimes X^t & Y^*\otimes \Gamma
\end{bmatrix},
\end{equation}
 
\begin{equation}\label{Q2}
Q_2(A;X,Y):=
\det
\begin{bmatrix}
A_1-\mathbb{I}_{K_1}\otimes X & -Y^*\otimes \Sigma     \\
  Y\otimes \Gamma        &              A_2 -     \mathbb{I}_{K_2}\otimes X^*
\end{bmatrix}, 
\end{equation}
and \begin{equation}\label{Q4}
Q_4(A;X,Y):=(-1)^{KN}
{\rm Pf}
\begin{bmatrix}
iY\otimes \left( \Sigma\mathbb{J}_N \right) & A-\mathbb{I}_{K}\otimes X    \\
-A^t+\mathbb{I}_{K}\otimes X^t & iY^*\otimes \left( \mathbb{J}_N\Gamma \right)
\end{bmatrix},
\end{equation}
 where   a complex square matrix $A$ has order  $KN$ and $2KN$   for $\beta=1, 4$, while 
$A=\mathrm{diag}(A_1, A_2)$  
with square matrices  $A_1$, $A_2$  of  order $K_1 N$ and $K_2N$ for $\beta=2$. 

Note that the  case of $\beta=2$ has essentially   been introduced in  \cite{BGN14, Gr}, where $\Sigma=\Gamma=\mathbb{I}_{N}$, $K_1=K_2$,  and $A_1, A_2$ are assumed to  be  tensor products of two matrices; see \cite[eqn (36)]{Gr}.

 Duality formulas  between the Ginibre ensembles and the dual Ginibre ensembles 
 given  in Definition  \ref{dualdef} can be stated as follows.

\end{definition}

\begin{theorem} \label{dualities}
With the same notations as in Definitions \ref{Ginb},  \ref{dualdef} and \ref{Ob},  one has 
\begin{equation}\label{dualequation}
\int Q_{\beta}(A;X,Y_0)P_{N,\beta}(\tau;X,X_0){\rm d}X=\int Q_{\beta}(A;X_0,Y) \widehat{P}_{K,4/\beta}(\tau;Y,Y_0){\rm d}Y,
\end{equation}
where $dX$ and $dY$ denote Lebesgue  measure on associated matrix spaces.
\end{theorem}

One immediate application of the duality formula \eqref{dualequation}  lies in  calculating autocorrelation  functions   of characteristic polynomials in the deformed ensembles. It  is believed to have  great interest   in its own right.  Another application   in the matrix-valued Brownian motion will be presented somewhere else \cite{LZ}.

  \begin{corollary}\label{characteristic}
 Let  $Z={\rm diag}\left( z_1,\cdots,z_{K_1} \right)$ and $W
={\rm diag}\left(w_1,\cdots, {w_{K_2}}\right)$ with  complex diagonal entries,  with the same notations as in Definitions \ref{Ginb},  \ref{dualdef} and \ref{Ob},
   then   
\begin{equation}\label{Charac}
\begin{aligned}
&\int \prod_{i=1}^{K_1}\det\left(z_i-X\right)\prod_{j=1}^{K_2}\det\left(\overline{w_j}-X^*\right) \,P_{N,\beta}(\tau;X,X_0){\rm d}X \\
&=\int Q_{\beta}(A;X_0,Y) \widehat{P}_{K,4/\beta}(\tau;Y,0){\rm d}Y,
\end{aligned}
\end{equation}
where  $K=K_1+K_2$ and 
\begin{equation} \label{A124}A=\begin{cases}
{\rm diag}\left( Z\otimes \mathbb{I}_N,W^*\otimes \mathbb{I}_N \right), & \beta=1, 2; \\ 
{\rm diag}\left( Z\otimes \mathbb{I}_{2N},W^*\otimes \mathbb{I}_{2N} \right), & \beta=4.
\end{cases}
\end{equation}
\end{corollary}


\begin{proof}
Choose $Y_0$ as a zero matrix  in \eqref{dualequation} and use the identity for  the Pfaffian 
\begin{equation}\label{Pfaffian0}
{\rm Pf}\begin{bmatrix} 0_n & A \\ -A^t & 0_n \end{bmatrix}=(-1)^{n(n-1)/2}\det A,
\end{equation}
where A is a complex  $n\times n$ matrix,  we can rewrite the product of the determinants  as $Q_\beta$ up to some sign.  This thus completes the proof.
\end{proof} 

\begin{remark}  Some special cases of  Corollary  \ref{characteristic}  have been studied  by Afanasiev\cite{Af}, Akemann-Phillips-Sommers \cite{APS},  Grela \cite{Gr}, Forrester-Rains \cite{FR09}, Fyodorov \cite{Fy18} and so on.  In the case of the  real Ginibre ensemble, that is,  $\beta=1$,   $\Sigma=\Gamma=\mathbb{I}_N, X_0=0$,   see   \cite{APS} for  the product of two characteristic polynomials and  for the product of arbitrarily  finite characteristic polynomials  \cite{Af}.   In the case of $\beta=2$,    $\Sigma=\Gamma=\mathbb{I}_N$ and $K_1=K_2$,  or  general $\Sigma, \Gamma$ and $K_1=K_2=1$,   see   \cite{Gr}.  For $\beta=1,2,4$  but with $X_0=0, \Gamma=\mathbb{I}_N$ and general $\Sigma$, see   \cite{FR09} for  the  moments of characteristic polynomials.   Besides,  for the  deformed  complex Ginibre ensemble given  in Definition  \ref{Ginb},  the average ratio of two  ``generalized" characteristic polynomials   and  the spectral density has  been exactly evaluated  as a  double integral   \cite{GG16};   for certain rank-one deviation from the  real Ginibre ensemble  the average modulus of the  characteristic polynomial  is   also evaluated  \cite[Appendix B]{BFI}.    These very  interesting results are expected to hold in  the more general  ensemble as in  Definition  \ref{Ginb}.
\end{remark}

\subsection{Diffusion method without SUSY}

\subsubsection{Two differential identities}
Supersymmetry  seems    nowadays indispensable for many problems in RMT,  particularly for  the first finding     of dual  formulas,  see e.g.   \cite{bh, bh01}  for  Hermitian RMT and  \cite{APS,BGN14,BGN15,Gr}  for non-Hermitian RMT.    For the  deformed complex Ginibre ensemble in  Definition  \ref{Ginb} 
but with    $\Sigma=\Gamma=\mathbb{I}_N$,      Grela   in \cite{Gr}   combines the heat equation  method (also  referred to as  diffusion method, see e.g.  \cite{Gr}) and supersymmetry (SUSY) technique  to obtain  Corollary  \ref{characteristic}.     Here,   instead of  SUSY,  we turn to make full  use of  elementary  matrix  operations 
    to  give a new derivation of  the Grela's duality formula,  and further  complete the proof  in the  real and  quaternion cases.     This  new idea  has many potential applications both  in Hermitian  and non-Hermitian  matrix ensembles, some of which will be given in the future study.  
In order to verify  the  duality formulas we need  two differential identities  related to  the determinants and Pfaffians.  

 For a matrix-valued function   $A:=A(x)=\left[ a_{ij}(x) \right]_{i,j=1}^n$ of variable $x$ without symmetry,  and   for some $1\leq p\leq n$ deleting $p$ rows and $p$ columns  respectively indexed by
  $I=\left\{ i_1,\cdots,i_p \right\}$ and  $J=\left\{ j_1,\cdots,j_p \right\}$ as  two subsets of   $ \{1,2,\ldots,n\}$, the resulting sub-matrix is denoted by $A\left[ I;J \right]$. 
By  definition of determinant it's easy to find an identity \begin{equation}\label{differentiationdeterminant}
\frac{{\rm d}}{{\rm d}x} \det\!\left( A \right)
=\sum_{i,j=1}^n (-1)^{i+j}\frac{{\rm d}a_{ij}(x)}{{\rm d}x}
\det\!\left( A[i;j] \right).
\end{equation}
For the  Pfaffian, we  have a similar identity.
\begin{proposition}\label{differentiationPfaffian}
Let $A=\left[ a_{ij}(x) \right]_{i,j=1}^{2m}$ be an antisymmetric  complex matrix, whose entries are differentiable functions of variable $x$, then
\begin{equation}\label{differentiationPfequa}
\frac{{\rm d}}{{\rm d}x}\, {\rm Pf}\!\left( A \right) 
=\sum_{1\leq i<j\leq 2m} (-1)^{i+j+1}\frac{{\rm d}a_{ij}(x)}{{\rm d}x}\,
{\rm Pf}\!\left( A\left[ i,j;i,j \right] \right).
\end{equation}
\end{proposition}
\begin{proof} Introduce  independent variables $\left\{ x_{ij}: 1\leq i<j\leq 2m \right\}$ and   treat 
$a_{ij}=-a_{ji}=a_{ij}( x_{ij})$ 
 as a function of  $x_{ij}$ whenever $i<j$, while all $a_{ii}=0$. 
By definition of the Pfaffian,
\begin{equation*}\label{diffPfsingle}
\frac{\partial}{\partial x_{ij}} {\rm Pf}\left( [ a_{kl}(x_{kl}) ]_{k,l=1}^{2m} \right) 
=(-1)^{i+j+1}\frac{{\rm d}a_{ij}(x_{ij})}{{\rm d}x_{ij}}{\rm Pf}\big( \left[ a_{kl}(x_{kl}) \right]_{k,l\neq i,j } \big),
\end{equation*}
so it suffices to prove
\begin{equation}\label{diffPfmulti}
\frac{{\rm d}}{{\rm d}x} {\rm Pf}\left( A\right) =\bigg(
\sum_{i<j}\frac{\partial}{\partial x_{ij}}
{\rm Pf}\left(  [ a_{kl}(x_{kl}) ]_{k,l=1}^{2m}  \right)
\bigg)  \bigg |_{\mathrm{all}\,x_{ij}=x}.
\end{equation}
Taking the derivative with respect to  $x_{ij}$   on both sides     \begin{equation*}
{\rm Pf}\left(  [ a_{kl}(x_{kl}) ]_{k,l=1}^{2m}  \right)=
\sum_{\sigma\in A_{2m}}{\rm sgn}(\sigma)\prod_{k=1}^m a_{\sigma(2k-1)\sigma(2k)}
\left(
x_{\sigma(2k-1)\sigma(2k)}
\right),
\end{equation*}
where
\begin{equation*}
A_{2m}=\left\{
\sigma\in S_{2m}:\sigma(1)<\sigma(3)<\cdots<\sigma(2m-1),    \sigma(2k-1)<\sigma(2k),  \forall k
\right\},
\end{equation*}
and summing them together,  we have
\begin{small}
\begin{equation*}
\begin{aligned}&\sum_{i<j}\frac{\partial}{\partial x_{ij}}   
{\rm Pf}\left(  [ a_{kl}(x_{kl}) ]_{k,l=1}^{2m}  \right)=
\sum_{\sigma\in A_{2m}}{\rm sgn}(\sigma)\sum_{i<j}\frac{\partial}{\partial x_{ij}}\left(\prod_{k=1}^m a_{\sigma(2k-1)\sigma(2k)}
(x_{\sigma(2k-1)\sigma(2k)})
\right)        \\
&=\sum_{\sigma\in A_{2m}}{\rm sgn}(\sigma)\sum_{k=1}^m
\frac{{\rm d}a_{\sigma(2k-1)\sigma(2k)}(x_{\sigma(2k-1)\sigma(2k)})}{{\rm d}x_{\sigma(2k-1)\sigma(2k)}}
\prod_{l=1,l\not=k}^m a_{\sigma(2l-1)\sigma(2l)}(
x_{\sigma(2l-1)\sigma(2l)}
).
\end{aligned}
\end{equation*}
\end{small}
On the right hand side,  set  $x_{ij}=x$ for all $i<j$,  \eqref{diffPfmulti}  immediately follows from the definition of the Pfaffian. This thus completes the proof. \end{proof}

The Laplace operator    
\begin{equation}\label{diffrealviacomplex}
\frac{\partial^2}{\partial x^2}+\frac{\partial^2}{\partial y^2}
=4\frac{\partial^2}{\partial z \partial \overline{z}}
\end{equation}
will be used frequently, 
where  $z=x+{\rm i}y$ and  $\overline{z}=x-{\rm i}y$.
 
\subsubsection{Proof of Theorem \ref{dualities}: real case}
On the 
left-hand  side   of \eqref{dualequation} use  
 change of variables $X$ to $\Sigma^{1/2}X\Gamma^{1/2}$, while  on both sides   of \eqref{dualequation}  replace nonrandom matrices $X_0$ and $ A$  by  $\Sigma^{1/2}X_0\Gamma^{1/2}$
 and  $\left(\mathbb{I}_{K} \otimes \Sigma^{1/2}\right) A \left(\mathbb{I}_{ K} \otimes \Gamma^{1/2}\right)$,
 respectively,   divide  both sides    by $\sqrt{\det (\Sigma\Gamma)}$,  we then  see  that  the resulting duality identity is independent of  $\Sigma$ and $\Gamma$.
Without loss of generality, we may assume
$ \Sigma=\Gamma=\mathbb{I}_N$.

Write 
\begin{equation}\label{Average2}
\overline{Q_1}(\tau;A,Y_0):=\int Q_1(A;X,Y_0) P_{N,1}(\tau;X,X_0) {\rm d}X,
\end{equation}
noting that when $ \Sigma=\Gamma=\mathbb{I}_N$  the density given in \eqref{model}   satisfies the heat equation 
\begin{equation}\label{OHeat2}
{\partial \tau}P_{N,1}(\tau;X,X_0)=\frac{1}{4N}\Delta_{1,X} P_{N,2}(\tau;X,X_0),\quad
\Delta_{1,X}:=\sum_{a,b=1}^N \partial_{ x_{a,b}}^2,
\end{equation}
using integration by parts   one  finds  
\begin{equation}\label{heatproducre}
\begin{aligned}
{\partial \tau}\overline{Q_1}(\tau;A,Y_0)     &  =\int   Q_1(A;X,Y_0) {\partial \tau}P_{N,1}(\tau;X,X_0)   {\rm d}X          \\
&=\frac{1}{4N} \int    Q_1(A;X,Y_0) \Big( \Delta_{1,X} P_{N,1}(\tau;X,X_0) \Big) {\rm d}X     \\
&=\frac{1}{4N} \int    \Big( \Delta_{1,X}   Q_1(A;X,Y_0)  \Big)  P_{N,1}(\tau;X,X_0) {\rm d}X.
\end{aligned}
\end{equation}

Rewrite the  antisymmetric   complex  matrix  $Y_0$ as   $Y_0=[y_{j,k}]=[a_{j,k}+{\rm i}  b_{j,k}]$, and let 
$$  \Delta_{Y_0}:=\sum_{1\leq j<k\leq K}  \big (\partial_{a_{j,k}}^2+\partial_{  b_{j,k}}^2\big). $$
We claim that 
\begin{equation} \label{VIPid1}
 \Delta_{1,X}   Q_1(A;X,Y_0) = \frac{1}{2}\Delta_{Y_0} Q_1(A;X,Y_0).
\end{equation}
If so,   one obtains  a heat equation  from  \eqref{heatproducre} that 
\begin{equation*}
{\partial \tau}\overline{Q_1}(\tau;A,Y_0) =\frac{1}{8N}\Delta_{Y_0}  \overline{Q_1}(\tau;A,Y_0). 
\end{equation*}
Together with  the  initial boundary  condition
\begin{equation*}\label{heatinitial1}
\overline{Q_1}(0;A,Y_0)=Q_1(A;X_0,Y_0),
\end{equation*}
the solution is thus given by 
\begin{equation*}\label{final1}
\overline{Q_1}(\tau;A,Y_0)=\int Q_1(A;X_0,Y) \widehat{P}_{K,4}(\tau;Y,Y_0)  {\rm d}Y,
\end{equation*}
from which   the desired result  immediately  follows.

The remaining task is to verify  the identity \eqref{VIPid1}. Put 
\begin{equation*}
\widetilde{T}_1=\begin{bmatrix}
Y_0\otimes \mathbb{I}_N & A-\mathbb{I}_{K}\otimes X     \\
-A^t+\mathbb{I}_{K}\otimes X^t & Y_0^*\otimes \mathbb{I}_N
\end{bmatrix},
\end{equation*}
 one sees from  \eqref{diffrealviacomplex}  that it suffices to prove
\begin{equation}\label{diffdualQ1}
 \Delta_{1,X}   {\rm Pf}\left( \widetilde{T}_1 \right) = 2\sum_{1\leq\alpha<\beta\leq K}
 \frac{\partial^2}{\partial y_{\alpha,\beta}\partial \overline{y_{\alpha,\beta}}}{\rm Pf}\left( \widetilde{T}_1 \right).
\end{equation}
By  use of \eqref{differentiationPfequa},    for $1\leq a,b\leq N$ one has
\begin{equation*}
\partial_{x_{a,b}}{\rm Pf}\left( \widetilde{T}_1 \right)=(-1)^{KN+a+b}\sum_{\alpha=1}^K
{\rm Pf}\left( \widetilde{T}_1\left[
I_{0,1};I_{0,1}
\right] 
\right),
\end{equation*}
where
\begin{equation*}
I_{0,1}=\left\{
(\alpha-1)N+a,(K+\alpha-1)N+b
\right\}.
\end{equation*}
Further use leads to 
\begin{small}
\begin{equation*}
\partial_{x_{a,b}}{\rm Pf}\left( \widetilde{T}_1\left[
I_{0,1};I_{0,1}
\right] 
\right)=(-1)^{KN+a+b+1}\sum_{\beta\not=\alpha}^K{\rm Pf}\left( \widetilde{T}_1\left[
I_1;I_1
\right] 
\right),
\end{equation*}
\end{small}
where 
\begin{equation*}\label{beta1idenotion}
I_1=\left\{
(\alpha-1)N+a,(\beta-1)N+a,(\alpha+K-1)N+b,(\beta+K-1)N+b
\right\}.
\end{equation*}
Hence,
\begin{equation*}\label{beta1DeltaX}
\sum_{a,b=1}^N\partial_{x_{ab}}^2{\rm Pf}\left( \widetilde{T}_1 \right)=
-2\sum_{1\leq\alpha<\beta\leq K}\sum_{a,b=1}^N {\rm Pf}\left( \widetilde{T}_1\left[
I_1;I_1
\right] 
\right).
\end{equation*}
On the other hand, for $\alpha<\beta$ it's easy to get 
\begin{equation*}\label{beta1DeltaY}
\frac{\partial^2}{\partial y_{\alpha,\beta}\partial \overline{y_{\alpha,\beta}}}{\rm Pf}\left( \widetilde{T}_1 \right)
=-\sum_{a,b=1}^N {\rm Pf}\left( \widetilde{T}_1\left[
I_1;I_1
\right] 
\right),
\end{equation*}
from which the summation over $\alpha$ and $\beta$  implies  \eqref{diffdualQ1}.  This thus completes the proof.

\subsubsection{Proof of Theorem \ref{dualities}: quaternion case}

We proceed  in a  similar way as in the real  case.  On the 
left-hand  side   of \eqref{dualequation} use  
 change of variables $X$ to $\Sigma^{1/2}X\Gamma^{1/2}$, while  on both sides   of \eqref{dualequation}  replace nonrandom matrices $X_0$ and $ A$  by  $\Sigma^{1/2}X_0\Gamma^{1/2}$
 and  $\left(\mathbb{I}_{K} \otimes \Sigma^{1/2}\right) A \left(\mathbb{I}_{ K} \otimes \Gamma^{1/2}\right)$,
 respectively,   divide  both sides    by $\sqrt{\det (\Sigma\Gamma)}$,  we then  see 
from the relations 
$$ \Sigma\mathbb{J}_N=\Sigma^{1/2}\mathbb{J}_N \big( \Sigma^{1/2} \big)^t, \qquad \mathbb{J}_N\Gamma=\big( \Gamma^{1/2}\big)^t \mathbb{J}_N\Gamma^{1/2}$$ 
  that  the resulting duality identity is independent of  $\Sigma$ and $\Gamma$.
Without loss of generality, we may assume
$ \Sigma=\Gamma=\mathbb{I}_{2N}$.

Write 
\begin{equation}\label{Average2}
\overline{Q_4}(\tau;A,Y_0):=\int Q_4(A;X,Y_0) P_{N,4}(\tau;X,X_0) {\rm d}X,
\end{equation}
where $X$ is given in \eqref{42rep} with $X^{(j)}=[x^{(j)}_{a,b}]$ ($j=1,2$),
noting that when $ \Sigma=\Gamma=\mathbb{I}_{2N}$  the density given in \eqref{model}   satisfies the heat equation 
\begin{equation}\label{OHeat4}
{\partial \tau}P_{N,4}(\tau;X,X_0)=\frac{1}{4N}\Delta_{4,X} P_{N,4}(\tau;X,X_0),\quad \Delta_{4,X}:=\sum_{j=1}^2\sum_{a,b=1}^N \big(\partial_{\Re x^{(j)}_{a,b}}^2+\partial_{\Im x^{(j)}_{a,b}}^2\big).
\end{equation}
By  integration by parts   one  finds  
\begin{equation}\label{heatproducre4}
\begin{aligned}
{\partial \tau}\overline{Q_4}(\tau;A,Y_0)     &  =\int   Q_4(A;X,Y_0) {\partial \tau}P_{N,4}(\tau;X,X_0)   {\rm d}X          \\
&=\frac{1}{4N} \int    Q_4(A;X,Y_0) \Big( \Delta_{4,X} P_{N,1}(\tau;X,X_0) \Big) {\rm d}X     \\
&=\frac{1}{4N} \int    \Big( \Delta_{4,X}   Q_4(A;X,Y_0)  \Big)  P_{N,4}(\tau;X,X_0) {\rm d}X.
\end{aligned}
\end{equation}

Rewrite the complex  symmetric     matrix   $Y_0$ as   $Y_0=[y_{jk}]=[a_{j,k}+{\rm i}  b_{j,k}]$, and let 
$$  \Delta_{Y_0}:=\sum_{1\leq j\leq k\leq K}  \big (\partial_{a_{j,k}}^2+\partial_{  b_{j,k}}^2\big)  +\sum_{ j=1}^{K}  \big (\partial_{a_{j,j}}^2+\partial_{b_{j,j}}^2\big). $$
We claim that 
\begin{equation} \label{VIPid4}
 \Delta_{4,X}   Q_4(A;X,Y_0) = \Delta_{Y_0} Q_4(A;X,Y_0).
\end{equation}
If so,   one obtains  a heat equation  from  \eqref{heatproducre4} that 
\begin{equation*}
{\partial \tau}\overline{Q_4}(\tau;A,Y_0) =\frac{1}{4N}\Delta_{Y_0}  \overline{Q_4}(\tau;A,Y_0). 
\end{equation*}
Together with  the  initial boundary  condition
\begin{equation*}\label{heatinitial1}
\overline{Q_4}(0;A,Y_0)=Q_4(A;X_0,Y_0),
\end{equation*}
the solution is thus given by 
\begin{equation*}\label{final4}
\overline{Q_4}(\tau;A,Y_0)=\int Q_4(A;X_0,Y) \widehat{P}_{K,1}(\tau;Y,Y_0)  {\rm d}Y,
\end{equation*}
from which   the desired result  immediately  follows.

The remaining task is to verify  the identity \eqref{VIPid4}. Put 
\begin{equation*}
\widetilde{T}_4=\begin{bmatrix}
{\rm i}Y_0\otimes \mathbb{J}_N & A-\mathbb{I}_{K}\otimes X     \\
-A^t+\mathbb{I}_{K}\otimes X^t & {\rm i}Y_0^*\otimes \mathbb{J}_N
\end{bmatrix},
\end{equation*}
one sees from  \eqref{diffrealviacomplex}   that  it suffices to prove
\begin{equation}\label{diffdualQ4}
\begin{aligned}
& \sum_{a,b=1}^N\left( \frac{\partial^2}{\partial x_{a,b}^{(1)}\partial \overline{x_{a,b}^{(1)}}}
 +\frac{\partial^2}{\partial x_{a,b}^{(2)}\partial \overline{x_{a,b}^{(2)}}} \right)
    {\rm Pf}\left( \widetilde{T}_4 \right) =   \\
&\left( \sum_{\alpha<\beta}  
\frac{\partial^2}{\partial y_{\alpha,\beta}\partial \overline{y_{\alpha,\beta}}}+
2\sum_{\alpha=1}^K \frac{\partial^2}{\partial y_{\alpha,\alpha}\partial \overline{y_{\alpha,\alpha}}}
\right)
 {\rm Pf}\left( \widetilde{T}_4 \right).
\end{aligned}
\end{equation}
By use of  \eqref{differentiationPfequa}, for $\alpha<\beta$ one has 
\begin{equation*}
\frac{\partial}{\partial y_{\alpha,\beta}}{\rm Pf}\left( \widetilde{T}_4 \right)
=\sum_{a=1}^N (-1)^{N}{\rm i}\left\{ -{\rm Pf} \left(\widetilde{T}_4\left[
I_{4,0};I_{4,0}
\right]
\right)+       
{\rm Pf} \left(\widetilde{T}_4\left[
J_{4,0};J_{4,0} 
\right]
\right)
\right\},
\end{equation*}
where
\begin{equation*}
I_{4,0}=\left\{ 2(\alpha-1)N+a,(2\beta-1)N+a \right\}, \quad
J_{4,0}=\left\{ (2\alpha-1)N+a,2(\beta-1)N+a \right\},
\end{equation*}
and further 
\begin{small}
\begin{equation*}\label{beta4DeltaY}
\begin{aligned}
\frac{\partial^2}{\partial y_{\alpha,\beta}\partial \overline{y_{\alpha,\beta}}}{\rm Pf}\left( \widetilde{T}_4 \right)
&=\sum_{a,b=1}^N \left\{ {\rm Pf}\left( \widetilde{T}_4\left[
J_{\alpha\beta}^4;J_{\alpha\beta}^4
\right] 
\right)
-{\rm Pf}\left( \widetilde{T}_4\left[
I_{\alpha\beta}^4;I_{\alpha\beta}^4
\right] 
\right)
\right\}  \\
&+\sum_{a,b=1}^N \left\{
{\rm Pf}\left( \widetilde{T}_4\left[
J_{\beta\alpha}^4;J_{\beta\alpha}^4
\right] 
\right)-{\rm Pf}\left( \widetilde{T}_4\left[
I_{\beta\alpha}^4;I_{\beta\alpha}^4
\right] 
\right)
\right\},
\end{aligned}
\end{equation*}
\end{small}
where 
\begin{small}
\begin{equation*}\label{beta4idenotion}
\begin{aligned}
&I_{\alpha\beta}^4=\left\{
2(\alpha-1)N+a,(2\beta-1)N+a,2(\alpha+K-1)N+b,(2\beta+2K-1)N+b
\right\}, \\
&J_{\alpha\beta}^4=\left\{
2(\alpha-1)N+a,(2\beta-1)N+a,(2\alpha+2K-1)N+b,2(\beta+K-1)N+b
\right\}.
\end{aligned}
\end{equation*}
\end{small}
For $1\leq \alpha \leq K$, one has
\begin{small}
\begin{equation*}\label{beta4DeltaY1}
\frac{\partial^2}{\partial y_{\alpha,\alpha}\partial \overline{y_{\alpha,\alpha}}}{\rm Pf}\left( \widetilde{T}_4 \right)
=-\sum_{a,b=1}^N 
{\rm Pf}\left( \widetilde{T}_4\left[
I_{\alpha\alpha}^4;I_{\alpha\alpha}^4
\right] 
\right).
\end{equation*}
\end{small}
Hence,
\begin{small}
\begin{equation*}\label{beta4DeltaY2}
\begin{aligned}
&\left(
\sum_{\alpha<\beta}\frac{\partial^2}{\partial y_{\alpha,\beta}\partial \overline{y_{\alpha,\beta}}}
+2\sum_{\alpha=1}^K\frac{\partial^2}{\partial y_{\alpha,\alpha}\partial \overline{y_{\alpha,\alpha}}}
\right)
{\rm Pf}\left( \widetilde{T}_4 \right)=
-2\sum_{a,b=1}^N\sum_{\alpha=1}^K{\rm Pf}\left( \widetilde{T}_4\left[
I_{\alpha\alpha}^4;I_{\alpha\alpha}^4
\right] 
\right)   \\
&+\sum_{a,b=1}^N\sum_{\alpha\not=\beta}\left\{{\rm Pf}\left( \widetilde{T}_4\left[
J_{\alpha\beta}^4;J_{\alpha\beta}^4
\right] 
\right)- {\rm Pf}\left( \widetilde{T}_4\left[
I_{\alpha\beta}^4;I_{\alpha\beta}^4
\right] 
\right) \right\}.
\end{aligned}
\end{equation*}
\end{small}
On the other hand, for $1\leq a,b\leq N$ one gets 
\begin{equation*}\label{beta4DeltaX1}
\frac{\partial^2}{\partial x_{a,b}^{(1)}\partial \overline{x_{a,b}^{(1)}}}
{\rm Pf}\left( \widetilde{T}_4 \right)=
-\sum_{\alpha,\beta=1}^K{\rm Pf}\left( \widetilde{T}_4\left[
I_{\alpha\beta}^4;I_{\alpha\beta}^4
\right] 
\right)
\end{equation*}
and
\begin{small}
\begin{equation*}\label{beta4DeltaX2}
\frac{\partial^2}{\partial x_{a,b}^{(2)}\partial \overline{x_{a,b}^{(2)}}}
{\rm Pf}\left( \widetilde{T}_4 \right)=
-\sum_{\alpha=1}^K{\rm Pf}\left( \widetilde{T}_4\left[
I_{\alpha\alpha}^4;I_{\alpha\alpha}^4
\right] 
\right)+\sum_{\alpha\not=\beta}{\rm Pf}\left( \widetilde{T}_4\left[
J_{\alpha\beta}^4;J_{\alpha\beta}^4
\right] 
\right).
\end{equation*}
\end{small}
from which the summation over $a$ and $b$  implies  \eqref{diffdualQ4}.  This thus completes the proof.

\subsubsection{Proof of Theorem \ref{dualities}: complex case}

On the 
left-hand  side   of \eqref{dualequation} use  
 change of variables $X$ to $\Sigma^{1/2}X\Gamma^{1/2}$, while  on both sides   of \eqref{dualequation}  replace nonrandom matrices $X_0, A_1$ and $ A_2$  by  $\Sigma^{1/2}X_0\Gamma^{1/2}$, $\left(\mathbb{I}_{K_1} \otimes \Sigma^{1/2}\right)A_1\left(\mathbb{I}_{K_1} \otimes \Gamma^{1/2}\right)$ and  $\left(\mathbb{I}_{K_2} \otimes \Gamma^{1/2}\right)A_2\left(\mathbb{I}_{K_{2}} \otimes \Sigma^{1/2}\right)$ respectively,   divide  both sides    by $\det (\Sigma\Gamma)$,  we then  see  the resulting duality identity is independent of  $\Sigma$ and $\Gamma$.
Without loss of generality, we may assume
$ \Sigma=\Gamma=\mathbb{I}_N$ and general $A$.

Write 
\begin{equation}\label{Average2}
\overline{Q_2}(\tau;A,Y_0):=\int Q_2(A;X,Y_0) P_{N,2}(\tau;X,X_0) {\rm d}X,
\end{equation}
noting that when $ \Sigma=\Gamma=\mathbb{I}_N$  the density given in \eqref{model}   satisfies the heat equation 
\begin{equation}\label{OHeat2}
{\partial \tau}P_{N,2}(\tau;X,X_0)=\frac{1}{4N}\Delta_{2,X} P_{N,2}(\tau;X,X_0),\quad
\Delta_{2,X}:=\sum_{a,b=1}^N (\partial_{\Re x_{a,b}}^2+\partial_{\Im x_{a,b}}^2),
\end{equation}
using integration by parts   one  finds  
\begin{equation}\label{heatproducre2}
\begin{aligned}
&{\partial \tau}\overline{Q_2}(\tau;A,Y_0)       =\int   Q_2(A;X,Y_0) {\partial \tau}P_{N,2}(\tau;X,X_0)   {\rm d}X          \\
&=\frac{1}{4N} \int    Q_2(A;X,Y_0) \Big( \Delta_{2,X} P_{N,2}(\tau;X,X_0) \Big) {\rm d}X     \\
&=\frac{1}{4N} \int    \Big( \Delta_{2,X}   Q_2(A;X,Y_0)  \Big)  P_{N,2}(\tau;X,X_0) {\rm d}X.
\end{aligned}
\end{equation}

Rewrite the $K_2 \times K_1$ complex  matrix  $Y_0$ as   $Y_0=[y_{j,k}]=[a_{j,k}+{\rm i}  b_{j,k}]$, and let

$$  \Delta_{Y_0}:=\sum_{j=1}^{K_2}  \sum_{k=1}^{K_1} \big (\partial_{a_{j,k}}^2+\partial_{  b_{j,k}}^2\big). $$
We claim that 
\begin{equation} \label{VIPid}
 \Delta_{2,X}   Q_2(A;X,Y_0) = \Delta_{Y_0} Q_2(A;X,Y_0).
\end{equation}
If so,   one obtains  a heat equation  from  \eqref{heatproducre2} that 
\begin{equation*}
{\partial \tau}\overline{Q_2}(\tau;A,Y_0) =\frac{1}{4N}\Delta_{Y_0}  \overline{Q_2}(\tau;A,Y_0).
 \end{equation*}
Together with  the  initial boundary  condition
\begin{equation*}\label{heatinitial2}
\overline{Q_2}(0;A,Y_0)=Q_2(A;X_0,Y_0),
\end{equation*}
the unique solution is thus  given by 
\begin{equation*}\label{final2}
\overline{Q_2}(\tau;A,Y_0)=\int Q_2(A;X_0,Y) \widehat{P}_{K,2}(\tau;Y,Y_0)  {\rm d}Y,
\end{equation*}
from which   the desired result   follows.

The remaining task is to verify  the identity \eqref{VIPid}. Denote
\begin{equation*}
\widetilde{T}_2=\begin{bmatrix}
A_1-\mathbb{I}_{K_1}\otimes X & -Y_0^*\otimes \mathbb{I}_N     \\
Y_0\otimes \mathbb{I}_N & A_2-\mathbb{I}_{K_2}\otimes X^*
\end{bmatrix},
\end{equation*}
by use of \eqref{diffrealviacomplex}, it suffices to prove
\begin{equation}\label{diffdualQ2}
\sum_{a,b=1}^N\frac{\partial^2}{\partial x_{ab}\partial \overline{x_{ab}}}
\det\left( \widetilde{T}_2 \right) = \sum_{\alpha=1}^{K_2}\sum_{\beta=1}^{K_1}
 \frac{\partial^2}{\partial y_{\alpha,\beta}\partial \overline{y_{\alpha,\beta}}}
\det\left( \widetilde{T}_2 \right).
\end{equation}
One uses \eqref{differentiationdeterminant} to obtain 
\begin{equation*}\label{beta2DeltaX}
\sum_{a,b=1}^N\frac{\partial^2 }{\partial x_{ab}\partial \overline{x_{ab}}} \det\left( \widetilde{T}_2 \right) 
=\sum_{a,b=1}^N\sum_{\alpha=1}^{K_2}\sum_{\beta=1}^{K_1}
\det\left( \widetilde{T}_2\left[
I_2;J_2
\right] \right),
\end{equation*}
where
\begin{equation*}\label{beta2ijdenotion}
I_2=\left\{ (\beta-1)N+a,(K_1+\alpha-1)N+b \right\}, \quad
J_2=\left\{ (\beta-1)N+b,(K_1+\alpha-1)N+a \right\}.
\end{equation*}
On the other hand, for $1\leq \alpha \leq K_2$ and $1\leq \beta \leq K_1$ simple calculation shows 
\begin{equation*}\label{beta2DeltaY1}
\frac{\partial }{\partial y_{\alpha,\beta}}  \det\left( \widetilde{T}_2 \right) =\sum_{b=1}^N
(-1)^{K_1N+(\alpha+\beta)N}\det\left(
\widetilde{T}_2\left[ I_{2,0};J_{2,0} \right]
\right),
\end{equation*}
where
\begin{equation*}
I_{2,0}=\left\{ (K_1+\alpha-1)N+b \right\},\quad J_{2,0}=\left\{ (\beta-1)N+b  \right\},
\end{equation*}
and 
\begin{equation*}\label{beta2DeltaY2}
\frac{\partial }{\partial
 \overline{y_{\alpha,\beta}}}  \det\left(\widetilde{T}_2
\left[ I_{2,0};J_{2,0} \right]
\right) =\sum_{a=1}^N
(-1)^{K_1N+(\alpha+\beta)N}\det\left( \widetilde{T}_2\left[
I_{2};J_2
\right] \right).
\end{equation*}
So it's easy to obtain \eqref{diffdualQ2}.  The proof is thus complete.

\section{Autocorrelations of characteristic polynomials I}  \label{autocorrealtionI}

  \subsection{Scaling limits}
  
   
 In this section we just consider   additive deformations of the Ginibre ensembles, that is,   the density  $\,P_{N,\beta}(\tau;X,X_0)$ given  in  Definition \ref{Ginb} but with $\Sigma=\Gamma= \mathbb{I}_N$.    For convenience,  we will use the same notation    in subsequent sections and  assume that $K_2=K_1$ unless otherwise specified.  For two diagonal matrices  $Z={\rm diag}\left( z_1,\cdots,z_{K_1} \right)$ and $W
={\rm diag}\left(w_1,\cdots, {w_{K_1}}\right)$,  at time $\tau=2/\beta$ let      
\begin{equation} \label{CharacDenotion}
F^{(\beta)}_{K_1}(Z,W)=\int \prod_{i=1}^{K_1}\big(\det\left(z_i-X\right) \det\left(\overline{w_i}-X^*\right) \big)\,P_{N,\beta}(2/\beta;X,X_0){\rm d}X
\end{equation}
be autocorrelation functions of characteristic polynomials.  In the present paper  we focus on complex and quaternion ensembles, and leave the real case in the forthcoming paper \cite{LZ}.  

 To state main results in this section, we need to introduce four matrix integrals. 
Let 
\begin{equation}\label{F11denotion}
\Omega_{11}=\Omega_{22}^{*}=\mathbb{I}_{K_1}\otimes
\Big[
\bigoplus_{i=1}^m\left(
 \phi
_i \mathbb{I}_{\beta_{i,1}}
\bigoplus 0_{\sum_{j=2}^{\alpha_i}\beta_{i,j}}
\right)
\Big],  
\end{equation}
 and 
\begin{equation}\label{F21F12denotion}
\Omega_{21}=Y\otimes   (P^*P)_{I_1}  
,\quad
\Omega_{12}=-Y^*\otimes  ((P^*P)^{-1})_{I_2},
\end{equation}
for nonnegative integers $\alpha, {\beta}_1, \ldots, {\beta}_m,  \hat{\beta}_1, \ldots, \hat{\beta}_m$,  define matrix integrals over the space of $K_1\times K_1$ complex matrices   
\begin{small}
\begin{multline}\label{I2GI}
I^{(2)}(\phi_1,\ldots,\phi_m;\hat{Z},\hat{W})=
 \int \det\begin{bmatrix}
\Omega_{11} & \Omega_{12} \\ \Omega_{21} & \Omega_{22}
\end{bmatrix}\times
  \\
\exp\Big\{-\frac{1}{2}{\rm Tr}(YY^*)^2 -{\rm Tr}\big(\hat{Z}Y^*Y\big)-{\rm Tr}\big(\hat{W}^*YY^*\big)\Big\} {\rm d}Y,
\end{multline}
\end{small}
 and
 \begin{multline} \label{1real}
  I^{(2,\mathrm{o})}(a,\alpha;\hat{Z},\hat{W},x_1,\cdots,x_m)= \int 
 \prod_{i=1}^{m} 
   {\det}^{{\beta}_{i}}\! 
   \begin{bmatrix}
  (x_{i}\mathbb{I}_{K_1}  -\hat{Z})^{\xi}
 & -Y  \\ Y^* &
   ( \overline{x_{i}}\mathbb{I}_{K_1}  - \overline{\hat{W}})^{\xi}
\end{bmatrix}
  \\ \times  
  {\det}^{\hat{\beta}_{i}}\! 
 \begin{bmatrix}
  x_{i}^{p_{i,1}}\mathbb{I}_{K_1}  
 & -Y  \\ Y^* &
   \overline{x_{i}}^{p_{i,1}}\mathbb{I}_{K_1} 
\end{bmatrix}  
  {\det}^{\alpha}\!\big(Y^* Y\big)
e^{-a{\rm Tr}(YY^*)} {\rm d}Y,
\end{multline}
where 
\begin{equation}\label{beta's1realout}
\beta_{i,>}=\displaystyle\sum_{j=2}^{\alpha_i}  \beta_{i,j},\quad
\beta_i+\hat{\beta}_i=\beta_{i,1},
\end{equation}
while $\beta_{i}=  \beta_{i,1}$ if $ p_{i,1}=\xi$, otherwise  $\beta_{i}=0$.

Also  set
\begin{equation}\label{4UV}
\begin{aligned}
&U={\rm i}P^t\mathbb{J}_r P:=\begin{bmatrix} U_1 & U_2 \\ -\overline{U}_2 & \overline{U}_1 \end{bmatrix},\quad
V=U^{-1}:=
\begin{bmatrix} V_1 & V_2 \\ -\overline{V}_2 & \overline{V}_1 \end{bmatrix},
\end{aligned}
\end{equation}
and  for a $2K_1\times 2K_1$ complex matrix Y and $t\times t$ complex matrices $A,B$, put
\begin{equation}\label{4RRdeno}
\begin{aligned}
R\left( Y;A,B \right)=\begin{bmatrix}
Y\otimes A & \\ & Y^*\otimes \overline{B}
\end{bmatrix},
\end{aligned}
\end{equation}

\begin{equation}\label{4RSdeno}
\begin{aligned}
S\left( Y \right)=\begin{bmatrix}
\Omega_{11}^{(\mathrm{re})} & Y\otimes \left(V_2\right)_{I_2} \\ -Y^*\otimes \overline{\left(U_2\right)_{I_1}} & -\overline{\Omega_{11}^{(\mathrm{re})}}
\end{bmatrix},
\Omega_{11}^{(\mathrm{re})}=\mathbb{I}_{2K_1}\otimes \Big( 
\bigoplus_{i=1}^m\big( \phi_i\mathbb{I}_{\beta_{i,1}}\bigoplus 0_{\sum_{j=2}^{\alpha_i}\beta_{i,j}} \big) \Big).
\end{aligned}
\end{equation}
The other two matrix integrals  are  
\begin{small}
\begin{multline}\label{I4GI}
I^{(4)}(\phi_1,\ldots,\phi_m;\hat{Z},\hat{W})=
 \int \det\begin{bmatrix}
\Omega_{11} & Y\otimes \left(V_2\right)_{I_2} \\ -Y^*\otimes \overline{\left(U_2\right)_{I_1}} & -\Omega_{11}^*
\end{bmatrix}
  \\
\times \exp\Big\{-{\rm Tr}(YY^*)^2 -{\rm Tr}\big(\hat{Z}YY^*\big)-{\rm Tr}\big(\hat{W}^*Y^*Y\big)\Big\} {\rm d}Y,
\end{multline}
\end{small}
over the space of $K_1\times K_1$ complex matrices, and   
\begin{small}
\begin{multline}\label{I4GIR}
I^{(4,\mathrm{re})}(\phi_1,\ldots,\phi_m;\hat{\Lambda})=
 \int {\rm Pf}\begin{bmatrix}
R\left( Y;\left(V_1\right)_{I_2},\left(U_1\right)_{I_1} \right) & S\left( Y \right)  
\\  
-S\left( Y \right)^t  &  R\left( Y^*;\left(U_1\right)_{I_1},\left(V_1\right)_{I_2} \right)
\end{bmatrix}\\ \times
\exp\Big\{-\frac{1}{2}{\rm Tr}(YY^*)^2 -{\rm Tr}\big(\hat{\Lambda}YY^*\big)\Big\} {\rm d}Y,
\end{multline}
\end{small} over the space of $2K_1\times 2K_1$ complex symmetric matrices.

Let $z_0$  be a spectral parameter and $\omega$ be a positive   exponent, which will be determined later.   Introduce scaled variables   as follows
\begin{small}
\begin{equation}\label{Symbol}
\Lambda=\begin{bmatrix}
Z & \\ & W^*
\end{bmatrix}:=Z_0+N^{-\omega}\hat{\Lambda},\quad 
Z_0=\begin{bmatrix}
z_0\mathbb{I}_{K_1} & \\ & \overline{z_0} \mathbb{I}_{K_1}
\end{bmatrix}, \quad
\hat{\Lambda}=\begin{bmatrix}
\hat{Z} & \\ & \hat{W}^*
\end{bmatrix}.
\end{equation}
\end{small}

Our main results in this section are as follows.
\begin{theorem}\label{2-complex} For the $ {\mathrm{GinUE}}_{N}(A_0)$ ensemble with the assumption  \eqref{J2detail}, 
given an integer $0\leq m\leq \gamma$,  assume that    $ \theta_{j}\neq z_0$ is fixed  for all $j=m+1, \ldots, \gamma$ and  
\begin{equation}\label{beta2edgethetaexpression}
\begin{aligned}
&\theta_i=z_0+N^{-\chi_i}\hat{\theta}_i,\quad
i=1,2,\cdots,m.            
\end{aligned}
\end{equation}
 With    $Z$ and $W$ given in  \eqref{Symbol}, 
as $N\to \infty$  the following  hold   uniformly for  
$\hat{\Lambda}$
in a compact subset of $\mathbb{C}^{2K_{1}}$.


(i)  When  $|z_0|=1$, let  $\omega=\frac{1}{2}$, $\chi_i=\frac{1}{4p_{i,1}}$, 
 then 
\begin{multline} \label{2edgecomplex}
F^{(2)}_{K_1}(Z,W)=C_N^{(2)} D^{(2)}
N^{-\frac{1}{2}K_{1}^2- \frac{1}{2}  \sum_{i=1}^{m } \sum_{j=1}^{\alpha_{i}}  K_{1}\beta_{i,j}  }  e^{\sqrt{N}{\rm Tr}(z_0^{-1}\hat{Z}+\overline{z_{0}^{-1}\hat{W}})}
 e^{ -\frac{1}{2}{\rm Tr}\big(z_{0}^{-2}\hat{Z}^2+\overline{z_{0}^{-2}\hat{W}^2}\big) }
\\  \times 
\left(  I^{(2)}\left(  
( -\hat{\theta}_1)^{p_{1,1}},\cdots,( -\hat{\theta}_m )^{p_{m,1}};
z_0^{-1}\hat{Z},z_0^{-1}\hat{W}
\right)+O\big( N^{-\frac{1}{4\kappa}} \big) \right),
\end{multline}
%
where  $ \kappa=\max\left\{1/2,p_{1,1},\cdots,p_{m,1} \right\}$, $C_N^{(2)}$ 
and $D^{(2)}$ are  given in  \eqref{CNbeta} and \eqref{D2-2edge}  below respectively. 


(ii) When  $|z_0|>1$ and $m\geq 1$, let  $\omega=\frac{1}{2\xi}$ and $\chi_i=\frac{1}{2p_{i,1}}$ for $1\leq i\leq m$, where $\xi=\min\left\{ p_{1,1}, \ldots, p_{m,1} \right\}$, if $P$ in   \eqref{J2detail} is a unitary  matrix,  then 
\begin{multline} \label{2outcomplex}
F^{(2)}_{K_1}(Z,W)=C_N^{(2)} D^{(2,\mathrm{o})}
N^{-K_{1}^2- K_1 \sum_{i=1}^{m} (\beta_{i,>}+\beta_i+\hat{\beta}_i)} e^{ h_{2,0}(\Lambda)}
\\  \times 
\Big(  I^{(2,\mathrm{o})}\big(1-|z_0|^{-2},\sum_{i=1}^m \beta_{i,>};
\hat{W}^*,\hat{Z}^*,\overline{{\hat{\theta}}_{1}} ,\ldots, \overline{{\hat{\theta}}_{m}}\big)+O\big( N^{-\frac{1}{2\xi}} \big) \Big),
\end{multline}
%
where   $C_N^{(2)}$,   $D^{(2,\mathrm{o})}$ and $h_{2,0}(\Lambda)$,   are  given in  \eqref{CNbeta},  \eqref{D2-2outlier} and \eqref{h2Lambdabeta2outlier} below respectively, and   $\beta_{i,>},  \beta_{i}, \hat{\beta}_i$ are given in \eqref{beta's1realout}.

\end{theorem}



\begin{theorem}\label{4-complex} For the $ {\mathrm{GinSE}}_{N}(A_0)$ ensemble with the assumption  \eqref{J4detail}, 
 given an integer $0\leq m\leq \gamma$, assume that $ \theta_{j}\neq z_0$ is fixed  for all $j=m+1, \ldots, \gamma$ and  
\begin{equation}\label{beta4edgethetaexpression}
\begin{aligned}
&\theta_i=z_0+N^{-\frac{1}{4p_{i,1}}}\hat{\theta}_i,\quad
i=1,2,\cdots,m.            
\end{aligned}
\end{equation}
 With    $Z$ and $W$ given in  \eqref{Symbol} 
 where  $\omega=\frac{1}{2}$,   let   $ \kappa=\max\left\{1/2,p_{1,1},\cdots,p_{m,1} \right\}$ and $t=\sum_{i=1}^m\sum_{j=1}^{\alpha_i}\beta_{i,j}$, 
as $N\to \infty$  the following  hold   uniformly for  
$\hat{\Lambda}$
in a compact subset of $\mathbb{C}^{2K_1}$.

(i)  When   $ z_0=\pm 1$, 
 then 
\begin{multline} \label{4edgeR}
F^{(4)}_{K_1}(Z,W)=\frac{1}{\widehat{Z}_{2K_1,1}} D^{(4,\mathrm{re})}
N^{-\frac{1}{2}K_1(2K_1+1)-K_1t }  e^{2z_0^{-1}\sqrt{N}{\rm Tr}(\hat{Z}+\overline{\hat{W}})}
 e^{ -{\rm Tr}\big(\hat{Z}^2+\overline{\hat{W}^2}\big) }
\\  \times 
\left(  I^{(4,\mathrm{re})}\left(  
( -\hat{\theta}_1)^{p_{1,1}},\cdots,( -\hat{\theta}_m )^{p_{m,1}};
2z_0^{-1}\hat{\Lambda}
\right)+O\big( N^{-\frac{1}{4\kappa}} \big) \right),
\end{multline}
where $\widehat{Z}_{2K_1,1}$ 
and $D^{(4,\mathrm{re})}$ are  given in  \eqref{Zdual} and \eqref{D-4Redge} respectively. 

(ii)  When   $ \Im z_0 >0$ and   $|z_0|=1$,   
 then 
\begin{multline} \label{4edgecomplex}
F^{(4)}_{K_1}(Z,W)=C_N^{(4)} D^{(4)}
N^{-(K_1+1)K_1-\frac{1}{2}(K_1^2+K_1t) }  e^{2\sqrt{N}{\rm Tr}(z_0^{-1}\hat{Z}+\overline{z_{0}^{-1}\hat{W}})}
\\ \times
2^{-(K_1-1)K_1}\left( \frac{\pi}{\left| z_0-\overline{z}_0 \right|} \right)^{(K_1+1)K_1}
 e^{ -{\rm Tr}\big(z_{0}^{-2}\hat{Z}^2+\overline{z_{0}^{-2}\hat{W}^2}\big) }
\\  \times 
\left(  I^{(4)}\left(  
( -\hat{\theta}_1)^{p_{1,1}},\cdots,( -\hat{\theta}_m )^{p_{m,1}};
2z_0^{-1}\hat{Z},2z_0^{-1}\hat{W}
\right)+O\big( N^{-\frac{1}{4\kappa}} \big) \right),
\end{multline}
where $C_N^{(4)}$ 
and $D^{(4)}$ are  given in  \eqref{CNbeta} and \eqref{D-4edge}  below respectively.

\end{theorem}

We remark that in the deformed complex Ginibre case,  except for scaling limits in  the subcritical and critical regimes we also consider the supercritical regime in a special case where  $P$ in   \eqref{J2detail} is assumed to be  a unitary  matrix.  For general $P$,  calculations become much more complicated either  in the deformed complex or  quaternion  Ginibre cases. These will be left in the forthcoming paper \cite{LZ}.

 \subsection{General procedure}

In this subsection   we are devoted to obtaining partial  asymptotic results  which serve  as preparation for the complete proofs of limiting results above; see   Proposition  \ref{Analysisform} below.

By  Corollary \ref{characteristic}, autocorrelation functions defined in   \eqref{CharacDenotion} 
can be rewritten as

\begin{equation}\label{fKLbetaduality}
F^{(\beta)}_{K_1}(Z,W)=\int Q_{\beta}(A;X_0,Y) \hat{P}_{K,4/\beta}(2/\beta;Y,0){\rm d}Y,
\end{equation}
where $ A=\Lambda\otimes \mathbb{I}_N $ for $\beta=1,2$ and  $ A=\Lambda \otimes \mathbb{I}_{2N} $ for $\beta=4$.

To simplify the  right hand side of \eqref{fKLbetaduality}, 
we  need the  following well-known property for tensor product (see e.g. \cite[eqn(6)]{HS81} ) to evaluate $Q_{\beta}$.
\begin{proposition}\label{tensorproperty}
For  a $p\times q$  matrix  A and an $m\times n$ matrix   B,  there exist  permutation matrices $\mathbb{I}_{p,m}$,  which only depend on  $p, m$ and satisfy  $\mathbb{I}_{p,m}^{-1}=\mathbb{I}_{m,p}=\mathbb{I}_{p,m}^t$, such that 

\begin{equation}\label{tensorpropertyequation}
\mathbb{I}_{p,m}\left( A\otimes B \right)\mathbb{I}_{q,n}^{-1}=B\otimes A.
\end{equation}
  \end{proposition}

Next,  we interchange the order of two matrices in the tensor product and further perform some elementary matrix operations  on  $Q_{\beta}$.
For $\beta=2$,  
we have 
\begin{small}
\begin{equation}\label{Q2trasform0}
Q_2=\det(B_{11}^{(2)})\left[ \det(A_2) \right]^{N-r},\quad A_2:=\begin{bmatrix}
Z & -Y^* \\ Y & W^*
\end{bmatrix},
\end{equation}
\end{small}
where
\begin{small}
\begin{equation}\label{B112denotion}
B_{11}^{(2)}=\begin{bmatrix}
\mathbb{I}_r \otimes Z-A_0\otimes \mathbb{I}_{K_1}   &  -\mathbb{I}_r \otimes Y^*  \\
\mathbb{I}_r\otimes Y  &  \mathbb{I}_r \otimes W^*-A_{0}^*\otimes \mathbb{I}_{K_1}
\end{bmatrix}.
\end{equation}
\end{small}

For $\beta=4$, in view of \eqref{A04deno}, we introduce an  extra  error term  $O( \Xi ) $ for  non-negative $ \Xi$,  which will be  specified  somewhere, and 
  set 
$$ A_0=P\left( \begin{bmatrix}
J_2 & \\ & \overline{J}_2
\end{bmatrix}+O( \Xi ) \right)P^{-1}.$$
 Recall the definition in \eqref{4UV}, we see  from elementary matrix operations  that 
\begin{equation}\label{Q4transform1}
Q_4={\rm Pf}( \widetilde{B} )\left( \det(A_4) \right)^{N-r},\quad A_4:=\begin{bmatrix}
{\rm i}Y & \Lambda \\ -\Lambda & -{\rm i}Y^*
\end{bmatrix},
\end{equation}
where
\begin{small}
$$ \widetilde{B}=\begin{bmatrix}
V\otimes Y  &   \mathbb{I}_{2r}\otimes \Lambda-\bigg( \begin{bmatrix}
J_2 & \\ & \overline{J}_2
\end{bmatrix}+O\big( \Xi \big)\bigg)\otimes \mathbb{I}_{2K_1}  \\
-\bigg( \mathbb{I}_{2r}\otimes \Lambda-\bigg( \begin{bmatrix}
J_2 & \\ & \overline{J}_2
\end{bmatrix}+O\big( \Xi \big)\bigg)\otimes \mathbb{I}_{2K_1} \bigg)^t  &  U\otimes Y^*
\end{bmatrix}. $$
\end{small}
Reversing  the order of  tensor products and again perform elementary matrix operations, we get  from \eqref{tensorpropertyequation} that
\begin{equation}\label{Q4transform2}
Q_4=\left( \det(A_4) \right)^{N-r} {\rm Pf}\, \begin{bmatrix}
R & S \\ -S^t & T
\end{bmatrix}, 
\end{equation}
where
\begin{equation}\label{RSTdeno}
\begin{aligned}
&  R=\begin{bmatrix} Y\otimes V_1 & O\left( \Xi \right)
\\ O\left(  \Xi \right) &  Y^*\otimes \overline{U}_1  \end{bmatrix},\quad
T=\begin{bmatrix} Y^*\otimes U_1 & O\left( \Xi \right)
\\ O\left( \Xi \right) &  Y\otimes \overline{V}_1  \end{bmatrix},             \\
&S=\begin{bmatrix}   \Lambda\otimes \mathbb{I}_r-\mathbb{I}_{2K_1}\otimes J_2
+O\left( \Xi  \right)    & Y\otimes V_2
\\      -Y^*\otimes \overline{U}_2  &  -\Lambda\otimes \mathbb{I}_r+\mathbb{I}_{2K_1}\otimes J_2^*+
O\left( \Xi \right) \end{bmatrix}.
\end{aligned}
\end{equation}

In summary,  we have \begin{proposition}\label{propbetterform}
With $A_{\beta}$  given in \eqref{Q2trasform0} and  \eqref{Q4transform1},  and with  
$B_{11}^{(2)}$  in \eqref{B112denotion} and $\begin{bmatrix}
R & S \\ -S^t & T
\end{bmatrix}$ in
 \eqref{RSTdeno}, 
let
\begin{equation}\label{f0KLbetaexpression}
f_{\beta}^{(0)}(Y)={\rm Tr}(YY^*)-\log\det(A_{\beta}),\quad \beta=2,4,
\end{equation}
and 
\begin{small}
\begin{equation}\label{gbetaYLambda}
g_{\beta}(Y,\Lambda)=\begin{cases}
\det(B_{11}^{(2)})\left( \det(A_{2}) \right)^{-r}, & \beta=2; \\
{\rm Pf}\begin{bmatrix}
R & S \\ -S^t & T
\end{bmatrix}\left( \det(A_{4}) \right)^{-r}, & \beta=4.
\end{cases}
\end{equation}
\end{small}
Then 
for $\beta=2,4$,
\begin{equation}\label{beta24Restrictionequation}
F_{K_1}^{(\beta)}(Z,W)=
 \frac{1}{\widehat{Z}_{2K_1,4/\beta} }
 \int    g_{\beta}(Y,\Lambda) \, e^{ -N f_{\beta}^{(0)}(Y)  } {\rm d}Y,
\end{equation}
where $Y$ is a  complex matrix of size $K_1$ when  $\beta=2$,  while  Y is a complex symmetric matrix of size ${2K_1}$ when $\beta=4$. 
\end{proposition}

Whenever  $|z_0|\geq 1$,  we will use Laplace's  method to  prove that the overwhelming contribution to the integral    comes   from 
a small  neighborhood of the zero matrix. 
\begin{proposition}\label{Analysisform}
If  $|z_0|\geq 1$, then for any $\delta>0$ there exists $\Delta^{(\beta)}>0$ such that
\begin{equation}\label{analysisform}
F_{K_1}^{(\beta)}(Z,W)=C_N^{(\beta)}\left(I_{\delta,N}^{(\beta)}+O\big( e^{-\frac{1}{2}N\Delta^{(\beta)}} \big) \right),
\end{equation}
with
\begin{equation}\label{Ideltabeta}
I_{\delta,N}^{(\beta)}=\int\limits_{{\rm Tr}(YY^*)\leq \delta}\ g_{\beta}(Y,\Lambda)
\exp\!\left\{ -Nf_{\beta}(Y,Z_0)+Nh_{\beta}(Y) \right\} {\rm d}Y,
\end{equation}
and
\begin{equation}\label{CNbeta}
C_N^{(\beta)}= 
\frac{1}{\widehat{Z}_{2K_1,4/\beta} }\left( \det Z_0 \right)^{\frac{\beta N}{2}}.
\end{equation}
Here 
\begin{small}
\begin{equation}\label{fbetaYZ0}
f_{\beta}(Y,Z_0)=\begin{cases}
{\rm Tr}(YY^*)-\log\det\left( \mathbb{I}_{K_1}+\left| z_0 \right|^{-2}YY^* \right), & \beta=2; \\ 
{\rm Tr}(YY^*)-\log\det\left( \mathbb{I}_{2K_1}+YZ_0^{-1}Y^*Z_0^{-1} \right), & \beta=4,
\end{cases}
\end{equation}
\end{small}
and 
\begin{equation}\label{hbetaYLambda}
h_{\beta}(Y)= 
\log\det\left( \mathbb{I}_{\beta K_1}+N^{-\omega}H_{\beta} \right),
\end{equation}
with 
\begin{small}
\begin{equation}\label{Cbeta}
H_{\beta}:=H_{\beta}(Y)=\begin{cases} 
\begin{bmatrix}
z_0\mathbb{I}_{K_1}  & -Y^*  \\Y & \overline{z_0}\mathbb{I}_{K_1}
\end{bmatrix}^{-1}
\begin{bmatrix}
\hat{Z} & \\ & \hat{W}^*
\end{bmatrix}, & \beta=2; \\  
\begin{bmatrix}
{\rm i}Y  & Z_0  \\-Z_0 & -{\rm i}Y^*
\end{bmatrix}^{-1}
\begin{bmatrix}
0 & \hat{\Lambda}  \\ -\hat{\Lambda} & 0
\end{bmatrix}, & \beta=4.
\end{cases}
\end{equation}
\end{small}
\end{proposition}
\begin{proof}
We just give  a detailed proof   in the case of  $\beta=2$ since  similar argument applies to  $\beta=4$.

Noting  that 
\begin{equation}\label{A1A1}
A_2A_2^*=\begin{bmatrix}
Y^*Y+ZZ^* & ZY^*-Y^*W \\ YZ^*-W^*Y & YY^*+W^*W
\end{bmatrix},
\end{equation}
we use 
the Hadamard-Fischer inequality (see e.g. \cite[Theorem 7.8.5]{HJ})  to obtain  
\begin{equation*}\label{Fischer1}
\det(A_2A_2^*)\leq \det\left( YY^*+W^*W \right)\det\left( Y^*Y+ZZ^* \right).
\end{equation*} 
Recall  $f_{\beta}^{(0)}(Y)$ defined in  \eqref{f0KLbetaexpression}, 
we   arrive at 
\begin{small}
\begin{equation*}\label{Fischer1}
\begin{aligned}
&
\Re \left\{ f_2^{(0)}(Y)-f_2^{(0)}(0) \right\}\geq  
{\rm Tr}YY^*-\frac{1}{2}
\log\det\left( \mathbb{I}_{K_1}+Y^*Y(ZZ^*)^{-1} \right)     \\
&-\frac{1}{2}
\log\det\left( \mathbb{I}_{K_1}+YY^*(W^*W)^{-1} \right)
={\rm Tr}(YY^*)-\log\det\left( \mathbb{I}_{K_1}+\left| z_0 \right|^{-2}YY^*\right) \\
& -\frac{1}{2}
\log\det\left( \mathbb{I}_{K_1}+S_{1}S_{2}\right)-\frac{1}{2}
\log\det\left( \mathbb{I}_{K_1}+S_{3}S_{4}\right)\\
\end{aligned}
\end{equation*}
\end{small}
where 
$$
\begin{aligned}
&S_{1}=\left(\mathbb{I}_{K_1}+\left | z_0 \right |^{-2}YY^* \right)^{-1} YY^*, \quad 
S_2=(W^*W)^{-1}-\left| z_0 \right|^{-2} \mathbb{I}_{K_1}.     \\
&S_{3}=\left(\mathbb{I}_{K_1}+\left | z_0 \right |^{-2}Y^*Y \right)^{-1} Y^*Y, \quad 
S_4=(ZZ^*)^{-1}-\left| z_0 \right|^{-2} \mathbb{I}_{K_1}. 
\end{aligned}
$$
It is easy to see that  $S_1$ is  a positive definite  matrix with all eigenvalues less than one and   $S_2=O(N^{-\omega})$,   so for sufficiently large $N$
\begin{equation*}
\log\det\left( \mathbb{I}_{K_1}+S_{1}S_{2}\right) =O\big({\rm Tr}(S_{1}S_{2})\big)=O(N^{-\omega}).
\end{equation*}
Similarly,
\begin{equation*}
\log\det\left( \mathbb{I}_{K_1}+S_{3}S_{4}\right) =O\big({\rm Tr}(S_{3}S_{4})\big)=O(N^{-\omega}),
\end{equation*}
and we  further have 
\begin{equation*}
\begin{aligned}
&
\Re \left\{ f_2^{(0)}(Y)-f_2^{(0)}(0) \right\}\geq  
{\rm Tr}(YY^*)-\log\det\left( \mathbb{I}_{K_1}+\left| z_0 \right|^{-2}YY^*\right) +O(N^{-\omega}).
\end{aligned}
\end{equation*}

%

By the simple fact that $x-\log(1+x\left| z_0 \right|^{-2})\geq 0$ with equality if and only if $x=0$, whenever $\left| z_0 \right|\geq 1$,  use of  the singular value decomposition for $Y$   shows us  that  there exists $\Delta^{(2)}>0$  such that for sufficiently large $N$ and for  $  {\rm Tr}(YY^*)\geq \delta$, 
\begin{small}
\begin{equation}\label{f10Ylowerbound}
\Re \left\{f_2^{(0)}(Y)-f_2^{(0)}(0) \right\}\geq 2\Delta^{(2)}.
\end{equation}
\end{small}

%

Observe   that  the factor  $ \det(B_{11}^{(2)})$ in  $g_{2}(Y,\Lambda)$ is a polynomial  of matrix entries and  the other 
$\left( \det(A_2) \right)^{-r}$ can be absorbed, 
    therefore,  we can take a large   $N_0>r$ such that for some constant independent of  $\delta$ and $N$
\begin{small}
\begin{equation*}
\int\limits_{{\rm Tr}YY^*\geq \delta}\ 
\left| g_{2}(Y,\Lambda)  \exp\left\{ -N_0\left( f_2^{(0)}(Y)-f_2^{(0)}(0) \right) \right\} \right| {\rm d}Y \leq C.
\end{equation*}
\end{small} 
Together with \eqref{f10Ylowerbound}   we obtain  
\begin{small}
\begin{equation*}
\begin{aligned}
\int\limits_{{\rm Tr}(YY^*)\geq \delta}\ 
\left| g_{2}(Y,\Lambda)  \exp\left\{ -N\left( f_2^{(0)}(Y)-f_2^{(0)}(0) \right) \right\} \right| {\rm d}Y \leq   C e^{-(N-N_0)\Delta^{(2)} }=O(e^{-N\Delta^{(2)}}).
\end{aligned}
\end{equation*}
\end{small}

Finally,  it is easy to verify    
\begin{equation*}
 f_2^{(0)}(Y)= f_2(Y,Z_0)-h_{2}(Y)-\log\det(Z_0). 
\end{equation*}
On the other hand,  one can derive from $\Lambda=Z_0+N^{-\omega}\hat{\Lambda}$ that 
 $$\exp\left\{- N\big(-\log\det(Z_0)+f_2^{(0)}(0)\big)\right\}=\exp\left\{O(N^{1-\omega})\right\}.$$
The desired result immediately follows from the simple observation 
\begin{equation*}\label{experrorcontrol}
O\left(e^{-N\Delta^{(2)}+O(N^{1-\omega})}\right)=O\left(e^{-\frac{1}{2}N\Delta^{(2)}}\right).
\end{equation*}
\end{proof}

%


The following bound   on $H_{\beta}$   will be used in the subsequent sections.
\begin{lemma}\label{Cbetacrudebound}
For a sufficiently small $\delta>0$,  with $H_{\beta}$ defined in \eqref{Cbeta} where  $\beta=2,4$,  for any  given nonzero  complex number $z_0$ we have 
\begin{small}
\begin{equation}\label{Cbetacrudequation}
H_{\beta}=O(1),
\end{equation}
\end{small}
whenever  $ {\rm Tr}(YY^*)\leq \delta$.
\end{lemma}
\begin{proof}
Simple manipulation gives rise to 
\begin{small}
\begin{equation}\label{Cbetadetailed}
H_{\beta}=\begin{cases}
\begin{bmatrix}
\left( z_0\mathbb{I}_{K_1}+\overline{z_0}^{-1}Y^*Y \right)^{-1}\hat{Z}  & 
z_0^{-1}Y^*\left( \overline{z_0}\mathbb{I}_{K_1}+z_0^{-1}YY^* \right)^{-1}\hat{W}^*  
\\
-\overline{z_0}^{-1}Y\left( z_0\mathbb{I}_{K_1}+\overline{z_0}^{-1}Y^*Y \right)^{-1}\hat{Z} & 
\left( \overline{z_0}\mathbb{I}_{K_1}+z_0^{-1}YY^* \right)^{-1}\hat{W}^*
\end{bmatrix}, & \beta=2; \\  
\begin{bmatrix}
\left( Z_0+Y^*Z_0^{-1}Y \right)^{-1}\hat{\Lambda}  & -{\rm i}Z_0^{-1}Y^*\left( Z_0+Y Z_0^{-1}Y^* \right)^{-1}\hat{\Lambda}
\\
-{\rm i}Z_0^{-1}Y\left( Z_0+Y^*Z_0^{-1}Y \right)^{-1}\hat{\Lambda} & \left( Z_0+Y Z_0^{-1}Y^* \right)^{-1}\hat{\Lambda}
\end{bmatrix}, & \beta=4.
\end{cases}
\end{equation}
\end{small}

We just focus on  the case of $\beta=2$ since $\beta=4$ is similar.
With \eqref{Cbetadetailed} in mind,  we see  from  ${\rm Tr}(YY)^*\leq \delta$ that $Y^*Y=O(\delta).$  
For a sufficiently small $\delta>0$,
\begin{equation*}
\left( z_0\mathbb{I}_{K_1}+\overline{z_0}^{-1}Y^*Y \right)^{-1}
=z_0^{-1}\mathbb{I}_{K_1}+O(\delta),
\end{equation*}
so 
\begin{small}
\begin{equation*}
\left( z_0\mathbb{I}_{K_1}+\overline{z_0}^{-1}Y^*Y \right)^{-1}\hat{Z} =O(1),\ 
-\overline{z_0}^{-1}Y\left( z_0\mathbb{I}_{K_1}+\overline{z_0}^{-1}Y^*Y \right)^{-1}\hat{Z}=O(1).
\end{equation*}
\end{small}
Similarly, we can obtain 
\begin{small}
\begin{equation*}
\left( \overline{z_0}\mathbb{I}_{K_1}+z_0^{-1}YY^* \right)^{-1}\hat{W}^*=O(1),\ 
z_0^{-1}Y^*\left( \overline{z_0}\mathbb{I}_{K_1}+z_0^{-1}YY^* \right)^{-1}\hat{W}^* =O(1).
\end{equation*}
\end{small}
This thus completes the proof.
\end{proof}

 
The remaining task is  to    find    exact approximations by  doing 
Taylor expansions for matrix-variable functions   
$f_{\beta}(Y,Z_0)$, $h_{\beta}(Y,\Lambda)$ and $g_{\beta}(Y,\Lambda)$, given in Proposition \ref{Analysisform}, 
 in the region $\left\{Y:  {\rm Tr}(YY^*)\leq \delta \right\}$.  It will be left in  subsequent Section \ref{autocorrealtionII}.



 \section{Autocorrelations of characteristic polynomials II} \label{autocorrealtionII}

\subsection{Deformed GinUE ensemble}
In order to provide the proof of  Theorem \ref{2-complex}  we need  the following two key lemmas.


\begin{lemma}\label{determinantcalculationbeta2edge}
Let $Y$ be a complex $K_1\times K_1$ matrix,  set  \begin{small}
\begin{equation}
B_2=\begin{bmatrix}
\mathbb{I}_r \otimes Z-A_0\otimes \mathbb{I}_{K_1}   &  -\mathbb{I}_r \otimes Y^*  \\
\mathbb{I}_r\otimes Y  &  \mathbb{I}_r \otimes W^*-A_{0}^*\otimes \mathbb{I}_{K_1}
\end{bmatrix},
\end{equation}
\end{small}
where   $A_0$ is given by \eqref{J2detail} and  $Z,W$ are defined in \eqref{Symbol} with  $|z_0|=1$ and  $\omega=\frac{1}{2}$.
Given an integer $0\leq m\leq \gamma$,  assume that    $ \theta_{j}\neq z_0$ is fixed  for all $j=m+1, \ldots, \gamma$ and  
\begin{equation}\label{beta2edgethetaexpression}
\begin{aligned}
&\theta_i=z_0+N^{-\frac{1}{4p_{i,1}}}\hat{\theta}_i,\quad
i=1,2,\cdots,m.            
\end{aligned}
\end{equation}
Then 
\begin{small}
\begin{equation}\label{determinantcalcuequa}
\begin{aligned}
\det\left( B_{2} \right)&= \prod_{i=m+1}^\gamma   |z_{0}-\theta_{i}|^{2K_{1} \sum_{j=1}^{\alpha_{i}} 
p_{i,j}\beta_{i,j} }\\
&\times\bigg\{   \det\begin{bmatrix}
N^{-\frac{1}{4}}\Omega_{11} & \Omega_{12} \\ \Omega_{21} & N^{-\frac{1}{4}}\Omega_{22}
\end{bmatrix}+                                                       
O\Big(
\sum_{(\alpha_1,\alpha_2)}N^{-\alpha_1}\| Y \|^{\alpha_2}
\Big)\bigg\},
\end{aligned}
\end{equation}
\end{small}
where the finitely many integer pairs $(\alpha_1, \alpha_2)$ satisfy the restriction  
\begin{equation}\label{errorsumcondition}
\alpha_1+\frac{\alpha_2}{4}\geq
\frac{K_1}{2}\sum_{i=1}^m\sum_{j=1}^{\alpha_i}\beta_{ij}+\frac{1}{4\kappa},
\end{equation}
where  $\Omega_{ij}$ are defined in  \eqref{F11denotion} and  \eqref{F21F12denotion} with
$\phi_k =( -\hat{\theta}_k)^{p_{k,1}}$ for all $1\leq k\leq m$. 
\end{lemma}

\begin{proof}

Since   $A_0=PJ_2P^{-1}$,    we   have   matrix decomposition
\begin{small}
\begin{equation*} 
\begin{aligned} B_2=
&\begin{bmatrix}
P\otimes \mathbb{I}_{K_1}   &  \\
&  \left( P^{-1} \right)^*\otimes \mathbb{I}_{K_1}
\end{bmatrix}
 \begin{bmatrix}
\mathbb{I}_r \otimes Z-J_2\otimes \mathbb{I}_{K_1}   &  -\left( P^*P \right)^{-1} \otimes Y^*  \\
(P^*P)\otimes Y  &  \mathbb{I}_r \otimes W^*-J_{2}^*\otimes \mathbb{I}_{K_1}
\end{bmatrix}  \\
&\times  \begin{bmatrix}
P^{-1}\otimes \mathbb{I}_{K_1}   &  \\
&  P^*\otimes \mathbb{I}_{K_1}
\end{bmatrix},          
\end{aligned}
\end{equation*}
from which, together with Proposition \ref{tensorproperty}  we further obtain 
\begin{equation*}\label{B112trans}
\det\left( B_{2} \right)=
\det\begin{bmatrix}
Z \otimes \mathbb{I}_r-\mathbb{I}_{K_1}\otimes J_2   &  - Y^*\otimes (P^*P)^{-1}  \\
Y\otimes (P^*P)  &   W^*\otimes \mathbb{I}_r-\mathbb{I}_{K_1}\otimes J_2^*
\end{bmatrix}.
\end{equation*}
\end{small}

Next, we will make use  of column replacement operations to calculate the above  determinant
\begin{small}
\begin{equation}\label{detcolumnreplace}
\begin{aligned}
&\det\left( H^{(0)}+H^{(1)}+H^{(2)} \right)=
\sum_{i_1, \ldots, i_q=0}^{2}\det\big[ h^{(i_1)}_{1}\big| h^{(i_2)}_{2}\big|\cdots \big| h^{(i_q)}_{q}  \big],  
\end{aligned}
\end{equation}
\end{small}
where  $H^{(i)}=\big[ h^{(i)}_{1}\big| h^{(i)}_{2}\big|\cdots \big| h^{(i)}_{q}  \big]$ is a  $q\times q$ matrix with columns $h_{j}^{(i)}$, $i=0,1,2$.

Recalling  the definition of $J_2$ in \eqref{J2detail} and  the assumptions in \eqref{beta2edgethetaexpression},  rewrite $J_2$ as  a sum 
\begin{equation}\label{J2decom}
J_2=J_2^{(0)}+J_2^{(1)},
\end{equation}
where
\begin{small}
\begin{equation}\label{J2mdeno}
J_2^{(0)}=\Big(\bigoplus_{i=1}^m  \bigoplus_{j=1}^{\alpha_{i}}  
\overbrace{ R_{p_{i,j}}\left(z_0 \right)   \bigoplus \cdots \bigoplus  R_{p_{i,j}}\left(z_0 \right)
}^{\beta_{i,j}\ blocks}\Big)
\bigoplus J_2^{(2)}
\end{equation}
\end{small}
with 
\begin{small}
\begin{equation}\label{J2rdeno}
J_2^{(2)}=\bigoplus_{i=m+1}^\gamma  \bigoplus_{j=1}^{\alpha_{i}} \overbrace{ R_{p_{i,j}}\left(\theta_i \right)   \bigoplus \cdots \bigoplus  R_{p_{i,j}}\left(\theta_i \right)
}^{\beta_{i,j}\ blocks},
\end{equation}
and 
\end{small}
\begin{equation}\label{J2adeno}
J_2^{(1)}=\Big(\bigoplus_{i=1}^m
N^{-\frac{1}{4p_{i,1}}}\hat{\theta}_i \mathbb{I}_{\sum_{j=1}^{\alpha_i}\beta_{i,j}p_{i,j}}\Big)\bigoplus 0_{r_0},
\end{equation}
 with  $r_0=r-\sum_{i=1}^m\sum_{j=1}^{\alpha_i}\beta_{i,j}p_{i,j}$. Furthermore,  by choosing
\begin{equation*}\label{Mdenotion}
H^{(0)}=\begin{bmatrix}
 \mathbb{I}_{K_1}\otimes \big(z_0\mathbb{I}_r - J_2^{(0)}\big)    &    \\
  & 
   \mathbb{I}_{K_1}\otimes\big ( \overline{z_0}\mathbb{I}_r - {J_{2}^{(0)}}^*\big)  
\end{bmatrix},
\end{equation*}
\begin{equation*}\label{H2denotion}
H^{(1)}=\begin{bmatrix}
-\mathbb{I}_{K_1}\otimes J_2^{(1)} +N^{-\frac{1}{2}}\hat{Z}\otimes \mathbb{I}_r    &    \\
  &  -\mathbb{I}_{K_1}\otimes {J_2^{(1)}}^* +N^{-\frac{1}{2}}\hat{W}^*\otimes \mathbb{I}_r
\end{bmatrix},
\end{equation*}
and \begin{equation*}\label{H1denotion}
H^{(2)}=\begin{bmatrix}
 &  -Y^* \otimes  V \\
Y\otimes U & 
\end{bmatrix} ,
\end{equation*}
 we obtain  
\begin{equation*}\label{detB112MH}
\det\left( B_{2} \right)=
\det\left( H^{(0)}+H^{(1)}+H^{(2)} \right).
\end{equation*}

Now note that
\begin{small}
\begin{equation*}\label{upperleftdecom}
z_0 \mathbb{I}_{r}-J_2^{(0)}=\Big(-\bigoplus_{i=1}^m  \bigoplus_{j=1}^{\alpha_{i}}   \overbrace{ R_{p_{i,j}}(0)   \bigoplus \cdots \bigoplus  R_{p_{i,j}}(0)
}^{\beta_{i,j}\ blocks}\Big)
\bigoplus \left(
z_0 \mathbb{I}_{r_0}-J_2^{(2)}
\right)
\end{equation*}
\end{small}
and 
\begin{small}
\begin{equation*}\label{lowerbottomdecom}
\overline{z_0} \mathbb{I}_{r}-{J_2^{(0)}}^*=
\Big(-\bigoplus_{i=1}^m  \bigoplus_{j=1}^{\alpha_{i}}   \overbrace{ R_{p_{i,j}}(0)^*   \bigoplus \cdots \bigoplus  R_{p_{i,j}}(0)^*
}^{\beta_{i,j}\ blocks}\Big)
\bigoplus \left(
\overline{z_0} \mathbb{I}_{r_0}-{J_{2}^{(2)}}^*
\right)
\end{equation*}
\end{small}
have some zero columns  with indexes respectively  corresponding to  
  $I_{1}$ and $I_{2}$, which are    the  index sets consisting of  numbers corresponding to   first and last  columns of all 
 blocks  
 $R_{p_{i,j}}\left(\theta_i \right)$    from $J_2$   where  $ 1\leq i \leq m, 1\leq j \leq \alpha_i$.
Therefore,  some of the determinants  on  the right-hand side of \eqref{detcolumnreplace} are obviously zero and we need to find the leading non-zero terms.  For this,   we can take $H^{(0)}$ as a leading matrix and 
 replace zero columns with corresponding non-zero columns from $H^{(1)}$  or $H^{(2)}$ by following the following rules: 
\begin{itemize}
\item[(i)]
Do not replace any column of $z_0\mathbb{I}_{r_0}-J_2^{(2)}$ and $\overline{z_0}\mathbb{I}_{r_0}-{J_2^{(2)}}^*$;
\item[(ii)]
For  each  of the  blocks $-R_{p_{i,j}}(0)$ and   $-R_{p_{i,j}}(0)^*$ where   $ 1\leq i \leq m$ and $1\leq j \leq \alpha_i$,

\begin{itemize}
\item[(ii-1)]

when  $j=1$, replace the first  column  with  either that of   $H^{(1)}$  or $H^{(2)}$; in the former case, 
we need to replace the whole block, otherwise,  there exist  two columns proportional to each other; 
in the latter case,  just replace first column of $-R_{p_{i,1}}(0)$  or  the  last column of $-R_{p_{i,1}}(0)^*$; 
\item[(ii-2)]  when $j>1$, just replace the first column of $-R_{p_{i,j}}(0)$ or  the last column of $-R_{p_{i,j}}(0)^*$) with that of   $H^{(2)}$ because of $p_{i,1}<p_{i,j}$. 
\end{itemize}
\end{itemize}

Use the Laplace expansion  after the above replacement  procedure,  and we  then  separate  both  the determinants
$
z_0 \mathbb{I}_{r_0}-J_2^{(2)}
$ and $\overline{z_0} \mathbb{I}_{r_0}-{J_{2}^{(2)}}^*$. 
 Finally,   use  the expansion formula \eqref{detcolumnreplace} again,  collect the leading terms  and we thus rebuild  them into   the  first term on  the right-hand side of \eqref{determinantcalcuequa}.  All the other are left  in the error term.
\end{proof}
\begin{lemma} \label{complexbeta2Jordan}
With  the Jordan block   $R_{p}\left( \theta\right)$    in  \eqref{1Rpirealedge}, let 
\begin{equation}
\theta=z_0+N^{-\frac{1}{2\kappa}} \hat{\theta},
\end{equation}
and 
  \begin{small} 
\begin{equation}\label{Ppibeta2edge}
T_{p}(\theta)=\begin{bmatrix}
-\mathbb{I}_{p}\otimes Y^*  &   \mathbb{I}_{p}\otimes Z-
R_{p}(\theta) 
\otimes \mathbb{I}_{{K_1}}  \\
 \mathbb{I}_{p}\otimes W^*-
R_{p}(\theta)^* 
\otimes \mathbb{I}_{{K_2}}
  &  \mathbb{I}_{p}\otimes Y
\end{bmatrix}.
\end{equation}
\end{small}
With the assumptions  in  \eqref{Symbol} where $\omega=\frac{1}{2\xi}$,  if  $|z_0|>1$, then 
as $N\to \infty$  
the following  hold   uniformly for  
$\hat{\Lambda}$
in a compact subset of $\mathbb{C}^{K_{1}+K_2}$.
\begin{small}
\begin{equation} 
 \det\left( T_{p}(\theta)\right)=\big(
(-1)^{p\left(K_1K_2+K_1+K_2\right)+K_2^2} +O(\|Y\|)\big)
 \det\!\Big(S^{(2,\mathrm{o})}_N(\hat{\theta};\frac{p}{\kappa})+\Delta^{(2,\mathrm{o})}(Y) \Big),
\end{equation}
\end{small}
where 
\begin{equation}\label{S2outlier}
S^{(2,\mathrm{o})}_N(\hat{\theta};\frac{p}{\kappa})=
\begin{cases}
\begin{bmatrix}
 ( -1)^{p}  N^{-\frac{p}{2\xi}}  \big(\overline{\hat{W}}^p+O(N^{\frac{\kappa-\xi}{2\kappa\xi}})\big)
 & -Y  \\ Y^* &
 ( -1)^{p}  N^{-\frac{p}{2\xi}}  \big(\hat{Z}^p+O(N^{\frac{\kappa-\xi}{2\kappa\xi}})\big)
\end{bmatrix}, & \xi>\kappa;\\
\begin{bmatrix}
  N^{-\frac{p}{2\xi}}( \overline{\hat{\theta}}\mathbb{I}_{K_2}  - \overline{\hat{W}})^{p}
 & -Y  \\ Y^* &
 N^{-\frac{p}{2\xi}} (\hat{\theta}\mathbb{I}_{K_1}  -\hat{Z})^{p}
\end{bmatrix}, & \xi=\kappa;\\
\begin{bmatrix}
N^{-\frac{p}{2\kappa}}   \big(\overline{\hat{\theta}}^{p}\mathbb{I}_{K_2} +O(N^{\frac{\xi-\kappa}{2\kappa\xi}})\big)
 & -Y  \\ Y^* &
 N^{-\frac{p}{2\kappa}}  \big ( \hat{\theta}^{p} \mathbb{I}_{K_1}   +O(N^{\frac{\xi-\kappa}{2\kappa\xi}})\big)
\end{bmatrix}, & \xi<\kappa,
\end{cases}
\end{equation}
and  $$\Delta^{(2,\mathrm{o})}(Y)=O(\|Y\|^2)+O(N^{-\frac{1}{2\xi\vee \kappa}}\|Y\|).$$

\end{lemma}
\begin{proof} Let $K=K_1+K_2$.  
We proceed in two steps.

 {\bf Step 1: Recursion.}
 Since
\begin{small}
\begin{equation*}\label{beta2simplification4}
\begin{aligned}
&\mathbb{I}_{p}\otimes Z-R_p\left( \theta \right)\otimes \mathbb{I}_{K_1}=
\begin{bmatrix}
E_1 & -\mathbb{I}_{K_1} & \cdots & 0 \\  & \ddots & \ddots & \vdots \\
&& \ddots & -\mathbb{I}_{K_1} \\ &&& E_1
\end{bmatrix},                        \\
&\mathbb{I}_{p}\otimes W^*-R_p\left( \theta \right)^*\otimes \mathbb{I}_{K_2}=
\begin{bmatrix}
E_2 &&& \\
-\mathbb{I}_{K_2} & E_2 && \\
\vdots & \ddots & \ddots   & \\
0 & \cdots & -\mathbb{I}_{K_2} & E_2
\end{bmatrix},
\end{aligned}
\end{equation*}
\end{small}
where
\begin{small}
\begin{equation}\label{E1E2beta2complexedge}
E_1=Z-\left( z_0+N^{-\frac{1}{2\kappa}}\hat{\theta} \right)\mathbb{I}_{K_1},\quad 
E_2=W^*-\left( \overline{z_0}+N^{-\frac{1}{2\kappa}}\overline{\hat{\theta}} \right)\mathbb{I}_{K_2}.
\end{equation}
\end{small}
Elementary matrix operations allow us to  rewrite the determinant of $T_{p}(\theta)$ as  that of a $p\times p$  block tridiagonal  matrix with corners 
\begin{small}
\begin{equation*}
\det\left(T_p(\theta) \right)=(-1)^{(p-1)(K_1^2+K_2^2)}\det\left(T^{(0)}\right),\quad 
T^{(0)}=\left[ t_{ij}^{(0)} \right],
\end{equation*}
\end{small}
where 
\begin{small}
\begin{equation*}\label{tij0beta2edge}
\begin{aligned}
&t_{ii}^{(0)}=\begin{bmatrix}
-Y^* & -\mathbb{I}_{K_1}  \\  -\mathbb{I}_{K_2}  &  Y
\end{bmatrix},\quad  
 t_{i,i+1}^{(0)}=\begin{bmatrix}
0  &  0  \\  E_2  &  0
\end{bmatrix},\quad
t_{i+1,i}^{(0)}=\begin{bmatrix}
0  &  E_1  \\  0  &  0
\end{bmatrix},\quad  i=1,\cdots,p-1, \\
&t_{1,p}^{(0)}=\begin{bmatrix}
0  &  E_1  \\  0  &  0
\end{bmatrix},\quad
t_{p,1}^{(0)}=\begin{bmatrix}
0  &  0  \\  E_2  &  0
\end{bmatrix},\quad
t_{pp}^{(0)}=\begin{bmatrix}
-Y^* &  \\ &  Y
\end{bmatrix}.
\end{aligned}
\end{equation*}
\end{small}

We  thus deduce  from 
the  recursion  
\begin{small}
\begin{equation*}\label{inductionTprbeta2edge}
\det\left(T^{(r-1)}\right)=\det\left(t_{11}^{(r-1)}\right)\det\left(T^{(r)}\right),  \ T^{(r)}=\Big[t_{i,j}^{(r)} \Big]_{i,j=1}^{p-r}, \quad r=1,2,\dots, p-1,
\end{equation*}
\end{small}
that 
\begin{small}
\begin{equation}\label{recursiondetbeta2}
\det\left( T_{p}(\theta)\right) =(-1)^{(p-1)(K_1^2+K_2^2)}
\det\left( T^{(p-1)} \right)\prod_{i=0}^{p-2}\det\left( t_{1,1}^{(i)} \right),
\end{equation}
\end{small}
where $T^{(r)}$ is a block tridiagonal  matrix with  corners
\begin{small}
\begin{equation}\label{tijrbeta2edge}
\begin{aligned}
&t_{11}^{(r)}=\begin{bmatrix}
-Y^*-E_1A_{22}^{(r-1)}E_2 & -\mathbb{I}_{K_1}  \\  -\mathbb{I}_{K_2}  &  Y
\end{bmatrix},\quad
t_{ii}^{(r)}=\begin{bmatrix}
-Y^* & -\mathbb{I}_{K_1}  \\  -\mathbb{I}_{K_2}  &  Y
\end{bmatrix},\quad i=2,\cdots,p-r-1,  \\
&t_{i,i+1}^{(r)}=\begin{bmatrix}
0  &  0  \\  E_2  &  0
\end{bmatrix},    
t_{i+1,i}^{(r)}=\begin{bmatrix}
0  &  E_1  \\  0  &  0
\end{bmatrix},\quad  i=1,\cdots,p-r-1,                      \\
&t_{1,p-r}^{(r)}=\begin{bmatrix}
0  &  E_1^{(r)}  \\  0  &  0
\end{bmatrix},\quad
t_{p-r,1}^{(r)}=\begin{bmatrix}
0  &  0  \\  E_2^{(r)}  &  0
\end{bmatrix},\quad
t_{p-r,p-r}^{(r)}=\begin{bmatrix}
-Y^* &  \\ &  Y-\sum_{i=0}^{r-1}E_2^{(i)}A_{11}^{(i)}E_1^{(i)}
\end{bmatrix}.
\end{aligned}
\end{equation}
\end{small}
Here 
\begin{small}
\begin{equation}\label{E0A0beta2}
E_i^{(0)}:=E_i,\quad i=1,2,\quad   A^{(0)}:=
\begin{bmatrix}
A_{11}^{(0)}  &  A_{12}^{(0)}  \\  A_{21}^{(0)}  &  A_{22}^{(0)}
\end{bmatrix}
:=\begin{bmatrix}
-Y^* & -\mathbb{I}_{K_1}  \\  -\mathbb{I}_{K_2}  &  Y
\end{bmatrix}^{-1},
\end{equation}
\end{small}
and for $1\leq r\leq p-1$
  \begin{small}
\begin{equation*}\label{ErArbeta2edge}
\begin{aligned}
&E_1^{(r)}:=-E_1A_{21}^{(r-1)}E_1^{(r-1)},\quad E_2^{(r)}:=-E_2^{(r-1)}A_{12}^{(r-1)}E_2,  \\
&A^{(r)}:=
\begin{bmatrix}
A_{11}^{(r)}  &  A_{12}^{(r)}  \\  A_{21}^{(r)}  &  A_{22}^{(r)}
\end{bmatrix}
:=\begin{bmatrix}
-Y^*-E_1A_{22}^{(r-1)}E_2 & -\mathbb{I}_{K_1}  \\  -\mathbb{I}_{K_2}  &  Y
\end{bmatrix}^{-1}.
\end{aligned}
\end{equation*}
\end{small}

 {\bf Step 2:  Taylor expansions}.
Using   \eqref{Symbol} and 
   \eqref{E1E2beta2complexedge}, we can easily get  
\begin{equation}\label{Eibeta2edgeasmptotic}
E_1=-N^{-\frac{1}{2\kappa}}\hat{\theta}\mathbb{I}_{K_1}+N^{-\frac{1}{2\xi}}\hat{Z},\quad
E_2=-N^{-\frac{1}{2\kappa}}\overline{\hat{\theta}}\mathbb{I}_{K_2}+N^{-\frac{1}{2\xi}}\hat{W}^*.
\end{equation}
Note that as $\|Y\| \to 0$   the leading term of $A^{(0)}$ in  \eqref{E0A0beta2}  is given by
\begin{small}
\begin{equation*}\label{A0asymptoticbeta2edge}
A^{(0)}=\begin{bmatrix}
-Y & -\mathbb{I}_{K_2}  \\  -\mathbb{I}_{K_1}  &  Y^*
\end{bmatrix}+O(\|Y\|^2),
\end{equation*}
\end{small}
by induction  we can prove that 
\begin{small}
\begin{equation}\label{Arasymptoticbeta2edge}
A^{(r)}=\begin{bmatrix}
-Y & -\mathbb{I}_{K_2}  \\  -\mathbb{I}_{K_1}  &  Y^*+E_1A_{22}^{(r-1)}E_2
\end{bmatrix}+O(\|Y\|^2).
\end{equation}
\end{small}
Moreover, we obtain 
\begin{small}
\begin{equation}\label{A22rasymptoticbeta2edge}
A_{22}^{(r)}
=Y^*+\sum_{i=1}^r E_1^i Y^* E_2^i + O(\|Y\|^2),
\end{equation}
\end{small}
and 
\begin{small}
\begin{equation}\label{Erasymptoticbeta2edge}
E_i^{(r)}
=E_i^{r+1}+O(\|Y\|^2),\quad i=1,2,
\end{equation}
\end{small}
whenever $r=1, \ldots,p-1$.

For $ i=0, \ldots, p-2$, from (\ref{tijrbeta2edge}) and (\ref{A22rasymptoticbeta2edge}), it's easy to obtain   
\begin{small}
\begin{equation*}\label{prodibeta2}
 \det\left( t_{1,1}^{(i)} \right)=(-1)^{K_1K_2+K} +O(\|Y\|).
\end{equation*}
\end{small}
So  one knows from the recursion \eqref{recursiondetbeta2} that  the remaining  task is to calculate the determinant of  $T^{(p-1)}$ which is defined as 
\begin{small}
\begin{equation*}\label{Tppminus1beta2edge}
T^{(p-1)}=\begin{bmatrix}
-Y^*-E_1A_{22}^{(p-2)}E_2      &      E_1^{(p-1)} \\    E_2^{(p-1)}
& Y-\sum_{i=0}^{p-2}E_2^{(i)}A_{11}^{(i)}E_1^{(i)}
\end{bmatrix}.
\end{equation*}
\end{small}
From (\ref{Eibeta2edgeasmptotic}), (\ref{Arasymptoticbeta2edge}), (\ref{A22rasymptoticbeta2edge}) and (\ref{Erasymptoticbeta2edge}), we observe that
\begin{equation*}\label{Tpminus1asymptoticout}
T^{(p-1)}=\begin{bmatrix}
-Y^*  &  E_1^p \\
E_2^p  &  Y
\end{bmatrix}+\Delta^{(2,\mathrm{o})}(Y),
\end{equation*}
where
 $$\Delta^{(2,\mathrm{o})}(Y)=O(\|Y\|^2)+O(N^{-\frac{1}{2\xi\vee \kappa}}\|Y\|).$$
Compare $\xi$ and $\kappa$, and perform elementary operations of columns and rows of determinant, we can complete the proof of the desired lemma.
\end{proof}

\begin{proof}[Proof of  Theorem \ref{2-complex} (i)] 
When $ {\rm Tr}(YY)^*\leq \delta$, 
 for small $\delta$ take a Taylor expansion and rewrite \eqref{fbetaYZ0} as 
\begin{small}
\begin{equation}\label{f2YZ0decomposition}
f_2(Y,Z_0)=f_{2,0}(Y)+f_{2,1}(Y),
\end{equation}
\end{small}
where
\begin{equation*}\label{f2m}
f_{2,0}(Y)=\frac{1}{2}{\rm Tr}\left( (YY^*)^2 \right),
\end{equation*}
and 
\begin{equation}\label{f2r}
f_{2,1}(Y)=O\big(\|Y\|^6\big).
\end{equation}

At the  edge we have chosen the exponent  $\omega={1}/{2}$, for $h_2(Y)$ given in \eqref{hbetaYLambda},   by Lemma \ref{Cbetacrudebound} we  obtain 
\begin{small}
\begin{equation}\label{Nh2YLambdaAsympto}
Nh_2(Y)=\sqrt{N}{\rm Tr}(H_2)-\frac{1}{2}{\rm Tr\left(H_2^2\right)}+O\left(N^{-\frac{1}{2}}\right).
\end{equation}
\end{small}
Simple manipulations upon  on  \eqref{Cbetadetailed} give rise to           
\begin{small}
\begin{equation}\label{TrH2}
{\rm Tr}(H_2)={\rm Tr}(Z_0^{-1}\hat{\Lambda})
-\left(z_0^{-1}{\rm Tr}(Y^*Y\hat{Z})+\overline{z_0}^{-1}{\rm Tr}(YY^*\hat{W}^*) \right)
+
O\big(\|Y\|^4\big),
\end{equation}
\end{small}
and 
\begin{small}
\begin{equation}\label{TrH2H2}
{\rm Tr}(H_2^2)={\rm Tr}(Z_0^{-2}\hat{\Lambda}^2)+O(\|Y\|^2).
\end{equation}
\end{small}
The latter  follows from the rough estimates 
\begin{small}
\begin{equation*}\label{H2iAsymptotic}
H_2^i=\begin{bmatrix}
z_0^{-i}\hat{Z}^i+O(\|Y\|^2)   &  O(\|Y\|) 
\\     O(\|Y\|)&     \overline{z_0}^{-i}\left(\hat{W}^*\right)^i+O(\|Y\|^2) 
\end{bmatrix},\quad i\geq 1,
\end{equation*}
\end{small}
which can be easily verified by induction.

Combining \eqref{Nh2YLambdaAsympto}, \eqref{TrH2} and \eqref{TrH2H2}, we obtain 
\begin{small}
\begin{align}\label{Nh2YLambdaAsymptofinalform}
Nh_2(Y)&=h_{2,0,0}-\sqrt{N}
h_{2,0}(Y)+h_{2,1}(Y)
\end{align}
\end{small}
where
\begin{small}
\begin{equation*}\label{h2Lambda}
h_{2,0,0}=\sqrt{N}{\rm Tr}(Z_0^{-1}\hat{\Lambda})-\frac{1}{2}{\rm Tr}(Z_0^{-2}\hat{\Lambda}^2),
 \,h_{2,0}(Y)=z_0^{-1}{\rm Tr}(Y^*Y\hat{Z})+\overline{z_0}^{-1}{\rm Tr}(YY^*\hat{W}^*), 
\end{equation*}
\end{small}
and
\begin{small}
\begin{equation}\label{Ph2Y}
h_{2,1}(Y)=
O\left( \|Y\|^2+\sqrt{N}\|Y\|^4+N^{-\frac{1}{2}} \right).
\end{equation}
\end{small}

The remaining task is to estimate  the factors of   $g_2(Y,\Lambda)$ in \eqref{gbetaYLambda},  which are directly related to    two matrices $B_{11}^{(2)}$ and $A_2$.  For the latter,
 noting the form of $A_2$ given in  \eqref{Q2trasform0} and the rescaling of $\Lambda$  in \eqref{Symbol}, we have      
\begin{small}
\begin{equation*}\label{A2asymtobeta2edge}
A_2=Z_0+O\big(\big(N^{-\frac{1}{2}}+\|Y\|\big)\big).
\end{equation*}
\end{small}
Furthermore,  we have 
\begin{equation*}\label{detA2beta2edge}
\det(A_2)
=|z_{0}|^{2K_1}+O\big(\big(N^{-\frac{1}{2}}+\|Y\|\big)\big),
\end{equation*}
from which 
\begin{small}
\begin{equation*}\label{detA2ssquarebeta2edge}
\left( \det(A_2) \right)^{-r}=|z_{0}|^{-2r K_1}+O\big(\big(N^{-\frac{1}{2}}+\|Y\|\big)\big).
\end{equation*}
\end{small}
It remains to calculate the determinant of $B_{11}^{(2)}$  defined in  \eqref{B112denotion}. This is already 
done in Lemma \ref{determinantcalculationbeta2edge}.
Using    Lemma  \ref{determinantcalculationbeta2edge}  and  the same notation  therein we arrive at  \begin{small}
\begin{equation}\label{g2YLambdabeta2edge}
g_2(Y,\Lambda)= \Big(D^{(2)}+O\big(N^{-\frac{1}{4\kappa}}+\|Y\|\big)\Big)
\bigg\{   \det\begin{bmatrix}
N^{-\frac{1}{4}}\Omega_{11} & \Omega_{12} \\ \Omega_{21} & N^{-\frac{1}{4}}\Omega_{22}
\end{bmatrix}+                                                       
O\Big(
\sum_{(\alpha_1,\alpha_2)}N^{-\alpha_1}\| Y \|^{\alpha_2}
\Big)\bigg\},
\end{equation}
\end{small}
where 
\begin{small}
\begin{equation}\label{D2-2edge}
D^{(2)}=|z_{0}|^{-2r K_1} \prod_{i=m+1}^{\gamma}|z_0-\theta_i|^{2K_1  \sum_{j=1}^{\alpha_i}p_{i,j}\beta_{i,j}}.\end{equation}
\end{small}
Recall   \eqref{Ideltabeta}, \eqref{f2YZ0decomposition} and \eqref{Nh2YLambdaAsymptofinalform},  we rewrite 
\begin{equation}\label{Idelta2decompositionbeta2edge}
  e^{-h_{2,0,0}}I_{\delta,N}^{(2)}= J_{1,N}^{(2)}+J_{2,N}^{(2)},
\end{equation}
 where 
\begin{small}
\begin{equation*}\label{Idelta2decompositionbeta2edge1}
 J_{1,N}^{(2)}=
\int\limits_{{\rm Tr}(YY^*)\leq \delta}\ g_{2}(Y,\Lambda)
e^{ -Nf_{2,0}(Y)-\sqrt{N}h_{2,0}(Y) } {\rm d}Y,
\end{equation*}
\end{small}
and 
\begin{small}
\begin{equation*}\label{Idelta2decompositionbeta2edge2}
 J_{2,N}^{(2)}=
\int\limits_{{\rm Tr}(YY^*)\leq \delta}\ g_{2}(Y,\Lambda)
e^{ -Nf_{2,0}(Y)-\sqrt{N}h_{2,0}(Y) } \left(
e^{-Nf_{2,1}(Y)+h_{2,1}(Y)}-1\right) {\rm d}Y.
\end{equation*}
\end{small}
We claim that    
$J_{1,N}^{(2)}$  plays a leading role while   $J_{2,N}^{(2)}$  gives rise  to 
a relatively minor  contribution  when compared.  To be precise,   we have an estimate
 \begin{equation} \label{J21beta2} 
 J_{2,N}^{(2)}=O(N^{-\frac{1}{4}})J_{1,N}^{(2)}.  
 \end{equation}
For $J_{1,N}^{(2)}$,  after using   the  change of variables 
\begin{equation} \label{beta2Yscaling}
Y\to N^{-\frac{1}{4}}Y,
\end{equation}
 we can remove all restrictions on the integration domains of  $Y$ up to some decaying  exponential term. Also noting  \eqref{g2YLambdabeta2edge}, the standard asymptotic analysis shows that 
\begin{small}
\begin{equation}\label{I1delta2minusI2delta2beta2edge}
 J_{1,N}^{(2)}=D^{(2)}  N^{-\frac{1}{2}K_{1}^2- \frac{1}{2}K_{1}
 \sum_{i=1}^m\sum_{j=1}^{\alpha_i}\beta_{ij}  }
 \left( J_{\infty}^{(2)}+O\big( N^{-\frac{1}{4\kappa}} \big) \right),
\end{equation}
\end{small}
where
\begin{small}
\begin{equation*}\label{I0delta2beta2edge}
J_{\infty}^{(2)}=I^{(2)}\left(  
\big( -\hat{\theta}_1 \big)^{p_{1,1}},\cdots,\big( -\hat{\theta}_m \big)^{p_{m,1}};
z_0^{-1}\hat{Z},z_0^{-1}\hat{W}
\right).
\end{equation*}
\end{small}
For $J_{2,N}^{(2)}$,  
by the inequality
$
\left| e^z-1 \right|\leq \left| z \right|e^{\left| z \right|}$ we have 
\begin{equation*}\label{Idelta2minusI2delta2beta2edgeestimate1}
\left| e^{-Nf_{2,1}(Y)+h_{2,1}(Y)}-1 \right|\leq 
\big(N|f_{2,1}(Y)|+|h_{2,1}(Y)|\big)
e^{N|f_{2,1}(Y)|+|h_{2,1}(Y)|}.
\end{equation*}
Choose a sufficiently small $\delta>0$,   we can obtain from \eqref{f2r} and \eqref{Ph2Y} that 
\begin{small}
\begin{equation}\label{Idelta2minusI2delta2beta2edgeestimate2}
e^{N|f_{2,1}(Y)|+|h_{2,1}(Y)|} \leq e^{\frac{1}{2}N|f_{2,0}(Y)|+
\sqrt{N}
 O\big(\|Y\|^2\big) 
} O\left(1 \right).
\end{equation}
\end{small}
Moreover,   after   the  change of variables 
 \eqref{beta2Yscaling}   we observe that 
\begin{equation}\label{Idelta2minusI2delta2beta2edgeestimate3}
N|f_{2,1}(Y)|+|h_{2,1}(Y)| \to   N^{-\frac{1}{4}}O(1+\|Y\|^6).
\end{equation}
 Therefore,  as done  in  $J_{1,N}^{(2)}$   the desired estimate  \eqref{J21beta2} immediately follows from \eqref{Idelta2minusI2delta2beta2edgeestimate2} and  \eqref{Idelta2minusI2delta2beta2edgeestimate3}. 
At last,   combine   \eqref{analysisform}  in  Proposition \ref{Analysisform},   
   \eqref{Idelta2decompositionbeta2edge},   \eqref{J21beta2} and \eqref{I1delta2minusI2delta2beta2edge},    we  thus complete the proof  at the  edge.  
\end{proof}

\begin{proof}[Proof of  Theorem \ref{2-complex} (ii)] 
Repeat  the similar procedure as in  the proof of Part (i)  and   we can  verify Part (ii). Here we just point out a few of signifiant differences.  

When $|z_0|>1$,  for small $\delta$ take a   Taylor expansion and rewrite    \eqref{fbetaYZ0} as 
\begin{small}
\begin{equation*}\label{f2outlier}
f_2(Y,Z_0)=f_{2,0}(Y)+O(\|Y \|^4)
\end{equation*}
\end{small}
where
\begin{small}
\begin{equation*}\label{f2mYbeta2outlier}
f_{2,0}(Y)=\left( 1-|z_0|^{-2} \right){\rm Tr}\left( YY^* \right).
\end{equation*}
\end{small}

Use the induction argument as in  the edge case and we can similarly obtain 
\begin{small}
\begin{equation*}\label{H2out}
H_2^i=\begin{bmatrix}
z_0^{-i}\hat{Z}^i+O(\|Y\|^2)   &  O(\|Y\|) 
\\     O(\|Y\|)&     \overline{z_0}^{-i}\left( \hat{W}^* \right)^i+O(\|Y\|^2) 
\end{bmatrix},\quad i\geq 1.
\end{equation*}
\end{small}
Hence,  we see 
 from \eqref{hbetaYLambda} that 
\begin{small}
\begin{equation*}\label{Nh2YLambbdabeta2outlier}
Nh_{2}(Y,\Lambda)=h_{2,0}(\Lambda)+O\big(N^{-\frac{1}{2\xi}}(1+N\|Y\|^2)\big)
\end{equation*}
\end{small}
where
\begin{small}
\begin{equation}\label{h2Lambdabeta2outlier}
h_{2,0}(\Lambda)=\sum_{j=1}^{2\xi}\frac{1}{j}(-1)^{j-1}N^{1-\frac{j}{2\xi}}{\rm Tr}\big( Z_0^{-j}\hat{\Lambda}^j \big).
\end{equation}
\end{small}
As in the edge case we can similarly obtain
\begin{small}
\begin{equation}\label{detA2ssquarebeta2outlier}
\left( \det(A_2) \right)^{-r}=|z_{0}|^{-2r K_1} 
+O\big(\big(N^{-\frac{1}{2\xi}}+\|Y\|\big)\big).
\end{equation}
\end{small}

The remaining task is to calculate the determinant of $B_{11}^{(2)}$  defined in  \eqref{B112denotion}.
Firstly, note that
\begin{equation*}\label{B112B02}
\det\left( B_{11}^{(2)} \right)=(-1)^{r^2K_{1}^2}
\det\left( B_{0}^{(2)} \right),
\end{equation*}
where
\begin{small}
\begin{equation*}\label{B02denotionout}
B_{0}^{(2)}=\begin{bmatrix}
-\mathbb{I}_r \otimes Y^*  &  \mathbb{I}_r \otimes Z-J_2\otimes \mathbb{I}_{K_1}  \\
\mathbb{I}_r \otimes W^*-J_2^*\otimes \mathbb{I}_{K_1}   &   \mathbb{I}_r\otimes Y
\end{bmatrix}.
\end{equation*}
\end{small}
Use the block structure of $J_2$  and rewrite  $B_{0}^{(2)}$  as a block diagonal matrix,   
together with   basic properties of  the determinant we  obtain 

\begin{small}
\begin{equation*}\label{g2complexoutproduct}
\det\!\left(B_{0}^{(2)}\right)=
\prod _{i=1}^\gamma  \prod _{j=1}^{\alpha_{i}}   
\left(\det\! \left(T_{p_{i,j}}\left(\theta_i \right)\right)\right)
^{ \beta_{i,j} } ,
\end{equation*}
\end{small}
where $T_{p}(\theta)$ is defined in \eqref{Ppibeta2edge}.

Some  direct  calculations show that when  $m+1\leq i\leq \gamma$,
\begin{small}
 \begin{equation*}
\begin{aligned}
\det\! \left(T_{p_{i,j}}(\theta_{i} ) \right)
&=(-1)^{p_{ij}^2 K_{1}^2}
|\det\left(
z_0\mathbb{I}_{p_{ij}}\otimes \mathbb{I}_{K_1}-
R_{p_{ij}}(\theta_i) 
\otimes \mathbb{I}_{{K_1}} \right) |^2  
+O\big(\big(N^{-\frac{1}{2\xi}}+\|Y\|\big)\big) \\
&=(-1)^{p_{ij}^2K_{1}^2}|z_0-\theta_i |^{2p_{ij}K_1}
+O\big(\big(N^{-\frac{1}{2\xi}}+\|Y\|\big)\big).
\end{aligned}
 \end{equation*}
\end{small}

Together  with \eqref{detA2ssquarebeta2outlier},  using of Lemma  \ref{complexbeta2Jordan}  gives    
\begin{small}
\begin{align}\label{g2YLambdabeta2outlier}
g_2(Y,\Lambda)&= \Big(D^{(2,\mathrm{o})}+O\big(N^{-\frac{1}{2\xi}}+\|Y\|\big)\Big)
\prod_{i=1}^{m } \prod_{j=1}^{{\alpha}_{i}} 
\Big(
{\det}\,\Big(
S^{(2,\mathrm{o})}_N\big(\hat{\theta}_{i};\frac{p_{i,j}}{p_{i,1}}\big)+\Delta^{(2,\mathrm{o})}(Y) \Big)
\Big)^ {\beta_{i,j} },
\end{align}
\end{small}
where 
\begin{small}
\begin{align}\label{D2-2outlier}
D^{(2,\mathrm{o})}=&|z_{0}|^{-2r K_1}  
(-1)^{K_1^2\sum_{i=1}^m\sum_{j=1}^{\alpha_i}
 \beta_{i,j}} 
\prod_{i=m+1}^{\gamma}\prod_{j=1}^{\alpha_i} |z_0-\theta_i|^{2p_{i,j}\beta_{i,j}K_1},
\end{align}
\end{small}
$$\Delta^{(2,\mathrm{o})}(Y)=O(\|Y\|^2)+O(N^{-\frac{1}{2 \xi}}\|Y\|),$$
and
$$ \xi=\left\{ p_{i,1}:i=1,\cdots,m \right\}. $$

Finally,    following the almost same procedure as in the proof of Part (i) and using    the  change of variables 
\begin{equation*} 
Y\to N^{-\frac{1}{2}}Y,
\end{equation*}
we can complete the proof.
\end{proof}

\subsection{Deformed GinSE ensemble}


In order to provide the proof of  Theorem \ref{4-complex}  we also need   one key lemma.

\begin{lemma}\label{4lemma}

Let  $R, S$ and $T$ be  block matrices given in \eqref{RSTdeno},  
 $\kappa=\max\left\{ 1/2, p_{1,1},\cdots,p_{m,1} \right\}$ and  $t=\sum_{i=1}^m\sum_{j=1}^{\alpha_i}\beta_{i,j}$. 
Given an integer $0\leq m\leq \gamma$, assume that $ \theta_{j}\neq z_0$ is fixed  for all $j=m+1, \ldots, \gamma$ and                              
 \begin{equation}\label{lemma4zi}
\theta_i=z_0+N^{-\frac{1}{4p_{i,1}}}\hat{\theta}_i,\quad
i=1,\cdots,m.
\end{equation}

(i) If $z_0=\pm 1$, then
\begin{small}
\begin{equation}\label{4Rlemmaequ}
\begin{aligned}
{\rm Pf}&\begin{bmatrix}
R & S  \\  -S^t  &  T
\end{bmatrix}=\prod_{i=m+1}^{\gamma}
\left| z_0-\theta_i \right|^{4K_1 \sum_{j=1}^{\alpha_i} p_{i,j}\beta_{i,j}}\\
&\times {\rm Pf}\begin{bmatrix}
R\left( Y;\left(V_1\right)_{I_2},\left(U_1\right)_{I_1} \right) & N^{-\frac{1}{4}}S\left( N^{\frac{1}{4}}Y \right)  \\
-N^{-\frac{1}{4}}S\left( N^{\frac{1}{4}}Y \right)^t  &  R\left( Y^*;\left(U_1\right)_{I_1},\left(V_1\right)_{I_2} \right)
\end{bmatrix}
+O\bigg( \sum_{(\alpha_1,\alpha_2,\alpha_3)} N^{-\alpha_1}\| Y \|^{\alpha_2} \Xi^{\alpha_3} \bigg),
\end{aligned}
\end{equation}
\end{small}
where    $R\left( Y;A,B \right)$ and $S\left( Y \right)$ are  given respectively  in \eqref{4RRdeno} and \eqref{4RSdeno} with $\phi_i=( -\hat{\theta}_i)^{p_{i,1}}$, and 
the finitely many integer triples  $(\alpha_1,\alpha_2,\alpha_3)$  satisfy 
\begin{equation}\label{4RSumcondi1}
4\alpha_1+\alpha_2+\alpha_3\geq 4K_1t +\frac{1}{\kappa} \delta_{\alpha_3,0},
\end{equation}
where   $\delta_{\alpha_3,0}=1$ when  $\alpha_3=0$, otherwise $0$.

(ii) If    $ \Im z_0 >0$ and   $|z_0|=1$,  then
\begin{small}
\begin{equation}\label{4Clemmaequ}
\begin{aligned}
&{\rm Pf}\begin{bmatrix}
R & S  \\  -S^t  &  T
\end{bmatrix}=(-1)^{K_1t}
\left| \overline{z}_0-z_0 \right|^{2K_1(r-r_0)}
\prod_{i=m+1}^{\gamma}\prod_{j=1}^{\alpha_i}
\left| \left( z_0-\theta_i \right)\left( z_0-\overline{\theta_i} \right) \right|^{2K_1p_{i,j}\beta_{i,j}}
\\
&\times\det\begin{bmatrix}
N^{-\frac{1}{4}}\Omega_{11} & Y_{12}\otimes \left(V_2\right)_{I_2}
\\  
-Y_{12}^*\otimes \overline{\left(U_2\right)_{I_1}}   &  -N^{-\frac{1}{4}}\Omega_{11}^*
\end{bmatrix}
+O\bigg( \sum_{(\alpha_1,\alpha_2,\alpha_3,\alpha_4)} N^{-\alpha_1}\| Y_{11},Y_{22} \|^{\alpha_2} \| Y_{12} \|^{\alpha_3}   \Xi^{\alpha_4} \bigg),
\end{aligned}
\end{equation}
\end{small}
where $\Omega_{11}$ is defined in \eqref{F11denotion} with $\phi_i=( -\hat{\theta}_i)^{p_{i,1}}$,  and the finitely many  quartets $(\alpha_1,\alpha_2,\alpha_3,\alpha_4)$ satisfy 
\begin{equation}\label{4CSumcondi1}
4\alpha_1+\alpha_2+\alpha_3+\alpha_4\geq 2K_1t+\frac{1}{\kappa} \delta_{\alpha_2+\alpha_4,0}.
\end{equation} Here  for  two  $K_1\times K_1$ matrices $Y_{11}$ and  $Y_{22}$,  $Y$ has the form \begin{small} $$Y=\begin{bmatrix} Y_{11} & Y_{12} \\ Y_{12}^t & Y_{22} \end{bmatrix}.$$ \end{small}
 \end{lemma}

\begin{proof} 
A useful starting point is  the Laplace-type expansion formulas for the  Pfaffian  \cite[Proposition 2.3]{Ok}
\begin{small} 
\begin{equation}\label{lemmaproofkey1}
\begin{aligned}
{\rm Pf}\begin{bmatrix}
R & S  \\  -S^t  &  T
\end{bmatrix}=\sum_{I,J}\epsilon_r(I,J)
{\rm Pf}\left( R[I] \right){\rm Pf}\left( T[J] \right)
\det S\left[ [4rK_1]\setminus I;[4rK_1]\setminus J \right],
\end{aligned}
\end{equation} \end{small} 
where  the sum is taken over   all pairs of even-element  subsets    $[4rK_1]:=\{1,2,\ldots, 4rK_1\}$ such that $\#I=\#J$, and  the coefficient is given by
\begin{small}
 \begin{equation*}\label{4sign1}
\epsilon_r(I,J)=(-1)^{\sum(I)+\sum(J)+\binom{4rK_1-\#I}{2}}, \quad \sum(I)=\sum_{i\in I}i. 
\end{equation*}
\end{small} 
Note that here  we  use the following same  notation  for submatrix  as before but with a different meaning:   For a  $n\times n$ matrix $X$ and for   two subsets $I,J \subset  [n]=\{1,2\ldots,K\}$,   we denote by  $X[I;J]$  the
submatrix  of $X$ whose    rows and columns  are indexed by $I$ and $J$ respectively.

{\bf Case (i)}: $z_0=\pm 1$.  Introduce two   index sets
\begin{equation*}\label{poorfsetR}
\begin{aligned}
&I_{row}=\big\{ jr+ I_2:  j=0, \ldots, 2K_1-1\big\}
\cup \big\{ (2K_{1}+j)r+ I_1:  j=0, \ldots, 2K_1-1\big\},
\\&I_{col}=\big\{ jr+ I_1:  j=0, \ldots, 2K_1-1\big\}
\cup \big\{ (2K_{1}+j)r+ I_2:  j=0, \ldots, 2K_1-1\big\}.
\end{aligned}
\end{equation*}
We claim that  all the leading  terms  on the right-hand side of \eqref{lemmaproofkey1}  have  row and column indexes that satisfy  
\begin{equation}\label{4Ronly}
I\subset I_{row},\quad
J\subset I_{col},
\end{equation}
and the other terms can be treated as error. 

In fact, for $\#I=\#J=2k$ with  integer $k>0$, if $J\cap \left( [4rK_1]\setminus I_{col} \right)\not=\emptyset$, or $I\cap \left( [4rK_1]\setminus I_{row} \right)\not=\emptyset$,  we can use the column replacement approach as in Lemma \ref{determinantcalculationbeta2edge} to obtain 
$$ \det S\left[ [4rK_1]\setminus I;[4rK_1]\setminus J \right]=O\left( \sum N^{-\alpha_1}\| Y \|^{\alpha_2}   \Xi^{\alpha_3} \right), $$
where the  triple satisfies  the restriction  $4\alpha_1+\alpha_2+\alpha_3\geq 4K_1t-2k+\frac{1}{\kappa}$. And meanwhile, we have  $$ {\rm Pf}\left( R[I] \right){\rm Pf}\left( T[J] \right)=O\Big( \sum_{\alpha_2+\alpha_3=2k}\| Y \|^{\alpha_2}   \Xi ^{\alpha_3} \Big). $$ 
Hence,  the product of these two  is contained in the error term.

Now  consider the terms on the right-hand side of  \eqref{lemmaproofkey1}  with indexes  satisfying   \eqref{4Ronly}.  Take a similar expansion of the Pfaffian on the right-hand side of  
\eqref{4Rlemmaequ}
\begin{small}
\begin{equation}\label{lemmaproofkeyR}
\begin{aligned}
{\rm Pf}\begin{bmatrix}
R_n & S_n  \\  -S_n^t  &  T_n
\end{bmatrix}=\sum_{I',J'}\epsilon_t(I',J')
{\rm Pf}\left( R_n[I'] \right){\rm Pf}\left( T_n[J'] \right)
\det S_n\left[ [4tK_1]\setminus I';[4tK_1]\setminus J' \right],
\end{aligned}
\end{equation}
\end{small}
where for simplicity
$$R_n:=R\left( Y;\left(V_1\right)_{I_2},\left(U_1\right)_{I_1} \right),
\quad T_n:=R\left( Y^*;\left(U_1\right)_{I_1},\left(V_1\right)_{I_2} \right),\quad
S_n:=N^{-\frac{1}{4}}S\big( N^{\frac{1}{4}}Y \big).
$$ 
Note that $\#I_{row}=\#I_{col}=4tK_1$,  it is easy to find a natural bijection   between $(I,J)$ in \eqref{4Ronly}
and $(I',J')$ in \eqref{lemmaproofkeyR}. 
For example,  $I=\left\{ p_{11},r+p_{11},2K_1r+1 \right\}$ corresponds to $I'=\left\{ 1,t+1,2K_1t+1 \right\}$, and  $J=\left\{ r+p_{11}+1,(2K_1+1)r+2p_{11} \right\}$  to $J'=\left\{ t+2,(2K_1+1)t+2 \right\}$.

Therefore,   it suffices to prove  that 
\begin{small} 
\begin{equation*}
\begin{aligned}
&\epsilon_r(I,J) 
{\rm Pf}\left( R[I] \right)  {\rm Pf}\left( T[J] \right)
\det S\left[ [4rK_1]\setminus I;[4rK_1]\setminus J \right]=    \prod_{i=m+1}^{\gamma}
\left| z_0-\theta_i \right|^{4K_1\sum_{j=1}^{\alpha_i}p_{i,j}\beta_{i,j}} \\
&\times \epsilon_t(I',J')
{\rm Pf}\left( R_n[I'] \right){\rm Pf}\left( T_n[J'] \right)    \det S_n\left[ [4tK_1]\setminus I';[4tK_1]\setminus J' \right]+O\left( \sum N^{-\alpha_1}\| Y \|^{\alpha_2} \Xi ^{\alpha_3} \right).
\end{aligned}
\end{equation*}\end{small} 
For the natural bijection   we have
$$ {\rm Pf}\left( R[I] \right)={\rm Pf}\left( R_n[I'] \right),\quad
{\rm Pf}\left( T[J] \right)={\rm Pf}\left( T_n[J'] \right).
 $$
On the other hand,  we proceed as in Lemma \ref{determinantcalculationbeta2edge} and the  column replacement approach gives rise to   
\begin{small} 
\begin{equation*}
\begin{aligned}
\det S&\left[ [4rK_1]\setminus I;[4rK_1]\setminus J \right]=\widetilde{\epsilon}_{2k}(I,J)
   \prod_{i=m+1}^{\gamma} \left| z_0-\theta_i \right|^{4K_1\sum_{j=1}^{\alpha_i}p_{i,j}\beta_{i,j}}  \\
 &\times   \det S_n\left[ [4tK_1]\setminus I';[4tK_1]\setminus J' \right]
+O\left( \sum N^{-\alpha_1}\| Y \|^{\alpha_2} \Xi^{\alpha_3} \right),
\end{aligned}
\end{equation*} \end{small} 
where $\widetilde{\epsilon}_{2k}(I,J)$ takes  value  $\pm 1$  for $\#I=\#I'=\#J=\#J'=2k$,  and the finitely many  triples  $(\alpha_1,\alpha_2,\alpha_3)$ satisfy \eqref{4RSumcondi1}.

Finally, the problem  can be reduced to verifying  an identity
\begin{equation}\label{Rsignequal}
\epsilon_r(I,J)\widetilde{\epsilon}_{2k}(I,J)=\epsilon_t(I',J').
\end{equation}
This can  be proved  by induction on $k$ where $2k$ is the cardinality of  $I$. In this case, let's  rewrite it  as $I_k$.  

The $k=0$ case is trivial since all  three signs  equal to  1 because of  the evenness of $2K_1$.
Suppose that    \eqref{Rsignequal} holds  for  $k-1$,  take some $I_{k-1}\subset I_{k}$ and $J_{k-1}\subset J_{k}$. Note that  \begin{small} 
$$ (-1)^{\binom{4rK_1-2k}{2}}=(-1)^{\binom{4tK_1-2k}{2}}=(-1)^k,$$ \end{small} 
it suffices   to verify \begin{small} 
\begin{equation*}\label{4Rinductionkey}
\frac{\widetilde{\epsilon}_{2k}(I_k,J_k)}{\widetilde{\epsilon}_{2k-2}(I_{k-1},J_{k-1})}
=\frac{(-1)^{\sum(I'_k \setminus I'_{k-1})+\sum(J'_k \setminus J'_{k-1})}}{(-1)^{\sum(I_k \setminus I_{k-1})+\sum(J_k \setminus J_{k-1})}}.
\end{equation*} 
\end{small} 
The left-hand side  refers to  the sign change by deleting two rows and columns.  
We  just consider the change by deleting  two rows since the  column case is similar.

 Let   
$$ I_k \setminus I_{k-1}=\left\{ a_1,a_2 \right\},\quad a_1<a_2,  $$
then there are three cases: (i)
 $1\leq a_1<a_2\leq 2rK_1$; (ii) $1\leq a_1 \leq 2rK_1<a_2\leq 4rK_1$; (iii)
$2rK_1< a_1<a_2\leq 4rK_1$. We only consider case (i) since the other  two cases are very similar.

Now suppose that  between row  $a_1$ and  row $a_2$ there are  $t_2$  $J_2^{(2)}$'s  (as a whole,  cf.  \eqref{J2rdeno}) and $t_1$ Jordan blocks   of the form $R_{p_{ij}}\left( \theta_i \right)$,    $1\leq i \leq m$. 
  In the late case these  blocks contain $a_2$ but  not $a_1$  in the unique way,  and the block  sizes  are denoted  by $p_1,p_2,\cdots, p_{t_1}$.  According to the natural correspondence we have 
$$ (-1)^{\sum(I_k \setminus I_{k-1})}=(-1)^{a_2-a_1}=(-1)^{t_2 r_0+p_1+\cdots+p_{t_1}},
\quad (-1)^{\sum(I_k' \setminus I_{k-1}')}=(-1)^{t_1}, 
 $$
so the sign  change   caused by deleting row $a_1$ and row  $a_2$ when  taking the  Laplace expansion is 
\begin{small} $$ \left( \frac{\widetilde{\epsilon}_{2k}(I_k,J_k)}{\widetilde{\epsilon}_{2k-2}(I_{k-1},J_{k-1})} \right)_{row}
=(-1)^{t_2 r_0+(p_1-1)+\cdots+(p_{t_1}-1)}=\frac{(-1)^{\sum(I'_k \setminus I'_{k-1})}}{(-1)^{\sum(I_k \setminus I_{k-1})}}.
 $$ \end{small} Together with the sign change caused by the  column,  we know  that \eqref{Rsignequal}  holds for  $k$. 

This concludes the proof of Case (i).

{\bf Case (ii)}:  $\Im z_0>0$ \& $|z_0|=1$. As in Case (i), let 
\begin{equation*}\label{poorfsetC}
\begin{aligned}
&I^{c}_{row}=\big\{ jr+ I_2:  j=0, \ldots, K_1-1\big\}
\cup \big\{ (3K_{1}+j)r+ I_1:  j=0, \ldots, K_1-1\big\},
\\&I^{c}_{col}=\big\{ jr+ I_1:  j=0, \ldots, K_1-1\big\}
\cup \big\{ (3K_{1}+j)r+ I_2:  j=0, \ldots, K_1-1\big\},
\end{aligned}
\end{equation*}
  all the leading  terms  on the right-hand side of \eqref{lemmaproofkey1}  have  row and column indexes that satisfy  
$ I\subset I^{c}_{row}$ and $J\subset I^{c}_{col} $
and the other terms can be treated as error. 
We only need to consider the cases in \eqref{4Ronly}, terms in other cases are all contained in the error terms. If  $\#I=\#J=2k>0$,  then  we have
$$ {\rm Pf}\left( R[I] \right) {\rm Pf}\left( T[J] \right)=O\big( \| Y_{11},Y_{22} \|^{2k} \big),$$
and 
$$ \det S\left[ [4rK_1]\setminus I;[4rK_1]\setminus J \right]=O\left( \sum N^{-\alpha_1}\| Y_{11},Y_{22} \|^{\alpha_2} \| Y_{12} \|^{\alpha_3}  \Xi^{\alpha_4} \right), $$
where $4\alpha_1+\alpha_2+\alpha_3+\alpha_4\geq 2K_1t-2k$, from which 
the product of these two   is contained in the error terms. 

Therefore, the only leading term occurs when   $I=J=\emptyset$.  Again we can proceed  as in Lemma \ref{determinantcalculationbeta2edge} and obtain  the desired result.
\end{proof}

Let's first prove   part (ii) of Theorem \ref{4-complex}  in detail and then  part (i) quickly.
\begin{proof}[Proof of Theorem \ref{4-complex} (ii)]
For simplicity,  let's introduce a new matrix  notation 
\begin{small}
\begin{equation*}\label{tildeA1}
\widetilde{A}_1=YZ_0^{-1}Y^*Z_0^{-1},
\end{equation*}
\end{small}
then we  see  from (\ref{fbetaYZ0}) that  
\begin{small}
\begin{equation}\label{f4YZ0tildeA1form}
f_4(Y,Z_0)={\rm Tr}(YY^*)-\log\det\big( \mathbb{I}_{2K_1}+\widetilde{A}_1 \big).
\end{equation}
\end{small}
Rewrite  Y as  a $2\times 2$  block matrix 
\begin{small}
\begin{equation*}\label{Yblockform4}
Y=\begin{bmatrix}
Y_{11} & Y_{12}  \\   Y_{21} &  Y_{22}
\end{bmatrix},  \ Y_{11}^t=Y_{11},  \ Y_{22}^t=Y_{22}, \ Y_{12}^t=Y_{21}, 
\end{equation*}
\end{small}
where each block has the same size,   since  $|z_0|=1$ we have 
\begin{small}
\begin{equation}\label{tildeA1detailedform}
\widetilde{A}_1=\begin{bmatrix}
z_0^{-2}Y_{11}Y_{11}^*+Y_{12}Y_{12}^*  &
Y_{11}Y_{21}^*+\overline{z_0}^{-2}Y_{12}Y_{22}^*   \\
z_0^{-2}Y_{21}Y_{11}^*+Y_{22}Y_{12}^*   &
Y_{21}Y_{21}^*+\overline{z_0}^{-2}Y_{22}Y_{22}^*
\end{bmatrix}.
\end{equation}
\end{small}

Noting  $\widetilde{A}_1=O(\delta)$  whenever  $ {\rm Tr}(YY^*)\leq \delta$, 
 for small $\delta$ take a   Taylor expansion and rewrite    \eqref{f4YZ0tildeA1form} as 
 \begin{equation*}
f_4(Y,Z_0)={\rm Tr}(YY^*)-
{\rm Tr}(\widetilde{A}_1)+\frac{1}{2}{\rm Tr}(\widetilde{A}_{1}^2)+O(\| \widetilde{A}_1 \|^3).
\end{equation*}
The following two terms 
\begin{small}
\begin{equation}\label{TrYYminusTrtildeA1}
{\rm Tr}(YY^*)-{\rm Tr}(\widetilde{A}_1)=(1-z_0^{-2}){\rm Tr}(Y_{11}Y_{11}^*)+(1-\overline{z_0}^{-2}){\rm Tr}(Y_{22}Y_{22}^*)
+2(1-\left| z_0 \right|^{-2}){\rm Tr}(Y_{12}Y_{12}^*).
\end{equation}
\end{small}
play a major role in the following asymptotic analysis.

Particularly when  $|z_0|=1$  the last term  on the right-hand side of \eqref{TrYYminusTrtildeA1} vanishes, 
we need to find the next leading term. Noting  a useful  approximation    from \eqref{tildeA1detailedform} 
\begin{small}
\begin{equation}\label{tildeA1Asympto}
\widetilde{A}_1=\begin{bmatrix}
Y_{12}Y_{12}^*+O(\|Y_{11}\|^2)  &
O(\|Y_{11},Y_{22}\|\|Y_{12}\|)  \\
O(\|Y_{11},Y_{22}\|\|Y_{12}\|)  &
Y_{21}Y_{21}^*+O(\|Y_{22}\|^2)
\end{bmatrix},
\end{equation}
\end{small}
and 
\begin{equation*}
{\rm Tr}\big((Y_{21}Y_{21}^*)^2 \big)={\rm Tr}\big((Y_{12}^t\overline{Y_{12}})^2 \big)
={\rm Tr}\big((Y_{12}Y_{12}^*)^2 \big),
\end{equation*}
since $Y_{21}^{t}=Y_{12}$,
it's easy to obtain  
  \begin{equation}\label{TrtildeA1A1}
{\rm Tr}(\widetilde{A}_1^2)=2{\rm Tr}(Y_{12}Y_{12}^*Y_{12}Y_{12}^*)+
O\big(\|Y_{11},Y_{22}\|^2(\|Y_{12}\|^2+\|Y_{11},Y_{22}\|^2)\big).
\end{equation}

On the other hand,   one can easily   arrive at an estimate   from (\ref{tildeA1Asympto})
 \begin{equation}\label{TrA1A1A1}
\| \widetilde{A}_1 \|^3 =
O\big(\big(\|Y_{12}\|^2+ \|Y_{11},Y_{22}\|^2\big)^3\big).
\end{equation}
  Hence, combining  \eqref{f4YZ0tildeA1form}, \eqref{TrYYminusTrtildeA1}, \eqref{TrtildeA1A1} and \eqref{TrA1A1A1},    we get 
\begin{small}
\begin{equation}\label{f4YZ0decomposition}
f_4(Y,Z_0)=f_{4,0}(Y)+f_{4,1}(Y),
\end{equation}
\end{small}
where
\begin{equation*}\label{f4m}
f_{4,0}(Y)=(1-z_0^{-2}){\rm Tr}(Y_{11}Y_{11}^*)+(1-\overline{z_0}^{-2}){\rm Tr}(Y_{22}Y_{22}^*)
+{\rm Tr}((Y_{12}Y_{12}^*)^2),
\end{equation*}
and 

\begin{equation*}\label{f4r}
f_{4,1}(Y)=O\big(\|Y_{11},Y_{22}\|^2(\|Y_{12}\|^2+\|Y_{11},Y_{22}\|^2)+\|Y_{12}\|^6\big)
.
\end{equation*}

At the  edge we have chosen the exponent  $\omega=1/2$, for $h_4(Y)$ given in \eqref{hbetaYLambda},   we use Lemma \ref{Cbetacrudebound} to obtain 
\begin{small}
\begin{equation}\label{Nh4YLambdaAsympto}
Nh_4(Y)= \sqrt{N}{\rm Tr}(H_4)-\frac{1}{2}{\rm Tr(H_4^2)}+O(N^{-\frac{1}{2}}).
\end{equation}
\end{small}
Simple manipulations based on  \eqref{Cbetadetailed} give rise to           
\begin{small}
\begin{equation}\label{TrC4}
\begin{aligned}
&{\rm Tr}(H_4)=2 {\rm Tr}(Z_0^{-1}\hat{\Lambda})
-2\left(z_0^{-1}{\rm Tr}(Y_{12}Y_{12}^*\hat{Z})+\overline{z_0}^{-1}{\rm Tr}(Y_{12}^*Y_{12}\hat{W}^*) \right)\\
&+
O\big(\|Y_{11},Y_{22}\|(\|Y_{12}\|+\|Y_{11},Y_{22}\|)+\|Y_{12}\|^4\big),
\end{aligned}
\end{equation}
\end{small}
and 
\begin{small}
\begin{equation}\label{TrC4C4}
{\rm Tr}(H_4^2)=2{\rm Tr}(Z_0^{-2}\hat{\Lambda}^2)+O(\|Y\|^2).
\end{equation}
\end{small}
The latter  follows from the  rough estimates 
\begin{small}
\begin{equation*}\label{C4iAsymptotic}
H_4^i=\begin{bmatrix}
z_0^{-i}\hat{\Lambda}^i+O(\|Y\|^2)   &  O(\|Y\|) 
\\     O(\|Y\|)&     z_0^{-i}\hat{\Lambda}^i+O(\|Y\|^2) 
\end{bmatrix},\quad i\geq 1,
\end{equation*}
\end{small}
which can be easily verified by induction.

Combining \eqref{Nh4YLambdaAsympto}, \eqref{TrC4} and \eqref{TrC4C4}, we obtain 
\begin{small}
\begin{align}\label{Nh4YLambdaAsymptofinalform}
Nh_4(Y)&=h_{4,0,0}-\sqrt{N}
h_{4,0}(Y_{12})+h_{4,1}(Y)
\end{align}
\end{small}
where
\begin{small}
\begin{equation*}\label{h4Lambda}
 \begin{aligned}
&h_{4,0,0}=2\sqrt{N}{\rm Tr}(Z_0^{-1}\hat{\Lambda})-{\rm Tr}(Z_0^{-2}\hat{\Lambda}^2)      \\
&h_{4,0}(Y_{12})=2z_0^{-1}{\rm Tr}(Y_{12}Y_{12}^*\hat{Z})+2\overline{z_0}^{-1}{\rm Tr}(Y_{12}^*Y_{12}\hat{W}^*),
\end{aligned}
\end{equation*}
\end{small}

and
\begin{small}
\begin{equation*}\label{Ph4Y}
h_{4,1}(Y)=O(N^{-\frac{1}{2}})+O(\|Y\|)+\sqrt{N}
\left( O\big(\|Y_{11},Y_{22}\|(\|Y_{12}\|+\|Y_{11},Y_{22}\|)+\|Y_{12}\|^4\big) \right).
\end{equation*}
\end{small}

The remaining task is to estimate  the factors of   $g_4(Y,\Lambda)$ in \eqref{gbetaYLambda}.
For $A_4$ given in  \eqref{Q4transform1}, noting the rescaling of $Z$  in \eqref{Symbol}, we have      %
\begin{equation*}\label{A4asymto4edge}
A_4=\begin{bmatrix}
0  &  Z_0  \\ -Z_0  &  0
\end{bmatrix}+O\big(N^{-\frac{1}{2}}+\|Y\|\big),
\end{equation*}
from  
\begin{equation*}\label{detA4ssquare4edge}
\left( \det(A_4) \right)^{-r}=1
+O\big(N^{-\frac{1}{2}}+\|Y\|\big).
\end{equation*}

Use Lemma  \ref{4lemma} and take  $\Xi=0$ there,  we obtain     
\begin{small}
\begin{align}\label{g4YLambda4edge}
g_4(Y,\Lambda)= &\Big(D^{(4)}+O\big(N^{-\frac{1}{4\kappa}}+\|Y\|\big)\Big)         \\
&\times \Bigg( \det\begin{bmatrix}
N^{-\frac{1}{4}}\Omega_{11} & Y_{12}\otimes (V_2)_{I_2} 
\\  
-Y_{12}^*\otimes \overline{(U_2)_{I_1}}   &  -N^{-\frac{1}{4}}\Omega_{11}^*
\end{bmatrix}
    +O\bigg( \sum_{(\alpha_{1}',\alpha_{2}',\alpha_{3}')} N^{-\alpha_{1}'}\| Y_{11},Y_{22} \|^{\alpha_{2}'} \| Y_{12} \|^{\alpha_{3}'} \bigg) \Bigg)
\end{align}
\end{small}
where  $\kappa=\max\{1/2,    p_{1,1}, \ldots,  p_{m,1}\},$
\begin{small}
\begin{align}\label{D-4edge}
D^{(4)}=(-1)^{K_1t}
\left| \overline{z}_0-z_0 \right|^{2K_1(r-r_0)}
\prod_{i=m+1}^{\gamma}\prod_{j=1}^{\alpha_i}
\left| \left( z_0-\theta_i \right)\left( z_0-\overline{\theta_i} \right) \right|^{2K_1p_{ij}\beta_{ij}}
\end{align}
\end{small}
and 
 the triple  in the error term satisfies 
\begin{equation*}\label{4CCSumcondi2}
4\alpha_{1}'+\alpha_{2}'+\alpha_{3}'\geq 2K_1t +\frac{1}{\kappa} \delta_{\alpha_{2}',0}.
\end{equation*}

Recalling   \eqref{Ideltabeta}, \eqref{f4YZ0decomposition} and \eqref{Nh4YLambdaAsymptofinalform},  we rewrite  
\begin{equation}\label{Idelta4decomposition4edge}
  e^{-h_{4,0,0}}I_{\delta,N}^{(4)}= J_{1,N}^{(4)}+J_{2,N}^{(4)}
\end{equation}
 where 

\begin{small}
\begin{equation*}\label{Idelta4decomposition4edge1}
 J_{1,N}^{(4)}=
\int\limits_{{\rm Tr}(YY^*)\leq \delta}\ g_{4}(Y,\Lambda)
e^{ -Nf_{4,0}(Y)-\sqrt{N}h_{4,0}(Y_{12}) } {\rm d}Y,
\end{equation*}
\end{small}
and 
\begin{small}
\begin{equation*}\label{Idelta4decomposition4edge2}
 J_{2,N}^{(4)}=
\int\limits_{{\rm Tr}(YY^*)\leq \delta}\ g_{4}(Y,\Lambda)
e^{ -Nf_{4,0}(Y)-\sqrt{N}h_{4,0}(Y_{12}) } \left(
e^{-Nf_{4,1}(Y)+h_{4,1}(Y)}-1\right) {\rm d}Y.
\end{equation*}
\end{small}
As in the proof of Theorem \ref{2-complex}, we  can get  an estimate
 \begin{equation} \label{4J2-1} J_{2,N}^{(4)}=O(N^{-\frac{1}{4}})J_{1,N}^{(4)}.  
 \end{equation}

For $J_{1,N}$,  after using   the  change of variables 
\begin{equation*} \label{4Yscaling}Y_{11}\to N^{-\frac{1}{2}}Y_{11}, \   Y_{22}\to N^{-\frac{1}{2}}Y_{22}, \ Y_{12}\to N^{-\frac{1}{4}}Y_{12},
\end{equation*}
 we can remove all restrictions on the integration domains of  $Y_{11}, Y_{22}$ and $Y_{12}$  up to some decaying  exponential term. Also noting  \eqref{g4YLambda4edge}, 
 the standard asymptotic analysis shows that 
\begin{equation}\label{I1delta4minusI2delta4edge}
 J_{1,N}^{(4)}=D^{(4)}  N^{-(K_1+1)K_1-\frac{1}{2}(K_1^2+K_1t)} \left( J_{\infty}^{(4)}+O\big( N^{-\frac{1}{4\kappa}} \big) \right),
\end{equation}
where
\begin{small}
\begin{equation*}\label{I0delta14edge}
J_{\infty}^{(4)}=\int\limits_{Y=Y^t}
   \det\begin{bmatrix}
\Omega_{11} & Y_{12}\otimes \left(V_2\right)_{I_2} \\ -Y_{12}^*\otimes \overline{\left(U_2\right)_{I_1}} & -\Omega_{11}^*
\end{bmatrix}
\exp\left\{ -f_{4,0}(Y)-h_{4,0}(Y_{12}) \right\} {\rm d}Y.
\end{equation*}
\end{small}
Integrate out $Y_{11}$ and $Y_{22}$,   we see from  $J_{\infty}^{(4)}$  that 
\begin{small}
\begin{equation*}\label{J4edge}
J_{\infty}^{(4)}=2^{-(K_1-1)K_1}\left( \frac{\pi}{\left| z_0-\overline{z}_0 \right|} \right)^{(K_1+1)K_1}
 I^{(4)}\left(  
( -\hat{\theta}_1)^{p_{1,1}},\cdots,( -\hat{\theta}_m )^{p_{m,1}};
2z_0^{-1}\hat{Z},2z_0^{-1}\hat{W}\right).
\end{equation*}
\end{small}

At last,   combine   Proposition \ref{Analysisform},   
   \eqref{Idelta4decomposition4edge},   \eqref{4J2-1} and \eqref{I1delta4minusI2delta4edge},   we  complete the proof  at the complex edge.  
\end{proof}

\begin{proof}[Proof of Theorem \ref{4-complex} (i)]
Repeating almost the same procedures   as  in the proof of  Theorem \ref{4-complex} (i)  and applying  Lemma \ref{4lemma},  the proof is straightforward.  We just point out some key differences. 

When $z_0=\pm 1$,  we  see from    \eqref{fbetaYZ0} that 
\begin{equation*}\label{f4YZ0edgerealcase}
f_4(Y,Z_0)=\frac{1}{2}{\rm Tr}\left( (YY^*)^2 \right)+O(\|Y\|^6).
\end{equation*}
Recalling  \eqref{Cbetadetailed} simple  calculation shows        
\begin{equation}\label{TrC4realedgecase}
{\rm Tr}(H_4)=2z_0^{-1}{\rm Tr}\left( \hat{\Lambda} \right)-2z_0^{-1}{\rm Tr}\left( Y^*Y\hat{\Lambda} \right)+O(\|Y\|^4),
\end{equation}
and
\begin{equation}\label{TrC4C4realedge}
{\rm Tr}(H_4^2)=2{\rm Tr}\left( \hat{\Lambda}^2 \right)+O(\|Y\|^2).
\end{equation}

Combination of   \eqref{hbetaYLambda},  \eqref{TrC4realedgecase} and \eqref{TrC4C4realedge} yields 
\begin{small}
\begin{equation*}\label{Nh4YLambdaAsymptofinalformrealedge}
Nh_4(Y,\Lambda)=2z_0^{-1}\sqrt{N}{\rm Tr}( \hat{\Lambda} )-{\rm Tr}(\hat{\Lambda}^2)-2\sqrt{N}z_0^{-1}{\rm Tr}\big( Y^*Y\hat{\Lambda} \big)+h_{4,1}(Y),
\end{equation*}
\end{small}
where
\begin{equation*}
h_{4,1}(Y)=O(\sqrt{N}\|Y\|^4)+O(\|Y\|^2)+O(N^{-\frac{1}{2}}).
\end{equation*}

After similar  calculation as in the proof of     Theorem \ref{4-complex} (ii)  and application of Lemma \ref{4lemma} where  $\Xi=0$, we can obtain 
   \begin{small}
\begin{multline}
g_4(Y,\Lambda)= \Big(D^{(4,{\rm re})}+O\big(N^{-\frac{1}{4\kappa}}+\|Y\|\big)\Big)   \nonumber     \\
\times
\Bigg(
{\rm Pf}\begin{bmatrix}
R( Y;(V_1)_{I_2},(U_1)_{I_1} ) & N^{-\frac{1}{4}}S\left( N^{\frac{1}{4}}Y \right)  
\\  
-N^{-\frac{1}{4}}S\left( N^{\frac{1}{4}}Y \right)^t  &  R( Y^*;(U_1)_{I_1},(V_1)_{I_2} )
\end{bmatrix}
+O\bigg( \sum_{(\alpha_{1}',\alpha_{2}')} N^{-\alpha_{1}'}\| Y \|^{\alpha_{2}'} \bigg)
\Bigg),
\end{multline}
\end{small}
where 
\begin{small}
\begin{align}\label{D-4Redge}
D^{(4,{\rm re})}=\prod_{i=m+1}^{\gamma}\prod_{j=1}^{\alpha_i}\left| z_0-\theta_i \right|^{4K_1p_{ij}\beta_{ij}}.
\end{align}
\end{small}
 
Finally,    
use     the  change of variables $Y\to N^{-\frac{1}{4}}Y$ and  we can complete the proof.  
\end{proof}


\section{Proofs of main results}  \label{proofs}
 \subsection{ Proof of Proposition \ref{intrep}}

\begin{proof}[Proof of Proposition \ref{intrep}]  {\bf Step 1.} To verify   \eqref{densityeigen0},  we  use  Dyson’s approach  that any matrix can be  triangularized by a unitary matrix; see e.g. \cite[Chapter 6]{HKPV}.

 First, 
by Schur decomposition  
\begin{equation}\label{Schurdecom}
X=V^*\left(\Lambda+T\right)V,
\end{equation}
where $V=[v_{ij}]$ is unitary, $T$ is strictly upper triangular and 
$\Lambda=\left( z_1,\cdots,z_{N} \right)$ 
is diagonal. The decomposition is  unique in the following sense: 
  using  the lexicographic order on complex numbers to arrange the eigenvalues in increasing order  requiring that all diagonal entries $v_{ii} \geq 0$
   from $V$ and omitting  a lower dimensional set.  The Jacobian determinant formula can thus  be  obtained as follows 
\begin{equation}\label{SchurJacobian}
{\rm d}X=    \prod_{1\leq i<j\leq N}\left| z_i-z_j \right|^2  {\rm d}\Lambda
{\rm d}T{\rm d}V,
\end{equation}
see  \cite[Section 6.3]{HKPV} for more specific descriptions.  

 Put 
\begin{equation*}\label{Bdenotion}
B=[ b_{ij}]:=VX_0V^*,\quad X_0={\rm diag}\left( A_0,0_{N-r} \right), \quad V=(V_1,V_2),
\end{equation*}
where $V_1$ and $V_2$  are  of size $N\times r$ and   $N\times (N- r)$ respectively, then 
$$B=V_{1}A_{0}V_{1}^*.$$
By \eqref{Schurdecom}  we  arrive at     
\begin{equation*}\label{traceexpension}
\begin{aligned}
{\rm Tr}\left( X-X_0 \right)\left( X-X_0 \right)^*&={\rm Tr}\left( \Lambda\Lambda^* \right)-{\rm Tr}\left( \Lambda B^*+B\Lambda^* \right)+{\rm Tr}\left( (T-B)(T-B)^* \right).
\end{aligned}
\end{equation*}
Integrate out  $V_2$ and the strictly upper triangular 
matrix $T$ which is  indeed a Gaussian matrix  integral,  we get 

 \begin{equation}\label{densityeigen01}
\begin{aligned}
&f_{N}\left( A_0; z\right)=
\left( \frac{N}{\pi} \right)^{N^2}
 \Big(\frac{\pi}{N}\Big)^{\frac{1}{2}N(N-1)} C_{N,N-r}
\exp\bigg\{  
-N\sum_{k=1}^{N}\left| z_k \right|^2
\bigg\} \prod_{1\leq i<j\leq N}\left| z_i-z_j \right|^2  \\
&\times \int \delta\left( V_1^*V_1-\mathbb{I}_r \right)
\exp\bigg\{ 
N  \sum_{i=1}^N\left( z_i\overline{b_{ii}}+\overline{z_i}b_{ii} \right)
+N\sum_{N\geq i\geq j\geq 1}\left| b_{ij} \right|^2
\bigg\}{\rm d}V_1,
\end{aligned}
\end{equation}
from  which \eqref{densityeigen0} follows. Here $C_{N,N-r}$ denotes the normalization from the integral over $V_2$. To compute $Z_{N,r}$, note that it is independent from $A_0$. Hence, we can take $A_0=0$, reducing the integral in (\ref{densityeigen0}) to $\int \delta\left( V_1^*V_1-\mathbb{I}_r \right){\rm d}V_1$, which can be calculated through the formula\cite[Section 4.1]{ZL}:
\begin{equation*}
\int \delta\left( Y^*Y-\mathbb{I}_m \right){\rm d}Y=
\frac{\pi^{\frac{1}{2}m(2n-m+1)}}{\prod_{j=1}^m(n-j)!},
\end{equation*}
where $Y$ is any complex matrix of size $n\times m$, $m\leq n$. Combining the normalization condition
\begin{equation*}
\int \cdots \int f_{N}\left( 0; z_1,\cdots,z_{N}\right)\prod_{i=1}^N{\rm d}\Re z_i{\rm d}\Im z_i=1
\end{equation*}
and the integration formula
\begin{equation*}
\int \exp\left\{
-N\sum_{i=1}^N|z_i|^2
\right\}\prod_{1\leq i <j \leq N}|z_i-z_j|^2
\prod_{i=1}^N{\rm d}\Re z_i{\rm d}\Im z_i
=N^{-\frac{N(N+1)}{2}}\pi^N
\prod_{j=1}^N j!,
\end{equation*}
we arrive at (\ref{correZNr}).

{\bf Step 2. }  The goal is to obtain an  integral representation for correlation functions, that is, 
\begin{equation*}\label{correlationdef}
R_{N}^{(n)}\left(A_0;  z_1,\cdots,z_{n}\right)=\frac{N!}{(N-n)!}
\int \cdots \int f_{N}\left(A_0; z_1,\cdots,z_{N}\right){\rm d}z_{n+1}\cdots {\rm d}z_{N},
\end{equation*} 
where  ${\rm d}z={\rm d}\Re z  {\rm d}\Im z
$ for a complex variable $z$. 
 
Rewrite   
\begin{equation*}\label{Gdecompose}
V_1=\begin{bmatrix}
Q \\ V_{21}
\end{bmatrix}, \quad 
B=\begin{bmatrix}
QA_{0}Q^* & QA_{0}V_{21}^* \\ V_{21}A_{0}Q^* & V_{21}A_{0}V_{21}^*
\end{bmatrix},\quad Q=[q_{ij}],
\end{equation*}
where $Q$ is  of size $n\times r$.   
%
%
Next, let's split the summation $\sum_{i\geq j}$ in  \eqref{densityeigen0} into two parts, one only depends on the bottom-right block   $V_{21}A_{0}V_{21}^*$, the other can be reduced by the orthogonality  to 
 \begin{equation*}\label{secondpart}
\begin{aligned}
\sum_{j=1}^n\sum_{i=j}^{N}\left| b_{ij} \right|^2&=
\sum_{j=1}^n\sum_{l_1,l_2,k_1,k_2=1}^r \overline{v_{jl_1}}v_{jl_2} a_
{k_1l_1}\overline{a_{k_2l_2}} \sum_{i=j}^{N} v_{ik_1}\overline{v_{ik_2}}
 \\
&= \sum_{j=1}^n\sum_{l_1,l_2,k_1,k_2=1}^r \overline{v_{jl_1}}v_{jl_2} a_
{k_1l_1}\overline{a_{k_2l_2}} 
\Big( \delta_{k_1,k_2}-\sum_{i=1}^{j-1} v_{ik_1}\overline{v_{ik_2}} \Big).
\end{aligned}
\end{equation*}
On the other hand, noting that 
 \begin{equation*}\label{trans00}
 V_1^*V_1-\mathbb{I}_r =V_{21}^*V_{21}
 -(\mathbb{I}_r -Q^* Q), 
\end{equation*}
after a change of variables  $V_{21}=\widetilde{V}_1\sqrt{\mathbb{I}_r -Q^* Q}$  
 the integral   
in \eqref{densityeigen0}  can be rewritten as 
\begin{multline}\label{trans11}
\int \int {\rm d}V_{1}\delta\left( V_1^*V_1-\mathbb{I}_r \right)\Big(\cdot\Big)=\\
\int_{Q^*Q\leq \mathbb{I}_r}{\rm d}Q\left( \det\left( \mathbb{I}_r -Q^* Q \right) \right)^{N-n-r}
\int {\rm d}\widetilde{V}_1
\delta\left( \widetilde{V}_1^*\widetilde{V}_1-\mathbb{I}_r \right)\Big(\cdot\Big).
\end{multline}

Introduce \begin{small}
\begin{equation}\label{regularization}
\widetilde{A}_0= \sqrt{(N-n)/N}  \sqrt{\mathbb{I}_r -Q^* Q}A_{0} \sqrt{\mathbb{I}_r -Q^* Q}, 
z_{k+n}=\sqrt{(N-n)/N}\tilde{z}_k,\ k=1,\cdots,N-n, 
\end{equation}
\end{small}
treat $\{\tilde{z}_k\}$ as eigenvalues of a Ginibre $(N-n) \times (N-n)$ matrix  $\widetilde{X}$  perturbed by  an $r\times r$ matrix $\widetilde{A}_0$,  we see from \eqref{densityeigen0} and    \eqref{trans11}   that the correlation functions can be expressed in terms of  averaged products of characteristic polynomials. 

This  thus completes the proof.
\end{proof}

\subsection{Proof of Theorem \ref{2-complex-correlation}}

To verify  Theorem \ref{2-complex-correlation}, we will use the following lemma.
\begin{lemma}\label{maximumlemma}  Let $A$ and $Q$ be  $r\times r$  and  $n\times r$ complex matrices, respectively.   
Given $z_0 \in \mathbb{C}$ such that  $|z_0|\leq 1$, put
\begin{small}
\begin{equation}\label{JQdefinition}
J_n(Q;A)={\rm Tr}\big( Q(\overline{z}_0 A+z_0A^*-A^*A)Q^* \big)
+\log\det\left( \mathbb{I}_n-QQ^* \right) +\sum_{1\leq i<j\leq n}\big| (QAQ^*)_{i,j} \big|^2.
\end{equation}
\end{small}
Then $J_n(Q;A)\leq 0$ whenever   $QQ^*<\mathbb{I}_n$, and   equality holds if and only if $Q=0$.

\end{lemma}
\begin{proof}
We proceed  by induction on $n$. 
 For  $n=1$,  note that  $Q$  is a row vector, by the inequality $\log(1-x)\leq -x-x^2/2$ we arrive at \begin{small}
\begin{equation*}
\begin{aligned}
J_1(Q;A)&= \overline{z}_0QAQ^*+z_0QA^*Q^* 
-QA^*AQ^*
+\log\left( 1-QQ^* \right) \\
&\leq
 \overline{z}_0QAQ^*+z_0QA^*Q^* 
-QA^*AQ^*-QQ^*-\frac{1}{2}\left( QQ^* \right)^2  \\
&=-Q\left( z_0\mathbb{I}_r-A \right)^*\left( z_0\mathbb{I}_r-A \right)Q^*
-(1-|z_0|^2)QQ^*-\frac{1}{2}\left( QQ^* \right)^2\\
&
\leq 0.
\end{aligned}
\end{equation*}
\end{small}
It's easy to see that   $J_1(Q;A)=0$ if and only if $Q=0$.

Suppose that the lemma is true for  $1, \ldots, n-1$. 
 Rewrite $Q$  as  a $2 \times 1$  block matrix  
\begin{equation*}\label{Qdecompositionlemma}
Q=\begin{bmatrix}
Q_1 \\ Q_2
\end{bmatrix},\quad
Q_1\in \mathbb{C}^{(n-1)\times r},\quad
Q_2\in \mathbb{C}^{1\times r}, 
\end{equation*}
then  
\begin{equation*}
Q^{*}Q=Q_1^*Q_1+Q_2^*Q_2, \quad QAQ^*=\begin{bmatrix}
Q_1AQ_1^* & Q_1AQ_2^* \\
Q_2AQ_1^* & Q_2AQ_2^*
\end{bmatrix}.
\end{equation*}
Since $QQ^*<\mathbb{I}_n$,  we see that $Q_1Q_1^*<\mathbb{I}_{n-1}$, from which  use of  Schur complement technique  gives 
\begin{small}
\begin{equation*}\label{inductionkey3}
\log\det\left( \mathbb{I}_n-QQ^* \right)=
\log\det\left( \mathbb{I}_{n-1}-Q_1Q_1^* \right)+
\log\left( 1-Q_2FQ_2^* \right),
\end{equation*}
\end{small}
where
\begin{equation*}\label{Flemmadenotion}
F:=\mathbb{I}_r+Q_1^*\left( \mathbb{I}_{n-1}-Q_1Q_1^* \right)^{-1}Q_1.
\end{equation*}
Hence,  we can divide $J_{n}(Q;A)$ into two parts 
\begin{equation*}\label{inductionstep}
J_{n}(Q;A)=J_{n-1}(Q_1;A)+\widetilde{J}(Q),
\end{equation*}
where
\begin{small}
\begin{equation*}\label{widetildeJQ}
\widetilde{J}(Q)=Q_2(\overline{z}_0 A +z_0A^*)Q_2^* 
-Q_2A^*\left( \mathbb{I}_{r}-Q_1^*Q_1 \right)AQ_2^*+\log\left( 1-Q_2FQ_2^* \right).
\end{equation*}
\end{small}

We claim that 
\begin{equation} \label{claimJ1}
\widetilde{J}(Q)=J_{1}(Q_{2}F^{\frac{1}{2}};   F^{-\frac{1}{2}}AF^{-\frac{1}{2}}).
\end{equation}
If so, then  the desired result immediately from   
the induction hypothesis.

To prove \eqref{claimJ1}, let 
\begin{equation*}
\widetilde{Q}_2=Q_2F^{\frac{1}{2}},\quad
\widetilde{A}=F^{-\frac{1}{2}}AF^{-\frac{1}{2}},
\end{equation*} 
then
\begin{small}
\begin{equation*}\label{widetildeJQchange}
\widetilde{J}(Q)=\widetilde{Q}_2(\overline{z}_0\widetilde{A}
+z_0\widetilde{A}^*)\widetilde{Q}_2^* 
-\widetilde{Q}_2\widetilde{A}^*F^{\frac{1}{2}}\left( \mathbb{I}_{r}-Q_1^*Q_1 \right)F^{\frac{1}{2}}\widetilde{A}\widetilde{Q}_2^*
+\log\big( 1-\widetilde{Q}_2\widetilde{Q}_2^* \big).
\end{equation*}
\end{small}
It's easy to see from the simple fact $$ Q_1\left( \mathbb{I}_{r}-Q_1^*Q_1 \right)=
\left( \mathbb{I}_{n-1}-Q_1Q_1^* \right)Q_1
 $$
that 
$$
F^{\frac{1}{2}}\left( \mathbb{I}_{r}-Q_1^*Q_1 \right)F^{\frac{1}{2}}=\mathbb{I}_{r},
$$
from which  \eqref{claimJ1}  immediately follows.

\end{proof}

\begin{proof}[Proof of Theorem \ref{2-complex-correlation}]
  In \eqref{algebraequa} of Proposition \ref{intrep},  
 let $A_0=PJ_2 P^{-1}$, set  
   $U=P^*P$. 
Make a change of variables $ Q\to QP^*$,  and then  by Lemma \ref{maximumlemma}  we can proceed similarly  as   in Proposition \ref{Analysisform}     and  further prove that   for  sufficiently small $\delta>0$ there exists $\widetilde{\Delta}^{(2)}>0$ such that
\begin{equation}\label{correlationbeta2expansion}
R_N^{(n)}\left( A_0;z_1,\cdots,z_n \right)=
\widetilde{C}_N^{(2)}(z_1,\ldots,z_n) \left(
\widetilde{I}_{\delta,N}^{(2)}+O\big(e^{-N\widetilde{\Delta}^{(2)}}\big)\right)
\end{equation}
where $z_k=z_0+ \hat{z}_{k}/\sqrt{N}$, $k=1,2,\ldots,n$,
\begin{small}
\begin{equation*}\label{tildeCN2algebra}
\widetilde{C}_N^{(2)}(z_1,\ldots,z_n)=
\frac{1}{C_N}
\prod_{1\leq i<j\leq n}\big| z_i-z_j \big|^2
e^{-N\sum_{k=1}^n|z_k|^2}
\left( \det U \right)^n,
\end{equation*}
\end{small}
and
\begin{small}
\begin{equation}\label{INdeltac}
\widetilde{I}_{\delta,N}^{(2)}=\int_{{\rm Tr}Q^*Q\leq \delta} \widetilde{g}_2( Q)
\exp\!\big\{
-N \widetilde{f}_2(Q)+\sqrt{N}\widetilde{h}_2( Q)
\big\} {\rm d}Q.
\end{equation}
\end{small}
Here 
\begin{small}
\begin{equation*}\label{gcQ}
\widetilde{g}_2( Q)=\big( 
\det\left(
\mathbb{I}_n-QUQ^*
\big)
\right)^{-n-r}
{ {\mathbb{E}}_{\mathrm{GinUE}_{N-n}}( \widetilde{A}_0)}
\bigg[
\prod_{i=1}^n\left|
\det\left(
\sqrt{\frac{N}{N-n}}z_i-\widetilde{X}_{N-n}
\right)
\right|^2
\bigg],
\end{equation*}
\end{small}
in which
\begin{equation*}\label{tildeA0integrationchange}
\widetilde{A}_0=\sqrt{N/(N-n)}\sqrt{\mathbb{I}_r-PQ^*QP^*}A_0\sqrt{\mathbb{I}_r-PQ^*QP^*},
\end{equation*}
\begin{small}
\begin{equation*}\label{algebrafcQ}
\begin{aligned}
\widetilde{f}_2(Q)=&-{\rm Tr}\left(
 Q\big( \overline{z}_{0}UJ_{2}+z_0J_{2}^*U\big)Q^* 
 \right)
+{\rm Tr}\left( QJ_{2}^*UJ_{2}Q^* \right)    -\log\det\left(
\mathbb{I}_n-QUQ^*
\right)
\\
&-\sum_{1\leq i<j\leq n}\big| (QUJ_2Q^*)_{i,j} \big|^2    
\end{aligned}
\end{equation*}
\end{small} 
and 
\begin{small}
\begin{equation}\label{algebrahcQ}
\widetilde{h}_2(Q)=
{\rm Tr}\big( \hat{Z}^* QUJ_2Q^*+\hat{Z}Q J_{2}^*UQ^* \big).
\end{equation}
\end{small}
Moreover,  by the Taylor expansion for the log determinant,  we have 
 \begin{small}
\begin{equation}\label{algebrafcQ2}
\begin{aligned}
\widetilde{f}_2(Q) 
=&{\rm Tr}\big( Q\left( z_0\mathbb{I}_r-J_2 \right)^*U\left( z_0\mathbb{I}_r-J_2\right)Q^* \big)+\frac{1}{2}{\rm Tr}( QUQ^* )^2\\
&
- \sum_{1\leq i<j\leq n}\big| (QUJ_2Q^*)_{i,j} \big|^2
+O\left( \| Q \|^6 \right).
\end{aligned}
\end{equation}
\end{small}

 On the other hand, 
 by 
  the Stirling formula 
$$ N!=\sqrt{2\pi N}\Big( \frac{N}{e} \Big)^N\Big( 1+O\big( N^{-1} \big) \Big), $$
 it is easy to see from   \eqref{correZNr} that 
\begin{multline}\label{tildeCNcexpan}
\widetilde{C}_N^{(2)}(z_1,\ldots,z_n)=\Big( 1+O\Big( N^{-1} \Big) \Big)
N^{n\left(r+1-\frac{1}{2}n\right)}e^{-\sqrt{N}{\rm Tr}\big(z_0^{-1}\hat{Z}+\overline{z_{0}^{-1}\hat{Z}}\big)}
\\ \times
e^{-n^2}(2\pi)^{-\frac{n}{2}}\pi^{-n-nr}\left( \det(U) \right)^n
e^{-\sum_{i=1}^n|\hat{z}_i|^2}\prod_{1\leq i<j \leq n}|\hat{z}_i-\hat{z}_j|^2,
\end{multline}
which holds 
uniformly for   all 
$\hat{z}_{k}$
in a compact subset of $\mathbb{C}$.

The remaining task is to   analysize $\widetilde{I}_{\delta,N}^{(2)}$. For this, we proceed in three steps. 

{\bf Step 1:  Taylor expansions  of $ \widetilde{f}_2(Q)$  and $\widetilde{h}_2(Q)$.}

Noting   $J_2$ in \eqref{J2detail} and   the assumption,  consider \eqref{beta2edgethetaexpression}
where $\hat{\theta}_i=0$, $i=1,\cdots,m$.  Combine \eqref{J2decom}-\eqref{J2adeno} and   use the notation therein, we get $J_2=J_2^{(0)}$.  Next, we  are devoted to  establishing  leading terms in \eqref{algebrahcQ} and  \eqref{algebrafcQ2}.  

According to   the Jordan block structure of  $J_2^{(0)}$,   both $U$ and  
\begin{equation*}\label{tildeUdenotion}
\widetilde{U}:=\big( z_0\mathbb{I}_r-J_2^{(0)} \big)^*U
\big( z_0\mathbb{I}_r-J_2^{(0)} \big)
\end{equation*}
admit  a three-layer  structure  

\begin{equation*}\label{Ublockform1}
U=[ U_{k;l} ]_{k,l=1}^{m+1},  \quad \widetilde{U}=[ \widetilde{U}_{k;l} ]_{k,l=1}^{m+1}
\end{equation*}
where  for the former  $U_{m+1;m+1}$  is of size $r_0\times r_0$  and for $k,l=1,\cdots,m$ 
\begin{equation*}\label{Ublockform23}
\begin{aligned}
&U_{k;l}=\left[ U_{k,a;l,b} \right]_{\alpha_{k}\times \alpha_{l}},\quad
U_{k,a;l,b}=\left[ U_{k,a,c;l,b,d} \right]_{\beta_{k,a}\times \beta_{l,b}}, \\
&
U_{m+1;l}=\left[ U_{m+1;l,b} \right]_{1\times  \alpha_{l}}, 
 \quad
U_{m+1;l,b}=\left[ U_{m+1;l,b,d} \right]_{1\times \beta_{l,b}},\\
&U_{k;m+1}=\left[ U_{k,a;m+1} \right]_{\alpha_{k} \times  1},\quad
U_{k,a;m+1}=\left[ U_{k,a,c;m+1} \right]_{\beta_{k,a}\times 1},
\end{aligned}
\end{equation*}
with   $U_{k,a,c;l,b,d}$, $U_{m+1;l,b,d}$ and $U_{k,a,c;m+1}$ being  of sizes $p_{k,a}\times p_{l,b}$, $r_0\times p_{l,b}$ and $p_{k,a}\times r_0$ respectively.    For the latter, the notations are  similar. 
It's easy to find  the following identities 
\begin{equation*}\label{UtildeUrelation}
\begin{aligned}
&\widetilde{U}_{k,a,c;l,b,d}=R_{p_{k,a}}(0)^*U_{k,a,c;l,b,d}R_{p_{l,b}}(0), \\
&\widetilde{U}_{m+1;l,b,d}=-\big( z_0\mathbb{I}_{r_0}-J_2^{(2)} \big)^*U_{m+1;l,b,d}R_{p_{l,b}}(0);  \\
&\widetilde{U}_{k,a,c;m+1}=-R_{p_{k,a}}(0)^*U_{k,a,c;m+1}\big( z_0\mathbb{I}_{r_0}-J_2^{(2)} \big);   \\
&\widetilde{U}_{m+1;m+1}=\big( z_0\mathbb{I}_{r_0}-J_2^{(2)} \big)^*
U_{m+1;m+1}\big( z_0\mathbb{I}_{r_0}-J_2^{(2)} \big).
\end{aligned}
\end{equation*}
Furthermore,  by the structure of  $R_{p_{k,a}}(0)$,    for $k,l=1,\cdots,m$ set 
\begin{small}
\begin{equation}\label{Ublockform4}
\begin{aligned}
&U_{k,a,c;l,b,d}=\begin{bmatrix}
U_{kac;lbd} ^{(11)}& U_{kac;lbd} ^{(12)} \\
U_{kac;lbd} ^{(21)}& U_{kac;lbd} ^{(22)}
\end{bmatrix},        \quad
U_{k,a,c;m+1}=\begin{bmatrix}
U_{kac;m+1}^{(U)} \\ U_{kac;m+1}^{(D)}
\end{bmatrix},\\
&U_{m+1;l,b,d}=\begin{bmatrix}
U_{m+1;lbd}^{(L)} & U_{m+1;lbd}^{(R)}
\end{bmatrix},\quad
\end{aligned}
\end{equation}
\end{small}
where $U_{kac;lbd} ^{(11)}$, $U_{m+1;lbd}^{(L)}$ and $U_{kac;m+1}^{(U)}$ are of sizes $(p_{k,a}-1)\times(p_{l,b}-1)$, ${r_0}\times(p_{l,b}-1)$ and $(p_{k,a}-1)\times {r_0}$, respectively.
Then simple   calculations show that 
\begin{small}
\begin{equation}\label{UtildeUrelation1}
\begin{aligned}
&\widetilde{U}_{k,a,c;l,b,d}=\begin{bmatrix}
0 & 0 \\
0 & U_{kac;lbd}^{(11)}
\end{bmatrix},   \quad  \widetilde{U}_{k,a,c;m+1}=\begin{bmatrix}
0 \\ -U_{kac;m+1}^{(U)}\big( z_0\mathbb{I}_{r_0}-J_2^{(2)} \big)
\end{bmatrix},   \\
&\widetilde{U}_{m+1;l,b,d}=\begin{bmatrix}
0 & -\big( z_0\mathbb{I}_{r_0}-J_2^{(2)} \big)^*U_{m+1;lbd}^{(L)}
\end{bmatrix}, \quad  \widetilde{U}_{m+1;m+1}=U_{m+1;m+1}.
\end{aligned}
\end{equation}
\end{small}

As previously introduced, 
put 
\begin{equation*}\label{Qblock1}
Q=\begin{bmatrix}
Q_1 & \cdots & Q_m & Q_{m+1}
\end{bmatrix},
\end{equation*}
where $Q_{m+1}$ is of size $n\times r_0$ and for $l=1,\cdots,m$ \begin{small}
\begin{equation*}\label{Qblock2}
Q_l=[
Q_{l,b}]_{1\times  \alpha_{l}},\quad
Q_{l,b}=\big[
Q_{l,b,d}\big]_{1\times  \beta_{l,b}},\quad
Q_{l,b,d}=\begin{bmatrix}
Q_{lbd}^{(L)} & Q_{lbd}^{(R)} 
\end{bmatrix},
\end{equation*}
\end{small}
where $Q_{l,b,d}$ and $Q_{lbd}^{(L)} $ are of sizes $n\times p_{l,b}$ and $n\times 1$, respectively.
Now   we can draw some blocks from $Q$ and  restructure it into  two new  matrices 
\begin{small}
\begin{equation*}\label{QUcompression1}
Q_R=\begin{bmatrix}
Q_1^{(L)} & \cdots & Q_m^{(R)} & Q_{m+1} 
\end{bmatrix},\quad   
Q_L=\begin{bmatrix}
Q_1^{(L)} & \cdots & Q_{m}^{(L)}
\end{bmatrix},
\end{equation*}
\end{small}
where 
\begin{small}
\begin{equation}\label{QUcompression2}
 Q_{l}^{(L)}=[Q_{l,b}^{(L)}],  Q_{l,b}^{(L)}=[Q_{lbd}^{(L)}], \quad 
  Q_{l}^{(R)}=[Q_{l,b}^{(R)}],  Q_{l,b}^{(R)}=[Q_{lbd}^{(R)}].
 \end{equation}
\end{small}

Similarly,  with \eqref{UtildeUrelation1} in mind, define 
   \begin{equation} \label{Ut22}\widetilde{U}^{(22)}=\big[ \widetilde{U}_{kl}^{(22)} \big]_{k,l=1}^{m+1},\end{equation}  where for $k,l=1,\cdots,m$\begin{small}
\begin{equation*}\label{QUcompressiondetail1}
\begin{aligned}
&\widetilde{U}_{kl}^{(22)}=\begin{bmatrix}
U_{kac;lbd}^{(11)}
\end{bmatrix}, \quad
\widetilde{U}_{m+1,l}^{(22)}=\begin{bmatrix}
-\left( z_0\mathbb{I}_{r_0}-J_2^{(2)} \right)^*U_{m+1;lbd}^{(L)}
\end{bmatrix}, \\
&\widetilde{U}_{k,m+1}^{(22)}=\begin{bmatrix}
-U_{kac;m+1}^{(U)}\left( z_0\mathbb{I}_s-J_2^{(2)} \right)
\end{bmatrix},\quad
\widetilde{U}_{m+1,m+1}^{(22)}=\widetilde{U}_{m+1,m+1}.
\end{aligned}
\end{equation*}
\end{small}
Also,  in \eqref{Ublockform4},  choose the   (1,1)-entry of   $U_{kac;lbd} ^{(11)}$ and construct  a matrix $\hat{U}^{(1,1)}$ as follows 
\begin{small}
\begin{equation} \label{U11}
\hat{U}^{(1,1)}=[ \hat{U}_{k,l}],\quad
\hat{U}_{k,l}=\left[  \big(U_{kac;lbd} ^{(11)}\big)_{1,1}
\right]_{\alpha_{k}\times \alpha_l}.
\end{equation}
\end{small}
It's worth emphasizing at this point that $\hat{U}^{(1,1)}=U_{I_1}$, which is 
a  sub-matrix by deleting all other rows and columns except for those  in  $I_1$.  Here   $I_{1}$ is the  index set consisting of  numbers corresponding to   first columns of all 
 blocks  
 $R_{p_{i,j}}\left(\theta_i \right)$    from $J_2$  in  \eqref{J2detail}  where  $ 1\leq i \leq m, 1\leq j \leq \alpha_i$.  

We are ready to do some calculation. 
First,  it's easy to see  that
\begin{equation}\label{TrQtildeQ}
{\rm Tr}\big( Q\left( z_0\mathbb{I}_r-J_2 \right)^*U\left( z_0\mathbb{I}_r-J_2\right)Q^* \big)={\rm Tr}( Q\widetilde{U}Q^*)=
{\rm Tr}\big( Q_R\widetilde{U}^{(22)}Q_R^*\big).
\end{equation}
On the other hand, if we rewrite 
$J_2^{(0)}$ as
\begin{equation*}
J_2^{(0)}:=z_0\mathbb{I}_{r-r_0}\bigoplus 0_{r_0}+J_2^{(3)},
\end{equation*}
then   we arrive at 
\begin{equation*}\label{J2jerrorcontrol}
QUJ_2^{(3)}Q^*=O\left( \|Q_R\|\|Q\| \right)
\end{equation*}
and 
\begin{small}
\begin{equation*}\label{J2jmaincontrol}
QU\left( J_2^{(0)}-J_2^{(3)} \right)Q^*=z_0Q_L\hat{U}^{(1,1)}Q_L^*+
O\left( \|Q_R\|\|Q\| \right),
\end{equation*}
\end{small}
from which 
\begin{equation}\label{B11asymptoexpan}
QUJ_2^{(0)}Q^*=z_0Q_L \hat{U}^{(1,1)}Q_L^*+
O\left( \|Q_R\|\|Q\| \right).
\end{equation}
Noting that  combination of  \eqref{QUcompression2} and  \eqref{U11} gives   
 \begin{equation*}\label{QUQasymptoexpan}
QUQ^*=Q_L \hat{U}^{(1,1)}Q_L^*+
O\left( \|Q_R\|\|Q\| \right), 
\end{equation*}
we immediately see from $|z_0|=1$ and $( \hat{U}^{(1,1)})^{*}= \hat{U}^{(1,1)}$ that 
\begin{small}
\begin{equation}\label{fcQsecondexpan}
\frac{1}{2}{\rm Tr}\left( QUQ^* \right)^2
-\sum_{i<j}^n\big| \big(QUJ_2^{(0)}Q^*\big)_{i,j} \big|^2=
\frac{1}{2}\sum_{i=1}^n
\left( Q_L \hat{U}^{(1,1)}Q_L^* \right)_{i,i}^2
+O\left( \|Q_R\|\|Q\|^3 \right).
\end{equation}
\end{small}
Combining  \eqref{algebrafcQ2}, \eqref{TrQtildeQ} and \eqref{fcQsecondexpan}, we  obtain 
\begin{equation*}\label{fcQexpan}
\widetilde{f}_2(Q)={\rm Tr}\big( Q_R\widetilde{U}^{(22)}Q_R^*\big)+\frac{1}{2}\sum_{i=1}^n
\left( Q_L \hat{U}^{(1,1)}Q_L^* \right)_{i,i}^2+O\left( \|Q_R\|\|Q\|^3+\|Q\|^6 \right).
\end{equation*}
Meanwhile,  by   \eqref{B11asymptoexpan} we get 
\begin{equation*}\label{hcQexpan}
\widetilde{h}_2(Q)={\rm Tr}
\Big( \big( z_0\hat{Z}^*+\overline{z_0}\hat{Z} \big)
Q_L \hat{U}^{(1,1)} Q_L^* \Big)+O\left(\|Q_R\|\|Q\| \right).
\end{equation*}
 
{\bf Step 2: Modified estimate of 
$\widetilde{g}_2(Q)$}.  For the first
factor  in  \eqref{gcQ},  it's easy to see  
\begin{equation*}\label{detcoeexpan}
\left( 
\det\left(
\mathbb{I}_n-QUQ^*
\right)
\right)^{-n-r}=
1+O\left( \| Q \|^2 \right).
\end{equation*}
As for the second factor,   the averaged product of characteristic polynomials has actually been verified in  Theorem \ref{2-complex} (i), except that   some minor correction  will be needed because of $Q$ and the scaling factor $(N-n)/N$.   So we  repeat almost the same procedure quickly,  with emphasis on  the differences.    

Note that
\begin{small}
\begin{equation*}\label{gcQdiffexpan1}
\sqrt{\mathbb{I}_r-PQ^*QP^*}=\mathbb{I}_r+O\left( \| Q \|^2 \right),
\end{equation*}
\end{small} 
we have
\begin{small}
\begin{equation}\label{gcQdiffexpan2}
\widetilde{A}_0=A_0
+O\left( N^{-1}+\| Q \|^2 \right).
\end{equation}
\end{small}
Obviously,  
\begin{equation}\label{gcQdiffexpan3}
\sqrt{\frac{N}{N-n}} z_i=z_0+\frac{\hat{z}_i}{\sqrt{N}}+\frac{nz_0}{2N}+
O\big( N^{-\frac{3}{2}} \big).
\end{equation}

Using the  duality relation \eqref{Charac} in Corollary \ref{characteristic}, we have
\begin{equation*}\label{Characcorrelation}
\mathbb{E}\bigg(
\prod_{i=1}^n\Big|
\det\Big(
\sqrt{\frac{N}{N-n}}z_i-\widetilde{X}_{N-n}
\Big)
\Big|^2
\bigg)=\int_{\mathbb{C}^{n^2}} Q_{2}\big(A;\widetilde{X}_0,Y\big) \widehat{P}_{K,2}(1;Y,0){\rm d}Y,
\end{equation*}
where $Q_{2}\big(A;\widetilde{X}_0,Y\big)$, $\widehat{P}_{K,2}(1;Y,0)$ and  $A$ are defined in (\ref{Q2}), (\ref{dualdensity}), \eqref{A124} respectively,  but with $N\to N-n$,  
 $\widetilde{X}_0={\rm diag}( \widetilde{A}_0,0_{N-n-r})$ and 
$Z=\sqrt{N/(N-n)}{\rm diag}(z_1,\cdots,z_n)$.

Proceeding  as in the proof of  Proposition \ref{Analysisform},   we  can prove that  for sufficiently   small $\delta>0$ and  for $|z_0|=1$ 
\begin{small}
\begin{equation*}\label{correlationanalysisform}
F_{n}^{(2)}(Z,Z)=\frac{1}{\widehat{Z}_{K,2}}\Big(
I_{\delta,N}^{(2)}+O\big( e^{-\frac{1}{2}N\Delta^{(2)}} \big) \Big),
\end{equation*}
\end{small}
with
\begin{equation}\label{Idelta2correlation}
I_{\delta,N}^{(2)}=\int\limits_{{\rm Tr}(YY^*)\leq \delta}\ g_{2}(Y)
\exp\left\{ -Nf_{2}(Y,Z_0)+Nh_{2}(Y) \right\} {\rm d}Y,
\end{equation}
where
\begin{small}
\begin{equation}\label{g2YLambdacorrelation}
g_2( Y)=\det(
B_{11}^{(2)})\left( \det\left( A_2 \right) \right)^{-n-r},
\end{equation}
\end{small} 
\begin{small}
\begin{equation*}\label{fbetaYZ0correlation}
f_{2}(Y)=
\Big( 1-\frac{n}{N} \Big)
{\rm Tr}(YY^*)-\log\det\left( \mathbb{I}_{n}+YY^* \right).
\end{equation*}
\end{small}
Noting that we need to  replace $\hat{Z}$ with $\hat{Z}+\frac{1}{2}nz_0 N^{-\frac{1}{2}}\mathbb{I}_n+O\left( N^{-1} \right)$ both  in (\ref{Cbeta}) and (\ref{Cbetadetailed}), we then get
\begin{small}
\begin{equation*}\label{hbetaYLambdacorrelation}
h_{2}(Y)= 
\log\det\big( \mathbb{I}_{2n}+N^{-\frac{1}{2}}H_2 \big).
\end{equation*}
\end{small}
Moreover,  we can obtain 
\begin{equation*}
f_{2}(Y)= \frac{1}{2}{\rm Tr}\left( (YY^*)^2 \right)+O\left(
\| Y \|^6+N^{-1}\| Y \|^2
\right), 
\end{equation*}
and 
\begin{small}
\begin{align}\label{correNh2YLambdaAsymptofinalform}
Nh_2(Y)=h_{2,0,0}-\sqrt{N}
h_{2,0}(Y)+O\left( \|Y\|^2+\sqrt{N}\|Y\|^4+N^{-\frac{1}{2}} \right)
\end{align}
\end{small}
where
\begin{small}
\begin{equation*}\label{correh2Lambda}
h_{2,0,0}=\sqrt{N}{\rm Tr}(Z_0^{-1}\hat{\Lambda})+n^2-\frac{1}{2}{\rm Tr}(Z_0^{-2}\hat{\Lambda}^2),
 \,h_{2,0}(Y)=z_0^{-1}{\rm Tr}(Y^*Y\hat{Z})+\overline{z_0}^{-1}{\rm Tr}(YY^*\hat{Z}^*).
\end{equation*}
\end{small}

For the factors of $g_2(Y)$ in \eqref{g2YLambdacorrelation}, noting (\ref{gcQdiffexpan3}), we have
\begin{small}
\begin{equation*}\label{correA2asymtobeta2edge}
A_2=Z_0+O\big(N^{-\frac{1}{2}}+\|Y\|\big).
\end{equation*}
\end{small}
from which
\begin{small}
\begin{equation}\label{corredetA2ssquarebeta2edge}
\left( \det(A_2) \right)^{-n-r}=1
+O\big(N^{-\frac{1}{2}}+\|Y\|\big).
\end{equation}
\end{small}
 On the other hand, like  Lemma \ref{determinantcalculationbeta2edge}, taking $\hat{\theta}_i=0$ ($i=1,\cdots,m$), and with \eqref{gcQdiffexpan2} in mind,  
we have 
\begin{small}
\begin{equation}\label{corredeterminantcalcuequa}
\begin{aligned}
\det\left( B_{11}^{(2)} \right)&= \prod_{i=m+1}^\gamma   |z_{0}-\theta_{i}|^{2n \sum_{j=1}^{\alpha_{i}} 
p_{i,j}\beta_{i,j} }                             \\
&\times\bigg\{   \det\begin{bmatrix}
N^{-\frac{1}{4}}\Omega_{11} & \Omega_{12} \\ \Omega_{21} & N^{-\frac{1}{4}}\Omega_{22}
\end{bmatrix}+                                                       
O\Big(
\sum_{(\alpha_1,\alpha_2, \alpha_3)}N^{-\alpha_1}\| Y \|^{\alpha_{2}} \| Q \|^{2\alpha_{3}}
\Big)\bigg\},
\end{aligned}
\end{equation}
\end{small}
where $\alpha_1, \alpha_2, \alpha_3$ satisfy the restriction  
\begin{equation*}
4\alpha_1+\alpha_2+\alpha_3\geq
2n\sum_{i=1}^m\sum_{j=1}^{\alpha_i}\beta_{i,j}+\delta_{\alpha_3,0},
\end{equation*}
where  $\Omega_{ij}$ are defined in  \eqref{F11denotion} and  \eqref{F21F12denotion}. 
Also noting $\Omega_{11}=\Omega_{22}=0$ and
\begin{equation*}\label{corredeterminant}
 \det\begin{bmatrix}
0 & \Omega_{12} \\ \Omega_{21} & 0
\end{bmatrix}=
\Big(\det\big( U_{I_1}\big)  \det\big(\big(U^{-1})_{I_2} \big) \Big)^n
\big(\det(Y^*Y )\big)^{\sum_{i=1}^m\sum_{j=1}^{\alpha_i}\beta_{i,j}},
\end{equation*}
where   $I_{2}$ is the  index set consisting of  numbers corresponding to   last columns of all 
 blocks  
 $R_{p_{i,j}}\left(\theta_i \right)$    from $J_2$  in  \eqref{J2detail}  where  $ 1\leq i \leq m, 1\leq j \leq \alpha_i$.  

Combine \eqref{g2YLambdacorrelation},  \eqref{corredetA2ssquarebeta2edge}   and \eqref{corredeterminantcalcuequa}, we get 
\begin{equation*}\label{g2YLambdacorreexpan}
g_2(Y,\Lambda)=\left(
D^{(2,c)}+O\big(N^{-\frac{1}{4}}+\|Y\|\big)
\right)   \Big(
\big( \det(Y^*Y) \big)^t+
O\Big(
\sum_{(\alpha_1,\alpha_2, \alpha_3)}N^{-\alpha_1}\| Y \|^{\alpha_{2}} \| Q \|^{2\alpha_{3}}
\Big)
\Big),
\end{equation*}
where  $t=\sum_{i=1}^m\sum_{j=1}^{\alpha_i}\beta_{i,j}$ and 
\begin{small}
\begin{equation*}\label{D2ct}
D^{(2,c)}=  \Big(\det\big( U_{I_1}\big)  \det\big(\big(U^{-1})_{I_2} \big) \Big)^n
  |z_{0}-\theta_{i}|^{2n \sum_{j=1}^{\alpha_{i}} 
p_{i,j}\beta_{i,j} }.
\end{equation*}
\end{small}

By   the change of variable $Y\to
N^{-\frac{1}{4}}Y$ in \eqref{Idelta2correlation}, 
as   in the proof of Theorem \ref{2-complex} (i) we can get 
\begin{multline} \label{corre2edgecomplex}
F^{(2)}_{n}(Z,Z)=\frac{e^{n^2}}{\hat{Z}_{K,2}}  |z_{0}-\theta_{i}|^{2n \sum_{j=1}^{\alpha_{i}} 
p_{i,j}\beta_{i,j} }
N^{-\frac{1}{2}n^2- \frac{1}{2}\sum_{i=1}^{m}\sum_{j=1}^{\alpha_{i}}n\beta_{i,j}  }  e^{\sqrt{N}{\rm Tr}(z_0^{-1}\hat{Z}+\overline{z_{0}^{-1}\hat{Z}})}\\
\times  e^{ -\frac{1}{2}{\rm Tr}\big(z_{0}^{-2}\hat{Z}^2+\overline{z_{0}^{-2}\hat{Z}^2}\big) }
\Big( I^{(2)}(z_0^{-1}\hat{Z})
+O\big( \sum_{(\alpha_1,\alpha_3)}N^{-\alpha_1}\| Q \|^{2\alpha_3} \big) \Big),
\end{multline}
where  $I^{(2)}(z_0^{-1}\hat{Z}):= I^{(2)}(0,\ldots,0;z_0^{-1}\hat{Z},z_0^{-1}\hat{Z}) $  is defined  in \eqref{I2GI}
and 
 $\alpha_1, \alpha_3$ satisfy the restriction  
 $4\alpha_1+\alpha_3\geq \delta_{\alpha_3,0}$.

So this completes a modified estimate of 
$\widetilde{g}_2(Q)$.

{\bf Step 3: Summary and matrix integrals}.

For the matrix  integral \eqref{INdeltac},   make the change of variables  
$$
Q_R\to N^{-\frac{1}{2}}Q_R,\quad
Q_L\to N^{-\frac{1}{4}}Q_L  (\hat{U}^{(1,1)})^{-\frac{1}{2}},
$$
and then  integrate out  $Q_R$,     through the calculation we  obtain  
\begin{multline} \label{midcorre2edgecomplex}
\frac{1}{N^n}R_N^{(n)}\left( A_0;z_1,\cdots,z_n \right)
=
(2\pi)^{-\frac{n}{2}}\pi^{-tn-n-n^2}      \rho^n
\\ \times 
\prod_{1\leq i<j \leq n}\left|\hat{z}_i-\hat{z}_j\right|^2
e^{ -\frac{1}{2}{\rm Tr}\left(z_{0}^{-1}\hat{Z}+\overline{z_{0}^{-1}\hat{Z}} \right)^2 }
\left(  I^{(2)}(z_0^{-1}\hat{Z})
\tilde{I}^{(2)}(  z_0^{-1}\hat{Z})
+O\big( N^{-\frac{1}{4}} \big) \right),
\end{multline}
where  for an $n\times t$ matrix $Q$
\begin{small}
\begin{equation}\label{I2C00}
\tilde{I}^{(2)}(\hat{Z})=  \int 
\exp\!\Big\{
-\frac{1}{2}\sum_{i=1}^n
\left( QQ^* \right)_{ii}^2+
{\rm Tr}\left( \big( \hat{Z}^*+\hat{Z} \big) QQ^* \right)
\Big\} {\rm d}Q,
\end{equation}
\end{small}
and $$
\rho:=  |z_{0}-\theta_{i}|^{2 \sum_{j=1}^{\alpha_{i}} 
p_{i,j}\beta_{i,j} }
  \frac{1
  }{ \det(\widetilde{U}^{(22)})
 } \det(U) 
   \det\big(\big(U^{-1})_{I_2} \big)=1.$$
Here the second identity in the latter  follows  from  $\widetilde{U}^{(22)}$ in  \eqref{Ut22},  $(U^{-1})_{I_2}$ and the 
Schur complement  technique. 
 
 Now we need to evaluate the two matrix integrals   $\tilde{I}^{(2)}$ and $I^{(2)}$. For 
 $\tilde{I}^{(2)}$  in \eqref{I2C00},  treated as $n$  integrals separately according  to  its rows  it is simplified to 
\begin{small}
\begin{equation*}
\begin{aligned}
\tilde{I}^{(2)}\big(  z_0^{-1}\hat{Z} \big)&=
\prod_{i=1}^n \int \cdots \int 
\exp\Big\{
-\frac{1}{2}\big( \sum_{j=1}^t\left| q_{1j} \right|^2 \big)^2
+\left(
z_0^{-1}\hat{z}_i+\overline{z_0}^{-1}\overline{\hat{z}_i}
\right)\sum_{j=1}^t\left| q_{1j} \right|^2
\Big\}        \prod_{j=1}^t{\rm d}\Re q_{1j}{\rm d}\Im q_{1j}.
\end{aligned}
\end{equation*}
\end{small}
Apply the  spherical polar coordinates and we can get 
\begin{small}
\begin{equation}\label{midcorre2}
\tilde{I}^{(2)}\big(  z_0^{-1}\hat{Z} \big)= \pi^{(t+\frac{1}{2})n}
\prod_{i=1}^n\Big( e^{\frac{1}{2}( z_0^{-1}\hat{z}_i +\overline{z_0}^{-1}\overline{\hat{z}_i})^2} \mathrm{IE}_{t-1}\big( z_0^{-1}\hat{z}_i +\overline{z_0}^{-1}\overline{\hat{z}_i}\big)\Big),
\end{equation}
\end{small}
where   $\mathrm{IE}_{n}(z) $ is defined in \eqref{IEF}.

For  $I^{(2)}(z_0^{-1}\hat{Z}):= I^{(2)}(0,\ldots,0;z_0^{-1}\hat{Z},z_0^{-1}\hat{Z}) $   defined  in \eqref{I2GI},  by the singular value decomposition:
\begin{small}
\begin{equation}\label{singularcorre}
Y=U\sqrt{R}V,\quad
R={\rm diag}\left( r_1,\cdots,r_n \right),\quad
r_1\geq \cdots \geq r_n \geq 0,
\end{equation}
\end{small}
 the Jacobian reads 
\begin{equation}\label{singularjacobian}
{\rm d}Y=\pi^{n^2}\Big( \prod_{i=1}^{n-1}i! \Big)^{-2}
\prod_{1\leq i<j\leq n}(r_{j}-r_{i})^2 
{\rm d}R{\rm d}U{\rm d}V,
\end{equation}
where 
$U$ and $ V$ are chosen from the unitary group  $\mathcal{U}(n)$ with the Haar measure. 
 Use  Harish-Chandra-Itzykson-Zuber integration formula \cite{HC,IZ}  (see also \cite{Me})
\begin{equation}\label{HCIZ}
\int_{\mathcal{U}(n)}\exp\!\left\{
{\rm Tr}\left( AUBU^* \right)
\right\}{\rm d}U=
\frac{\det\big( [ e^{a_i b_j} ]_{i,j=1}^n \big)}
{\prod_{1\leq i<j\leq n}(a_{j}-a_{i}) (b_{j}-b_{i})} \prod_{i=1}^{n-1}i!,
\end{equation}
with  $A={\rm diag}\left( a_1,\cdots,a_n \right)$ and  $B={\rm diag}\left( b_1,\cdots,b_n \right)$, and also
Andr\'{e}ief's   integration formula \cite{And}   \begin{small}
\begin{equation}\label{Cauchybinet}
 \int  \det\big( [ f_i(x_j) ]_{i,j=1}^n \big)
\det\big( [ g_i(x_j) ]_{i,j=1}^n \big) \prod_{i=1}^n{\rm d}\mu(x_i)=
n!\det\Big( \Big[ \int f_i(x)g_j(x){\rm d}\mu(x) \Big]_{i,j=1}^n \Big),
\end{equation}
\end{small}
we obtain 
\begin{small}
\begin{equation}\label{corremid1}
\begin{aligned}
(2\pi)^{-\frac{n}{2}}\pi^{-n^2} &
\prod_{1\leq i<j \leq n}|\hat{z}_i-\hat{z}_j|^2
e^{ -\frac{1}{2}{\rm Tr}\big(z_{0}^{-1}\hat{Z}+\overline{z_{0}^{-1}\hat{Z}} \big)^2 }
I^{(2)}(  z_0^{-1}\hat{Z}) \\
&
=
\det\left( \left[ \Gamma(t+1) 
e^{-\frac{1}{2}
\left( \left| \hat{z}_i \right|^2+\left| \hat{z}_j \right|^2 \right)
+\hat{z}_i\overline{\hat{z}_j}}
\mathrm{IE}_{t}\left( z_0^{-1}\hat{z}_i+\overline{z_0}^{-1}\overline{\hat{z}_j} \right) 
\right]_{i,j=1}^n \right).
\end{aligned}
\end{equation}
\end{small}

 Combining  \eqref{midcorre2edgecomplex},  \eqref{midcorre2} and \eqref{corremid1}, 
 we thus  give a complete proof of Theorem  \ref{2-complex-correlation}. \end{proof}


 \subsection{ Proof of Proposition \ref{4intrep}}
\begin{proof}[Proof of Proposition \ref{4intrep}]  
{\bf Step 1.} To verify   \eqref{4densityeigen0},  we  use  the fact that any real quaternion matrix can be  triangularized by a unitary quaternion matrix; see e.g. \cite{HCL}. 
 That is,  the Schur decomposition   gives rise to 
\begin{equation}\label{4Schurdecom}
X=V^*\left(\begin{bmatrix}
\Lambda &  \\  & \overline{\Lambda}
\end{bmatrix}+T\right)
V,
\end{equation}
where  $\Lambda=\left( z_1,\cdots,z_{N} \right)$ is diagonal  and $$V= \begin{bmatrix} V^{(1)} & V^{(2)} \\ -\overline{V^{(2)}} & \overline{V^{(1)}}
\end{bmatrix},\quad T= \begin{bmatrix} T_1 & T_2 \\ -\overline{T}_2 & \overline{T}_1 \end{bmatrix},$$ in which     $V$ is a unitary matrix, both    $T_1$ and  $T_2$ are strictly upper triangular.
The Jacobian is given by 
\begin{equation}\label{4SchurJacobian}
{\rm d}X=    
\prod_{i=1}^N\left| z_i-\overline{z}_i\right|^2
\prod_{1\leq i<j\leq N}\left| z_i-z_j \right|^2\left| z_i-\overline{z}_j\right|^2  {\rm d}\Lambda
{\rm d}T_1{\rm d}T_2{\rm d}V,
\end{equation}
see e.g. \cite{AIS}.

 Put 
\begin{equation*}\label{4Bdenotion}
B=VX_0V^*,\quad V^{(i)}=\begin{bmatrix} V_{i1} & V_{i2} \end{bmatrix},\quad
V_1=\begin{bmatrix} V_{11} & V_{21} \\ -\overline{V_{21}} & \overline{V_{11}}\end{bmatrix},
\end{equation*}
where for $i=1,2$, $V_{i1}$ and $V_{i2}$  are  of sizes $N\times r$ and   $N\times (N- r)$ respectively, then 
\begin{equation*}\label{4BS}
B=V_{1}A_{0}V_{1}^*.
\end{equation*}

By \eqref{4Schurdecom}  we  arrive at     
\begin{equation*}\label{4traceexpension}
\begin{aligned}
{\rm Tr}\left( X-X_0 \right)&\left( X-X_0 \right)^*=2{\rm Tr}\left( 
\Lambda\Lambda^* 
\right)            +{\rm Tr}\left( (T-B)(T-B)^* \right)                                 \\
&-{\rm Tr}\left( \begin{bmatrix} \Lambda & \\ & \Lambda^* \end{bmatrix}
 B^*+B\begin{bmatrix} \Lambda^* & \\ & \Lambda \end{bmatrix} \right).
\end{aligned}
\end{equation*}
Integrate out  $V_{12}$ and  $V_{22}$, we get 
\begin{small}
 \begin{equation}\label{4densityMid}
\begin{aligned}
&f_{N,4}\left( A_0; z_1,\cdots,z_{N}\right)=\frac{1}{\widetilde{Z}_{N,r;4}}
e^{-2N\sum_{k=1}^{N}\left| z_k \right|^2} \prod_{i=1}^N\left| z_i-\overline{z_i} \right|^2
\prod_{1\leq i<j\leq N}\left| z_i-z_j \right|^2 \left| z_i-\overline{z_j} \right|^2       \\
& \times 
 \int     \delta\left( V_1^*V_1-\mathbb{I}_{2r} \right) e^{N\Phi_N( \Lambda,B,T)}
  {\rm d}V_{11}{\rm d}V_{21}{\rm d}T_1{\rm d}T_2, 
\end{aligned}
\end{equation}
\end{small}
where 
\begin{equation*}\label{PhiN}
\Phi_N( \Lambda,B,T)=
 -{\rm Tr}\big( (T-B)(T-B)^* \big)                
+{\rm Tr}\bigg(
\begin{bmatrix} \Lambda & \\ & \Lambda^* \end{bmatrix}
 B^*+B\begin{bmatrix} \Lambda^* & \\ & \Lambda \end{bmatrix} \bigg).
\end{equation*}

Again, integrate out the strictly upper triangular 
matrices $T_1$ and $T_2$,  which are indeed Gaussian matrix integrals,  we  immediately get \eqref{4densityeigen0}.

 To compute constants $Z_{N,r;4}$ and $\widetilde{Z}_{N,r;4}$, noting  that they are independent from $A_0$,  choose  $A_0=0$.  For  $Z_{N,r;4}$,  the integral in (\ref{4densityeigen0}) reduces to   the volume  
\begin{equation*}
\int \delta\left( V_1^*V_1-\mathbb{I}_{2r} \right){\rm d}V_{11}{\rm d}V_{21}=
\frac{ \pi^{2Nr+r-r^2}  }
{ \prod_{j=0}^{r-1} \Gamma(2(N-j)) }
\end{equation*}
 which  follows from  the formula \cite[Proposition 3.3]{KL}. 
Combining the fact that $f_{N}\left( 0; z_1,\cdots,z_{N}\right)$ is  a density    
and  the integration formula (see e.g. \cite{Me})
\begin{equation*}
\begin{aligned}
\int_{\Im z_1, \ldots, \Im z_N\geq 0} 
e^{-2N\sum_{k=1}^{N}\left| z_k \right|^2} &\prod_{i=1}^N\left| z_i-\overline{z_i} \right|^2
\prod_{1\leq i<j\leq N}\left| z_i-z_j \right|^2 \left| z_i-\overline{z_j} \right|^2        
d^N z                 \\
&=(2N)^{-N(N+1)}\pi^N N!
\prod_{k=1}^N (2k-1)!,
\end{aligned}
\end{equation*}
we thus arrive at \eqref{4correZNr}. Similarly, we can obtain   \begin{equation*}
\widetilde{Z}_{N,r;4}=(2N)^{-2N^2}N! \pi^{N(N+2r)+r-r^2}\prod_{k=1}^{N-r}\Gamma(2k).
\end{equation*}

{\bf Step 2. }  The goal is to obtain  integral representations for correlation functions.

For $i=1,2$, set 
\begin{equation*}\label{4Gdecompose}
\begin{aligned}
&B=\begin{bmatrix} B_1 & B_2 \\ -\overline{B_2} & \overline{B_1}  \end{bmatrix},\quad
B_i=\left[ B_{i;k,l} \right]_{k,l=1}^2;  V_{i1}=\begin{bmatrix}
Q_i \\ V_{D;i1}
\end{bmatrix},      \\
& \quad                 
 V_D=\begin{bmatrix} V_{D;11} & V_{D;21} \\ -\overline{V_{D;21}} & \overline{V_{D;11}} \end{bmatrix};  \quad \Lambda=\begin{bmatrix} \Lambda_1 & \\  & \Lambda_2  \end{bmatrix};\quad
T_i=\begin{bmatrix} T_{i;11} & T_{i;12} \\  0_{(N-n)\times n} & T_{i;22}  \end{bmatrix},
\end{aligned}
\end{equation*}
where $B_1$ and $B_{i;11}$ are of size $N\times N$ and $n\times n$ respectively; $Q_i$ is  of size $n\times r$; $T_{i;11}$ is  of size $n\times n$ and $\Lambda_1={\rm diag}\left( z_1,\cdots,z_n \right)$.

In view of \eqref{4densityMid}, integrate out  $T_{i;12}$ and we get an extra factor $( 2N/\pi)^{2n(n-N)}$. 
If  we  set \begin{equation*}
\begin{aligned}
&B_U=\begin{bmatrix} B_{1;11} & B_{2;11} \\ -\overline{B_{2;11}} & \overline{B_{1;11}} \end{bmatrix},\quad
B_D=\begin{bmatrix} B_{1;22} & B_{2;22} \\ -\overline{B_{2;22}} & \overline{B_{1;22}} \end{bmatrix},\quad
B_{21}=\begin{bmatrix} B_{1;21} & B_{2;21} \\ -\overline{B_{2;21}} & \overline{B_{1;21}} \end{bmatrix};
\\
&T_U=\begin{bmatrix} T_{1;11} & T_{2;11} \\ -\overline{T_{2;11}} & \overline{T_{1;11}} \end{bmatrix},\quad
T_D=\begin{bmatrix} T_{1;22} & T_{2;22} \\ -\overline{T_{2;22}} & \overline{T_{1;22}} \end{bmatrix},
\end{aligned}
\end{equation*}
then   $\Phi_N( \Lambda,B,T)$ can be divided  
 into three  parts  
 \begin{equation*}\label{PhiN}
\Phi_N( \Lambda,B,T)=\Phi_n( \Lambda_1,B_U,T_U)+ \Phi_{N-n}( \Lambda_1,B_D,T_D)                
-{\rm Tr}\left( B_{21}B_{21}^* \right).
\end{equation*} 

 It's easy to see 
$$ B_D=V_DA_0V_D^*, $$
and 
\begin{equation*}\label{4key}
{\rm Tr}\left( B_{21}B_{21}^* \right)=
{\rm Tr}\left( V_D^*V_DA_0Q^*QA_0^* \right),\quad
V_D^*V_D+Q^*Q=\mathbb{I}_{2r}.
\end{equation*}

After a change of variables  $V_{D}=\widetilde{V}_D\sqrt{\mathbb{I}_{2r} -Q^* Q}$   the integral in \eqref{4densityMid}  can be rewritten as 
\begin{small}
\begin{multline}\label{trans41}
\int \int {\rm d}V_{11}{\rm d}V_{21}
\delta\left( V_1^*V_1-\mathbb{I}_{2r} \right)\Big(\cdot\Big)=\\
\int_{Q^*Q\leq \mathbb{I}_{2r}}{\rm d}Q_1{\rm d}Q_2\left( \det\left( \mathbb{I}_{2r} -Q^* Q \right) \right)^{N-n-r+\frac{1}{2}}
\int {\rm d}\widetilde{V}_{D;11}{\rm d}\widetilde{V}_{D;21}
\delta\left( \widetilde{V}_D^*\widetilde{V}_D-\mathbb{I}_{2r} \right)\Big(\cdot\Big).
\end{multline}
\end{small}
So if we set  $\tilde{z}_k= \sqrt{N/(N-n)} z_{k+n}, k=1,\cdots,N-n,$ and  \begin{small}
\begin{equation*}\label{4regularization}
\begin{aligned}
&\widetilde{A}_0= \sqrt{N/(N-n)}  \sqrt{\mathbb{I}_{2r} -Q^* Q}A_{0} \sqrt{\mathbb{I}_{2r} -Q^* Q},\quad
\widetilde{T}_D= \sqrt{N/(N-n)}  T_D,                   
\end{aligned}
\end{equation*}
\end{small} 
and treat $\{\tilde{z}_k\}$ as eigenvalues of a $(N-n) \times (N-n)$  quaternion  Ginibre matrix  $\widetilde{X}$  perturbed by  a $2r\times 2r$ matrix $\widetilde{A}_0$,  we immediately obtain the desired integral representation from \eqref{4densityMid} and    \eqref{trans41}. 

This  thus completes the proof. 
\end{proof}

 \subsection{Proof of Theorem \ref{4-complex-correlation}}
To verify  Theorem \ref{4-complex-correlation}, we will use the following lemma.
\begin{lemma}\label{maximumlemmabeta4} Given $z_0\in \mathbb{C}$  such that $|z_0|\leq 1$,  let 
\begin{equation}\label{Qlemmadefbeta4}
Z_{0,n}=\begin{bmatrix}
z_0\mathbb{I}_n &  \\   & \overline{z_0}\mathbb{I}_n
\end{bmatrix},\quad
A_0=\begin{bmatrix}
A_1 & A_2 \\ -\overline{A_2} & \overline{A_1}
\end{bmatrix}, \quad Q=\begin{bmatrix}
Q_1 & Q_2 \\ -\overline{Q_2} & \overline{Q_1}
\end{bmatrix}, 
\end{equation}
where  $Q_1$ and  $Q_2$ are  $n\times r$ complex matrices,   and  both $A_1$ and $A_2$  are $r\times r$ complex matrices.
Put 
\begin{small}
\begin{equation}\label{beta4JQdefinition}
\begin{aligned}
J_n(Q,A_0):={\rm Tr}&\big(Z_{0,n}^* QA_0Q^*+Z_{0,n} QA_0^*Q^* \big)
-{\rm Tr}\left( QA_0^*A_0Q^* \right) \\
&+2\sum_{i<j}^n\left(
\big| (QA_0Q^*)_{i,j} \big|^2+\big| (QA_0Q^*)_{i,j+n} \big|^2
\right)
+\log\det\left( \mathbb{I}_{2n}-QQ^* \right),
\end{aligned}
\end{equation}
\end{small}
then $J_n(Q)\leq 0$  whenever 
$QQ^*<\mathbb{I}_{2n}$,  with equality if and only if $Q=0$.

\end{lemma}
\begin{proof}

We proceed  by induction on $n$. 
 For  $n=1$,  note that both  $Q_1$  and $Q_2$ are  row vectors, so the summation  term vanishes   and  we have 
\begin{equation*}
\begin{aligned}
J_{1}(Q;A_0)&\leq
{\rm Tr}\big(Z_{0,n}^* QA_0Q^*+Z_{0,n} QA_0^*Q^* \big)
-{\rm Tr}\left( QA_0^*A_0Q^* \right) -{\rm Tr}\left( QQ^* \right)-\frac{1}{2}{\rm Tr}\left( QQ^* \right)^2    \\
&=-{\rm Tr}\big(\left(   Z_{0,n} ^*Q-QA_{}^* \right)\left(  Z_{0,n} ^*Q-QA_{0}^* \right)^* \big)-(1-|z_0|^2){\rm Tr}( QQ^*)
-\frac{1}{2}{\rm Tr}\left( QQ^* \right)^2\\
&\leq 0.
\end{aligned}
\end{equation*}
  Moreover, 
$J_1(Q;A_0)=0$ if and only if $Q=0$.
 
 Suppose that the lemma is true for  $1, \ldots, n-1$.  Rewrite $Q_1$ and $Q_2$   as   $2 \times 1$  block matrices 
\begin{equation*}\label{beta4Qdecompositionlemma}
Q_i=\begin{bmatrix}
Q_{i1} \\ Q_{i2}
\end{bmatrix},\quad
Q_{11}\, \&\,  Q_{21}\in \mathbb{C}^{(n-1)\times r},\quad
Q_{12}\,\&\, Q_{22}\in \mathbb{C}^{1\times r},
\end{equation*}
 and set 
 \begin{equation*}\label{beta4inductionkey2deno}
\hat{Q}_{i}=\begin{bmatrix}
Q_{1i} & Q_{2i} \\ -\overline{Q_{2i}} & \overline{Q_{1i}}
\end{bmatrix},\quad
i=1,2.
\end{equation*}
One can directly verify  that 
\begin{small}
\begin{equation}\label{4lemma-2}
\begin{aligned}
2\sum_{i<j}^n\left(
\big| (QA_0Q^*)_{i,j} \big|^2+\big| (QA_0Q^*)_{i,j+n} \big|^2
\right)=
& 2\sum_{i<j}^n\left(
\big| (\hat{Q}_{1}A_0 \hat{Q}_{1}^*)_{i,j} \big|^2+\big| 
(\hat{Q}_{1}A_0 \hat{Q}_{1}^*)_{i,j+n} \big|^2
\right)\\
& +{\rm Tr}\left(
\hat{Q}_2 A_0^*   \hat{Q}_{1}^* \hat{Q}_1 A_0 \hat{Q}_{2}^{*} 
\right), \end{aligned}
\end{equation}
\end{small}
and
\begin{small}
\begin{equation}\label{4lemma-3}
\begin{aligned}
{\rm Tr}&\left( Z_{0,n} ^*QA_0Q^*+Z_{0,n} QA_0^*Q^* \right)
-{\rm Tr}\left( QA_0^*A_0Q^* \right)=                   \\
&{\rm Tr}\left( Z_{0,n-1} ^*\hat{Q}_1A_0\hat{Q}_1^*+Z_{0,n-1} \hat{Q}_1A_0^*\hat{Q}_1^* \right)
-{\rm Tr}\left( \hat{Q}_1A_0^*A_0\hat{Q}_1^* \right)       \\
&+{\rm Tr}\left( Z_{0,1} ^*\hat{Q}_2A_0\hat{Q}_2^*+Z_{0,1} \hat{Q}_2A_0^*\hat{Q}_2^* \right)
-{\rm Tr}\left( \hat{Q}_2A_0^*A_0\hat{Q}_2^* \right).
\end{aligned}
\end{equation}
\end{small}

As $QQ^*<\mathbb{I}_{2n}$, we know that  $\hat{Q}_1\hat{Q}_1^*<\mathbb{I}_{2n-2}$,
 from which  use of  Schur complement technique  shows
 
\begin{small}
\begin{equation}\label{4lemma-4}
\log\det\left( \mathbb{I}_{2n}-QQ^* \right)=
\log\det\left( \mathbb{I}_{2n-2}-\hat{Q}_1\hat{Q}_1^* \right)+
\log\det\left( \mathbb{I}_{2}-\hat{Q}_2F\hat{Q}_2^* \right),
\end{equation}
\end{small}
where
\begin{equation*}\label{Flemmadenotion}
F=\mathbb{I}_{2r}+\hat{Q}_1^*\left( \mathbb{I}_{2n-2}-\hat{Q}_1\hat{Q}_1^* \right)^{-1}\hat{Q}_1.
\end{equation*}
Combine\eqref{4lemma-2},  \eqref{4lemma-3} and \eqref{4lemma-4}, 
we   arrive at  \begin{equation} \label{4lemma-5}
J_{n}(Q;A_0)=J_{n-1}(\hat{Q}_1;A_0)+\widetilde{J}(Q),
\end{equation}
 where 
\begin{small}
\begin{equation*}\label{beta4widetildeJQ}
\begin{aligned}
\widetilde{J}(Q)=\log\det\left( \mathbb{I}_{2}-\hat{Q}_2F\hat{Q}_2^* \right) 
&-{\rm Tr}\hat{Q}_2A_0^*\left( \mathbb{I}_{2r}-\hat{Q}_1^*\hat{Q}_1 \right)A_0\hat{Q}_2^*
\\
&+{\rm Tr}\left( Z_{0,1}^*\hat{Q}_2 A_0\hat{Q}_2^*+Z_{0,1} \hat{Q}_2A_0^*\hat{Q}_2^* \right).
\end{aligned}
\end{equation*}
\end{small}

Noting that 
\begin{equation*} 
F^{\frac{1}{2}}\left( \mathbb{I}_{2r}-\hat{Q}_1^*\hat{Q}_1 \right)F^{\frac{1}{2}}=\mathbb{I}_{2r},
\end{equation*}
which can be verified through  the identity
$$ \hat{Q}_1\left( \mathbb{I}_{2r}-\hat{Q}_1^*\hat{Q}_1 \right)=
\hat{Q}_1-\hat{Q}_1\hat{Q}_1^*\hat{Q}_1=\left( \mathbb{I}_{2n-2}-\hat{Q}_1\hat{Q}_1^* \right)\hat{Q}_1,$$
if introducing  
\begin{equation*}
\widetilde{Q}_2=\hat{Q}_2F^{\frac{1}{2}},\quad
\widetilde{A_0}=F^{-\frac{1}{2}}A_0F^{-\frac{1}{2}},
\end{equation*} 
then 
\begin{equation*} 
\widetilde{J}(Q)=J_{1}(\hat{Q}_1;\widetilde{A_0}).
\end{equation*}
 
 Finally,  we see immediately from   \eqref{4lemma-5} and the induction hypothesis that 
  the lemma holds true for $n$. \end{proof}


Now we are ready to prove  Theorem \ref{4-complex-correlation}. 

\begin{proof}[Proof of Theorem \ref{4-complex-correlation}]
  In Proposition \ref{4intrep},  
 let $A_0=PJ_{q}  P^{-1}$ where $J_{q}= J_2  \bigoplus   \overline{J}_2 $ and set     $D=P^*P$. 
First, make a change of variables $ Q\to QP^*$,  and then  by Lemma \ref{maximumlemmabeta4}  we can proceed   as   in Proposition \ref{Analysisform}    to  show  that   for  sufficiently small $\delta>0$ there exists $\widetilde{\Delta}^{(4)}>0$ such that
\begin{equation*}\label{correlationbeta4expansion}
R_{N,4}^{(n)}\left( A_0;z_1,\cdots,z_n \right)=
\widetilde{C}_N^{(4)}(z_1,\ldots,z_n) \left(
\widetilde{I}_{\delta,N}^{(4)}+O\big(e^{-N\widetilde{\Delta}^{(4)}}\big)\right)
\end{equation*}
where 
\begin{equation}\label{zk4C}
z_k:=z_0+ \hat{z}_{k}/\sqrt{N},\quad k=1,2,\ldots,n,
\end{equation}
\begin{equation*}\label{tildeCN4algebra}
\widetilde{C}_N^{(4)}(z_1,\ldots,z_n)=
\frac{1}{C_{N,4}}
 e^{-2N\sum_{k=1}^{N}\left| z_k \right|^2} \prod_{i=1}^N\left| z_i-\overline{z_i} \right|^2
\prod_{1\leq i<j\leq N}\left| z_i-z_j \right|^2 \left| z_i-\overline{z_j} \right|^2
\left( \det D \right)^n,
\end{equation*}
and
\begin{equation}\label{INdeltac4}
\widetilde{I}_{\delta,N}^{(4)}=\int_{{\rm Tr}Q^*Q\leq \delta} \widetilde{g}_4( Q)
\exp\!\big\{
-N \widetilde{f}_4(Q)+\sqrt{N}\widetilde{h}_4( Q)
\big\} {\rm d}Q.
\end{equation}
Here 
\begin{small}
\begin{equation}\label{gcQ4}
\widetilde{g}_4( Q)=\big( 
\det\left(
\mathbb{I}_{2n}-QDQ^*
\big)
\right)^{-n-r+\frac{1}{2}}
{ {\mathbb{E}}_{\mathrm{GinSE}_{N-n}}( \widetilde{A}_0)}
\bigg[
\prod_{i=1}^n\left|
\det\left(
\sqrt{\frac{N}{N-n}}z_i-\widetilde{X}_{N-n}
\right)
\right|^2
\bigg],
\end{equation}
\end{small}
in which
\begin{equation*}\label{tildeA0integrationchange4}
\widetilde{A}_0=\sqrt{N/(N-n)}\sqrt{\mathbb{I}_{2r}-PQ^*QP^*}A_0\sqrt{\mathbb{I}_{2r}-PQ^*QP^*},
\end{equation*}

\begin{small}
\begin{equation*}\label{algebrafcQ4}
\begin{aligned}
\widetilde{f}_4(Q)=&-{\rm Tr}\left(Z_0QJ_{q}^*DQ^*\right)
-{\rm Tr}\left(Z_0^*QDJ_{q}Q^*\right)
+{\rm Tr}\left( QJ_{q}^*DJ_{q}Q^* \right)  \\  
&
-\log\det\left(
\mathbb{I}_{2n}-QDQ^*
\right)
-2\sum_{1\leq i<j\leq n} \left( \big| (QDJ_{q}Q^*)_{i,j} \big|^2   
+ \big| (QDJ_{q}Q^*)_{i,j+n} \big|^2
\right)
\end{aligned}
\end{equation*} 
\end{small}
and 
\begin{equation}\label{algebrahcQ4}
\widetilde{h}_4(Q)=
{\rm Tr}\big( \hat{\Lambda}^* QDJ_{q}Q^*+\hat{\Lambda}QJ_{q}^*DQ^* \big).
\end{equation}
Moreover,  by the Taylor expansion for the log determinant,  we have 
 \begin{small}
\begin{equation}\label{algebrafcQ24}
\begin{aligned}
\widetilde{f}_4(Q) 
=&{\rm Tr}\left( Z_0^*Q-QJ_{q}^* \right)D\left( Z_0^*Q-QJ_{q}^* \right)^*+\frac{1}{2}{\rm Tr}( QDQ^* )^2\\
&
-2\sum_{1\leq i<j\leq n} \left( \big| (QDJ_{q}Q^*)_{i,j} \big|^2   
+ \big| (QDJ_{q}Q^*)_{i,j+n} \big|^2
\right)
+O\left( \| Q \|^6 \right).
\end{aligned}
\end{equation}
\end{small}

By  the Stirling formula 
we easily  see from   \eqref{CN4cor} that  the following estimates hold: for $z_0=\pm 1$, 
\begin{small}
\begin{multline}\label{tildeCNcexpan4R}
\widetilde{C}_N^{(4)}(z_1,\ldots,z_n)=\Big( 1+O\Big( N^{-1} \Big) \Big)
N^{n(2r+1-n)}e^{-2\sqrt{N}{\rm Tr}\big(z_0^{-1}\hat{Z}+\overline{z_{0}^{-1}\hat{Z}}\big)}
e^{-2n^2}(\pi N)^{-\frac{n}{2}}
\\ \times
\pi^{-n-2nr}2^{2nr-n}\left( \det(D) \right)^n
e^{-2\sum_{i=1}^n|\hat{z}_i|^2}\prod_{1\leq i<j \leq n}|\hat{z}_i-\hat{z}_j|^2|\hat{z}_i-\overline{\hat{z}}_j|^2
\prod_{i=1}^n |\hat{z}_i-\overline{\hat{z}}_i|^2,
\end{multline}
\end{small}
and for  $|z_0|=1$ and $\Im z_0>0$,
\begin{small}
\begin{multline}\label{tildeCNcexpan4C}
\widetilde{C}_N^{(4)}(z_1,\ldots,z_n)=
N^{n\left(2r+1-\frac{n-1}{2}\right)}e^{-2\sqrt{N}{\rm Tr}\big(z_0^{-1}\hat{Z}+\overline{z_{0}^{-1}\hat{Z}}\big)}
e^{-2n^2}(\pi N)^{-\frac{n}{2}}\pi^{-n-2nr}
\\ \times
2^{2nr-n}\left( \det(D) \right)^n
e^{-2\sum_{i=1}^n|\hat{z}_i|^2}\Big( |z_0-\overline{z}_0|^{n(n+1)}\prod_{1\leq i<j \leq n}|\hat{z}_i-\hat{z}_j|^2+O\Big( N^{-\frac{1}{2}} \Big) \Big),
\end{multline}
\end{small}
uniformly for   all 
$\hat{z}_{k}$
in a compact subset of $\mathbb{C}$.

The remaining task is to   analyze $\widetilde{I}_{\delta,N}^{(4)}$.  We mainly  focus on  the real case of $z_0=\pm 1$. For this, we proceed in three steps as in the proof of Theorem \ref{2-complex-correlation}.

{\bf Step 1:  Taylor expansions  of $ \widetilde{f}_4(Q)$  and $\widetilde{h}_4(Q)$.}

Noting   $J_2$ in \eqref{J2detail} and   the assumption \eqref{beta2edgethetaexpression}
where $\hat{\theta}_i=0$, $i=1,\cdots,m$.  Combine \eqref{J2decom}-\eqref{J2adeno} and   use the notation therein, we get $J_2=J_2^{(0)}$.  Next, we  are devoted to  establishing  leading terms in \eqref{algebrahcQ4} and  \eqref{algebrafcQ24}.  

For $ \widetilde{f}_4(Q)$  in   \eqref{algebrafcQ24},  noting 
$$ Z_0^*Q-QJ_{q}^*=Q\left( z_0\mathbb{I}_{2r}-J_{q} \right)^*, $$
for convenience let's  introduce 
\begin{small}
\begin{equation*}\label{tildeUdenotion4}
D:=\begin{bmatrix}
D_1 & D_2 \\
-\overline{D}_2 & \overline{D}_1
\end{bmatrix}, \quad 
\widetilde{D}=\big( z_0\mathbb{I}_{2r}-J_{q} \big)^*D
\big( z_0\mathbb{I}_{2r}-J_{q} \big):=\begin{bmatrix}
\widetilde{D}_1 & \widetilde{D}_2 \\
-\overline{\widetilde{D}}_2 & \overline{\widetilde{D}}_1
\end{bmatrix},
\end{equation*}
\end{small}
then we have 
\begin{equation*}\label{DtilDrela}
\widetilde{D}_1=\big( z_0\mathbb{I}_{r}-J_2^{(0)} \big)^*D_1\big( z_0\mathbb{I}_{r}-J_2^{(0)} \big),\quad
\widetilde{D}_2=\big( z_0\mathbb{I}_{r}-J_2^{(0)} \big)^*D_2\big( z_0\mathbb{I}_{r}-\overline{J_2^{(0)}} \big).
\end{equation*}

According to   the Jordan block structure of  $J_2^{(0)}$,  for $\alpha=1,2$, both $D_{\alpha}$ and  $\widetilde{D}_{\alpha}$
admit  a three-layer  structure  
\begin{equation*}\label{Ublockform14}
D_{(\alpha)}=[ D_{k;l}^{(\alpha)} ]_{k,l=1}^{m+1},  
\quad \widetilde{D}_{(\alpha)}=[ \widetilde{D}_{k;l}^{(\alpha)} ]_{k,l=1}^{m+1},
\end{equation*}
where  for the former  $D_{m+1;m+1}^{(\alpha)}$  is of size $r_0\times r_0$ with $r_0=r-\sum_{i=1}^m\sum_{j=1}^{\alpha_i}\beta_{i,j}p_{i,j}$,  and for $k,l=1,\cdots,m$ 
\begin{equation*}\label{Ublockform234}
\begin{aligned}
&D_{k;l}^{(\alpha)}=\left[ D_{k,a;l,b}^{(\alpha)} \right]_{\alpha_{k}\times \alpha_{l}},\quad
D_{k,a;l,b}^{(\alpha)}=\left[ D_{k,a,c;l,b,d}^{(\alpha)} \right]_{\beta_{k,a}\times \beta_{l,b}}, \\
&
D_{m+1;l}^{(\alpha)}=\left[ D_{m+1;l,b}^{(\alpha)}  \right]_{1\times  \alpha_{l}}, 
 \quad
D_{m+1;l,b}^{(\alpha)}=\left[ D_{m+1;l,b,d}^{(\alpha)} \right]_{1\times \beta_{l,b}},\\
&D_{k;m+1}^{(\alpha)}=\left[ D_{k,a;m+1}^{(\alpha)} \right]_{\alpha_{k} \times  1},\quad
D_{k,a;m+1}^{(\alpha)}=\left[ D_{k,a,c;m+1}^{(\alpha)} \right]_{\beta_{k,a}\times 1},
\end{aligned}
\end{equation*}
with   $D_{k,a,c;l,b,d}^{(\alpha)}$, $D_{m+1;l,b,d}^{(\alpha)}$ and $D_{k,a,c;m+1}^{(\alpha)}$ being  of sizes $p_{k,a}\times p_{l,b}$, $r_0\times p_{l,b}$ and $p_{k,a}\times r_0$ respectively.    For the latter, the notations are  similar.  Moreover, it's easy to establish  the following identities for $\alpha=1,2$,
\begin{equation*}\label{UtildeUrelation4}
\begin{aligned}
&\widetilde{D}_{k,a,c;l,b,d}^{(\alpha)}=R_{p_{k,a}}(0)^*D_{k,a,c;l,b,d}^{(\alpha)}R_{p_{l,b}}(0), 
\\
&\widetilde{D}_{m+1;l,b,d}^{(\alpha)}=-\big( z_0\mathbb{I}_{r_0}-J_2^{(2)} \big)^*D_{m+1;l,b,d}^{(\alpha)}R_{p_{l,b}}(0);  \\
&\widetilde{D}_{k,a,c;m+1}^{(1)}=-R_{p_{k,a}}(0)^*D_{k,a,c;m+1}^{(1)}\big( z_0\mathbb{I}_{r_0}-J_2^{(2)} \big);   
\\
&\widetilde{D}_{k,a,c;m+1}^{(2)}=-R_{p_{k,a}}(0)^*D_{k,a,c;m+1}^{(2)}\big( z_0\mathbb{I}_{r_0}-\overline{J_2^{(2)}} \big);   
\\
&\widetilde{D}_{m+1;m+1}^{(1)}=\big( z_0\mathbb{I}_{r_0}-J_2^{(2)} \big)^*
D_{m+1;m+1}^{(1)}\big( z_0\mathbb{I}_{r_0}-J_2^{(2)} \big);              
\\
&\widetilde{D}_{m+1;m+1}^{(2)}=\big( z_0\mathbb{I}_{r_0}-J_2^{(2)} \big)^*
D_{m+1;m+1}^{(2)}\big( z_0\mathbb{I}_{r_0}-\overline{J_2^{(2)}} \big);
\end{aligned}
\end{equation*}
Again,  by the structure of  $R_{p_{k,a}}(0)$,    for $k,l=1,\cdots,m$ set 
\begin{equation}\label{Ublockform44}
\begin{aligned}
&D_{k,a,c;l,b,d}^{(\alpha)}=\begin{bmatrix}
D_{kac;lbd} ^{(\alpha;11)}& D_{kac;lbd} ^{(\alpha;12)} \\
D_{kac;lbd} ^{(\alpha;21)}& D_{kac;lbd} ^{(\alpha;22)}
\end{bmatrix},        \quad
D_{k,a,c;m+1}^{(\alpha)}=\begin{bmatrix}
D_{kac;m+1}^{(\alpha;U)} \\ D_{kac;m+1}^{(\alpha;D)}
\end{bmatrix},\\
&D_{m+1;l,b,d}=\begin{bmatrix}
D_{m+1;lbd}^{(\alpha;L)} & D_{m+1;lbd}^{(\alpha;R)}
\end{bmatrix},\quad
\end{aligned}
\end{equation}
where $D_{kac;lbd} ^{(\alpha;11)}$, $D_{m+1;lbd}^{(\alpha;L)}$ and $D_{kac;m+1}^{(\alpha;U)}$ are of sizes $(p_{k,a}-1)\times(p_{l,b}-1)$, ${r_0}\times(p_{l,b}-1)$ and $(p_{k,a}-1)\times {r_0}$, respectively.
Then simple   calculations show that 
\begin{small}
\begin{equation}\label{UtildeUrelation14}
\begin{aligned}
&\widetilde{D}_{k,a,c;l,b,d}^{(\alpha)}=\begin{bmatrix}
0 & 0 \\
0 & D_{kac;lbd}^{(\alpha;11)}
\end{bmatrix},   \quad  \widetilde{D}_{k,a,c;m+1}^{(1)}=\begin{bmatrix}
0 \\ -D_{kac;m+1}^{(1;U)}\big( z_0\mathbb{I}_{r_0}-J_2^{(2)} \big)
\end{bmatrix},  \\
&\widetilde{D}_{k,a,c;m+1}^{(2)}=\begin{bmatrix}
0 \\ -D_{kac;m+1}^{(2;U)}\big( z_0\mathbb{I}_{r_0}-\overline{J_2^{(2)}} \big)
\end{bmatrix},
\widetilde{D}_{m+1;l,b,d}^{(\alpha)}=\begin{bmatrix}
0 & -\big( z_0\mathbb{I}_{r_0}-J_2^{(2)} \big)^*D_{m+1;lbd}^{(\alpha;L)}
\end{bmatrix}.
\end{aligned}
\end{equation}
\end{small}

For 
$$Q=\begin{bmatrix}
Q_1 & Q_2 \\-\overline{Q}_2 & \overline{Q}_1
\end{bmatrix},$$ 
put for $\alpha=1,2$,
\begin{equation*}\label{Qblock14}
Q_{\alpha}=\begin{bmatrix}
Q_1^{(\alpha)} & \cdots & Q_m^{(\alpha)} & Q_{m+1}^{(\alpha)}
\end{bmatrix},
\end{equation*}
where $Q_{m+1}^{(\alpha)}$ is of size $n\times r_0$ and for $l=1,\cdots,m$ 
\begin{small}
\begin{equation*}\label{Qblock2}
Q_l^{(\alpha)}=[
Q_{l,b}^{(\alpha)}]_{1\times  \alpha_{l}},\quad
Q_{l,b}^{(\alpha)}=\big[
Q_{l,b,d}^{(\alpha)}\big]_{1\times  \beta_{l,b}},\quad
Q_{l,b,d}^{(\alpha)}=\begin{bmatrix}
Q_{lbd}^{(\alpha;L)} & Q_{lbd}^{(\alpha;R)} 
\end{bmatrix},
\end{equation*}
\end{small}
while $Q_{l,b,d}^{(\alpha)}$ and $Q_{lbd}^{(\alpha;L)} $ are of size $n\times p_{l,b}$ and $n\times 1$, respectively.
Now   we can draw some blocks from $Q_{\alpha}$ and  restructure it into  two new  matrices 
\begin{small}
\begin{equation}\label{QUcompression14}
\begin{aligned}
&Q_R=\begin{bmatrix}
Q_{1}^{(1;R)} & \cdots & Q_m^{(1;R)} & Q_{m+1}^{(1)}
& Q_{1}^{(2;R)} & \cdots & Q_m^{(2;R)} & Q_{m+1}^{(2)}
\end{bmatrix}, 
\\
&Q_L=\begin{bmatrix}
Q_{1;L} & Q_{2;L} \\ -\overline{Q_{2;L}} & \overline{Q_{1;L}}
\end{bmatrix},\quad
Q_{\alpha;L}=\begin{bmatrix}
Q_1^{(\alpha;L)} & \cdots & Q_{m}^{(\alpha;L)}
\end{bmatrix},
\end{aligned}
\end{equation}
\end{small}
where 
\begin{small}
\begin{equation*}\label{QUcompression24}
 Q_{l}^{(\alpha;L)}=[Q_{l,b}^{(\alpha;L)}],\  Q_{l,b}^{(\alpha;L)}=[Q_{lbd}^{(\alpha;L)}], \quad 
  Q_{l}^{(\alpha;R)}=[Q_{l,b}^{(\alpha;R)}],\  Q_{l,b}^{(\alpha;R)}=[Q_{lbd}^{(\alpha;R)}].
 \end{equation*}
\end{small}

Similarly,  with \eqref{UtildeUrelation14} in mind, define 
\begin{equation} \label{Ut11R}
 \widetilde{D}^{(11)}=\begin{bmatrix}
 \widetilde{D}_{1;11} & \widetilde{D}_{2;11}  \\
 -\overline{\widetilde{D}_{2;11}} & \overline{\widetilde{D}_{1;11}}
 \end{bmatrix},\quad
\widetilde{D}_{\alpha;11}=\big[ \widetilde{D}_{k,l}^{(\alpha;11)} \big]_{k,l=1}^{m+1},
\end{equation} 
    where for $k,l=1,\cdots,m$,
    \begin{small}
\begin{equation*}\label{QUcompressiondetail14}
\begin{aligned}
&\widetilde{D}_{k,l}^{(\alpha;11)}=\begin{bmatrix}
D_{kac;lbd}^{(\alpha;11)}
\end{bmatrix}, \quad
\widetilde{D}_{m+1,l}^{(\alpha;11)}=\begin{bmatrix}
-\left( z_0\mathbb{I}_{r_0}-J_2^{(2)} \right)^*D_{m+1;lbd}^{(\alpha;L)}
\end{bmatrix}, \quad
\widetilde{D}_{m+1,m+1}^{(\alpha;11)}=\widetilde{D}_{m+1,m+1}^{(\alpha)}
\\
&\widetilde{D}_{k,m+1}^{(1;11)}=\begin{bmatrix}
-D_{kac;m+1}^{(1;U)}\left( z_0\mathbb{I}_s-J_2^{(2)} \right)
\end{bmatrix},\quad
\widetilde{D}_{k,m+1}^{(2;11)}=\begin{bmatrix}
-D_{kac;m+1}^{(2;U)}\left( z_0\mathbb{I}_s-\overline{J_2^{(2)}} \right)
\end{bmatrix}.
\end{aligned}
\end{equation*}
\end{small}
Also,  in \eqref{Ublockform44},  choose the   (1,1)-entry of   $D_{kac;lbd} ^{(\alpha;11)}$ and construct  a matrix $\hat{D}^{(1,1)}$ as follows 
\begin{small}
\begin{equation} \label{U114}
\widehat{D}_{11}=\begin{bmatrix}
\widehat{D}_{1;11} & \widehat{D}_{2;11} \\
-\overline{\widehat{D}_{2;11}} & \overline{\widehat{D}_{1;11}}
\end{bmatrix},\quad
\widehat{D}_{\alpha;11}=[ \hat{D}_{k,l}^{(\alpha)}],\quad
\hat{D}_{k,l}^{(\alpha)}=\left[  \big(D_{kac;lbd} ^{(\alpha; 11)}\big)_{1,1}
\right]_{\alpha_{k}\times \alpha_l}.
\end{equation}
\end{small}

If we set 
$$
D^{-1}:=\begin{bmatrix}
D_1^{(-1)} & D_2^{(-1)}  \\ -\overline{D^{(-1)}_2} & \overline{D^{(-1)}_1}
\end{bmatrix}, $$ 
then 
    $U$ given in \eqref{4UV} and $D=P^*P$ satisfy the following  relations 
\begin{equation}\label{UVDrelationequa}
\begin{aligned}
&\mathbb{J}_rU=-\mathrm{i}D,\quad
\overline{U}_2=\mathrm{i}D_1,\ \ \overline{U}_1=-\mathrm{i}D_2;     \\
&V\mathbb{J}_r=-\mathrm{i}D^{-1},\quad
V_2=\mathrm{i}D_1^{(-1)},\ \ V_1=-\mathrm{i}D_2^{(-1)}.
\end{aligned}
\end{equation}
 Therefore,  we have 
\begin{equation}\label{4RUUU}
-{\rm i}\widehat{D}_{11}=\mathbb{J}_t
\begin{bmatrix}
(U_1)_{I_1} & (U_2)_{I_1}  \\ -\overline{(U_2)_{I_1}} & \overline{(U_1)_{I_1}}
\end{bmatrix}; 
\end{equation}
see Section  \ref{sectnotation} for the meaning of the relevant symbols. 
%

We are ready to do some calculation. 
First,  
\begin{equation}\label{TrQtildeQ4}
{\rm Tr}\left( Z_0^*Q-QJ_{q}^* \right)D\left( Z_0^*Q-QJ_{q}^* \right)^*
={\rm Tr}( Q\widetilde{D}Q^*)=
2{\rm Tr}\big( Q_R\widetilde{D}^{(11)}Q_R^*\big).
\end{equation}
Secondly,  as in  the $\beta=2$ case we can obtain 
\begin{equation}\label{B11asymptoexpan4}
QDJ_qQ^*=z_0Q_L \hat{D}^{(1,1)}Q_L^*+
O\left( \|Q_R\|\|Q\| \right).
\end{equation}
Finally, noting that  combination of  \eqref{QUcompression14} and  \eqref{U114} gives    
\begin{equation*}\label{QUQasymptoexpan4}
QDQ^*=Q_L \hat{D}^{(1,1)}Q_L^*+
O\left( \|Q_R\|\|Q\| \right), 
\end{equation*}
we immediately see from $|z_0|=1$ and $( \hat{D}^{(1,1)})^{*}= \hat{D}^{(1,1)}$ that 
\begin{small}
\begin{equation}\label{fcQsecondexpan4}
\begin{aligned}
&\frac{1}{2}{\rm Tr}\left( QDQ^* \right)^2
-2\sum_{i<j}^n \left( \big| (QDJ_{q}Q^*)_{i,j} \big|^2   
+ \big| (QDJ_{q}Q^*)_{i,j+n} \big|^2
\right)=      \\
&\frac{1}{2}\sum_{i=1}^{2n}
\left( Q_L \hat{D}^{(1,1)}Q_L^* \right)_{i,i}^2
+O\left( \|Q_R\|\|Q\|^3 \right),
\end{aligned}
\end{equation}
\end{small}
since  $ \left( Q_L \hat{D}^{(1,1)}Q_L^* \right)_{i,i+n}=0 $, $i=1,\cdots,n$.

Combining  \eqref{algebrafcQ24}, \eqref{TrQtildeQ4} and \eqref{fcQsecondexpan4}, we  obtain 
\begin{equation*}\label{fcQexpan4}
\widetilde{f}_4(Q)=2{\rm Tr}\big( Q_R\widetilde{D}^{(11)}Q_R^*\big)+\frac{1}{2}\sum_{i=1}^{2n}
\left( Q_L \hat{D}^{(1,1)}Q_L^* \right)_{i,i}^2+O\left( \|Q_R\|\|Q\|^3+\|Q\|^6 \right).
\end{equation*}

Meanwhile,  by   \eqref{B11asymptoexpan4} we get 
\begin{equation*}\label{hcQexpan4}
\widetilde{h}_4(Q)=z_0{\rm Tr}
\Big( \big( \hat{\Lambda}+\hat{\Lambda}^* \big)
Q_L \hat{D}^{(1,1)} Q_L^* \Big)+O\left(\|Q_R\|\|Q\| \right).
\end{equation*}
 
{\bf Step 2: Modified estimate of 
$\widetilde{g}_4(Q)$}.  For the first
factor  in  \eqref{gcQ4},  it's easy to see  
\begin{equation*}\label{detcoeexpan4}
\left( 
\det\left(
\mathbb{I}_{2n}-QDQ^*
\right)
\right)^{-n-r+\frac{1}{2}}=
1+O\left( \| Q \|^2 \right).
\end{equation*}
As for the second factor,   the averaged product of characteristic polynomials has actually been verified in  Theorem \ref{4-complex} (ii), except that   some minor correction  will be needed because of $Q$ and the scaling factor $(N-n)/N$.   So we  repeat almost the same procedure quickly,  with emphasis on  the differences.    
As in the $\beta=2$ case and Proposition \ref{Analysisform}, denote $Z=\sqrt{N/(N-n)}{\rm diag}\left( z_1,\cdots,z_n \right)$ with $K_1=n$, we  can prove that  for sufficiently   small $\delta>0$ and  for $|z_0|=1$,
\begin{small}
\begin{equation*}\label{correlationanalysisform4}
F_{n}^{(4)}(Z,Z)=\frac{1}{\widehat{Z}_{2K_1,1}}\Big(
I_{\delta,N}^{(4)}+O\big( e^{-\frac{1}{2}N\Delta^{(4)}} \big) \Big),
\end{equation*}
\end{small}
with
\begin{equation}\label{Idelta4correlation}
I_{\delta,N}^{(4)}=\int\limits_{{\rm Tr}(YY^*)\leq \delta}\ g_{4}(Y)
\exp\left\{ -Nf_{4}(Y,Z_0)+Nh_{4}(Y) \right\} {\rm d}Y,
\end{equation}
where
\begin{equation}\label{g4YLambdacorrelation}
g_4( Y)={\rm Pf}\begin{bmatrix}
R & S \\ -S^t & T
\end{bmatrix}
\left( \det\left( A_4 \right) \right)^{-n-r},
\end{equation}
\begin{equation*}\label{fbetaYZ0correlation4}
f_{4}(Y)=
\Big( 1-\frac{n}{N} \Big)
{\rm Tr}(YY^*)-\log\det\left( \mathbb{I}_{2n}+YY^* \right).
\end{equation*}

Noting that we need to  replace $\hat{Z}$ with $\hat{Z}+\frac{1}{2}nz_0 N^{-\frac{1}{2}}\mathbb{I}_n+O\left( N^{-1} \right)$ both  in (\ref{Cbeta}) and (\ref{Cbetadetailed}), we then get
\begin{equation*}\label{hbetaYLambdacorrelation4}
h_{4}(Y)= 
\log\det\big( \mathbb{I}_{4n}+N^{-\frac{1}{2}}H_4 \big).
\end{equation*}
Moreover,  we can obtain 
\begin{equation*}
f_{4}(Y)= \frac{1}{2}{\rm Tr}\left( (YY^*)^2 \right)+O\left(
\| Y \|^6+N^{-1}\| Y \|^2
\right), 
\end{equation*}
and 
\begin{small}
\begin{align}\label{correNh4YLambdaAsymptofinalform}
Nh_4(Y)=h_{4,0,0}-\sqrt{N}
h_{4,0}(Y)+O\left( \|Y\|^2+\sqrt{N}\|Y\|^4+N^{-\frac{1}{2}} \right)
\end{align}
\end{small}
where
\begin{small}
\begin{equation*}\label{correh4Lambda}
h_{4,0,0}=2\sqrt{N}{\rm Tr}(Z_0^{-1}\hat{\Lambda})+2n^2-{\rm Tr}(Z_0^{-2}\hat{\Lambda}^2),
 \,h_{4,0}(Y)=2z_0^{-1}{\rm Tr}(YY^*\hat{\Lambda}).
\end{equation*}
\end{small}

For the factors of $g_4(Y)$ in \eqref{g4YLambdacorrelation}, we have
\begin{equation}\label{corredetA4ssquareedgeR}
\left( \det(A_2) \right)^{-n-r}=1
+O\big(N^{-\frac{1}{2}}+\|Y\|\big).
\end{equation}
 On the other hand, with  Lemma \ref{4lemma} taking $\Xi=N^{-1}+\| Q \|^2$ and $\hat{\theta}_i=0$ ($i=1,\cdots,m$), 
we have 
\begin{small}
\begin{equation*}\label{lemmaequ4R}
\begin{aligned}
{\rm Pf}\begin{bmatrix}
R & S  \\  -S^t  &  T
\end{bmatrix}=\prod_{i=m+1}^{\gamma}\prod_{j=1}^{\alpha_i}
\left| z_0-\theta_i \right|^{4np_{i,j}\beta_{i,j}}
&{\rm Pf}\begin{bmatrix}
R\left( Y;(V_1)_{I_2},(U_1)_{I_1} \right) & N^{-\frac{1}{4}}S\left( N^{\frac{1}{4}}Y \right)  
\\  
-N^{-\frac{1}{4}}S\left( N^{\frac{1}{4}}Y \right)^t  &  R\left( Y^*;(U_1)_{I_1},(V_1)_{I_2} \right)
\end{bmatrix}
\\
&+O\left( \sum N^{-\alpha_1}\| Y \|^{\alpha_2} \Xi ^{\alpha_3} \right),
\end{aligned}
\end{equation*}
\end{small}
where  $R\left( Y;A,B \right)$ and $S\left( Y \right)$ are defined in  \eqref{4RRdeno} and  \eqref{4RSdeno}, and  $(\alpha_1,\alpha_2,\alpha_3)$ satisfies  
\begin{equation*}\label{Sumcondi14R}
4\alpha_1+\alpha_2+\alpha_3\geq 4nt+\delta_{\alpha_3,0}.
\end{equation*}
Noting $\Omega_{11}^{{\rm (re)}}=0$, using the Pfaffian identity 
in  Section \ref{sectnotation} we can get
\begin{small}
\begin{equation}\label{corredeterminant4R}
\begin{aligned}
 {\rm Pf}&\begin{bmatrix}
R\left( Y;(V_1)_{I_2},(U_1)_{I_1} \right) & N^{-\frac{1}{4}}S\left( N^{\frac{1}{4}}Y \right)  
\\  
-N^{-\frac{1}{4}}S\left( N^{\frac{1}{4}}Y \right)^t  &  R\left( Y^*;(U_1)_{I_1},(V_1)_{I_2} \right)
\end{bmatrix}=\left(\det (YY^*)\right)^t 
\\
&\times \left(
\det\begin{bmatrix}
(V_1)_{I_2} & (V_2)_{I_2}  \\ -\overline{(V_2)}_{I_2} & \overline{(V_1)}_{I_2}
\end{bmatrix}
\det\begin{bmatrix}
(U_1)_{I_1} & (U_2)_{I_1}  \\ -\overline{(U_2)}_{I_1} & \overline{(U_1)}_{I_1}
\end{bmatrix}
\right)^n,
\end{aligned}
\end{equation}
\end{small}
where  $t=\sum_{i=1}^m\sum_{j=1}^{\alpha_i}\beta_{i,j}$. 

Combine \eqref{g4YLambdacorrelation},  \eqref{corredetA4ssquareedgeR}   and \eqref{corredeterminant4R}, we get 
\begin{equation*}\label{g4YLambdacorreexpan}
g_4(Y,\Lambda)=\left(
D^{(4,{\rm re},c)}+O\big(N^{-\frac{1}{4}}+\|Y\|\big)
\right)   \Big(
\big( \det(Y^*Y) \big)^t+
O\Big(
\sum_{(\alpha_1,\alpha_2, \alpha_3)}N^{-\alpha_1}\| Y \|^{\alpha_{2}} \Xi^{\alpha_{3}}
\Big)
\Big),
\end{equation*}
where 
\begin{small}
\begin{equation*}\label{D4ctR}
D^{(4,{\rm re},c)}=  \left(
\det\begin{bmatrix}
(V_1)_{I_2} & (V_2)_{I_2}  \\ -\overline{(V_2)}_{I_2} & \overline{(V_1)}_{I_2}
\end{bmatrix}
\det\begin{bmatrix}
(U_1)_{I_1} & (U_2)_{I_1}  \\ -\overline{(U_2)}_{I_1} & \overline{(U_1)}_{I_1}
\end{bmatrix}
\right)^n
 \prod_{i=m+1}^{\gamma} |z_{0}-\theta_{i}|^{4n \sum_{j=1}^{\alpha_{i}} 
p_{i,j}\beta_{i,j} }.
\end{equation*}
\end{small}

By   the change of variable $Y\to
N^{-\frac{1}{4}}Y$ in \eqref{Idelta4correlation}, 
as   in the proof of Theorem \ref{4-complex} (i) we  get 
\begin{multline} \label{corre2edgecomplex}
F^{(4)}_{n}(Z,Z)=\frac{e^{2n^2}}{\hat{Z}_{2K_1,1}} D^{(4,{\rm re},c)}
N^{-n^2-\frac{n}{2}- nt  }  e^{2z_0^{-1}\sqrt{N}{\rm Tr}(\hat{Z}+\overline{\hat{Z}})}\\
\times  e^{ -{\rm Tr}\big(\hat{Z}^2+\overline{\hat{Z}^2}\big) }
\Big( 
I^{(4,{\rm re},c)}\left( 2z_0^{-1}\hat{\Lambda} \right)+O\big( \sum_{(\alpha_1,\alpha_3)}N^{-\alpha_1}\Xi^{\alpha_3} \big) \Big),
\end{multline}
where   $(\alpha_1, \alpha_3)$ satisfies  the restriction  
 $$4\alpha_1+\alpha_3\geq \delta_{\alpha_3,0}$$ and  \begin{small}
\begin{equation}\label{IC4R}
I^{(4,{\rm re},c)}(\hat{\Lambda})=
 \int\left(\det (YY^*)\right)^t 
 \exp\Big\{-\frac{1}{2}{\rm Tr}(YY^*)^2 -{\rm Tr}\big(\hat{\Lambda}YY^*\big)\Big\} {\rm d}Y.
\end{equation}
\end{small}

So this completes a modified estimate of 
$\widetilde{g}_4(Q)$.

{\bf Step 3: Summary and matrix integrals}.

For the matrix  integral \eqref{INdeltac4},    using  the change of variables  
$$
Q_R\to N^{-\frac{1}{2}}Q_R,\quad
Q_L\to N^{-\frac{1}{4}}Q_L  (\hat{D}^{(1,1)})^{-\frac{1}{2}},
$$
  integrating  out  $Q_R$,    we  obtain  
\begin{multline} \label{midcorre4edgeR}
\frac{1}{N^n}R_N^{(n)}\left( A_0;z_1,\cdots,z_n \right)
=
(\pi)^{-\frac{n}{2}}\pi^{-2(tn+n+n^2)}2^{2(tn-n+n^2)}      
\rho^n e^{ -{\rm Tr}(\hat{Z}+\overline{\hat{Z}} )^2 } \prod_{i=1}^n |\hat{z}_i-\overline{\hat{z}}_i|^2
\\ \times 
\prod_{1\leq i<j \leq n}|\hat{z}_i-\hat{z}_j|^2|\hat{z}_i-\overline{\hat{z}}_j|^2
\left(  I^{(4,{\rm re},c)}( 2z_0^{-1}\hat{\Lambda} )
\tilde{I}^{(4)}(  z_0^{-1}\hat{\Lambda})
+O\big( N^{-\frac{1}{4}} \big) \right),
\end{multline}
where  for $$Q_L=\begin{bmatrix}
Q_{1;L} & Q_{2;L} \\ -\overline{Q_{2;L}} &  \overline{Q_{1;L} } &
\end{bmatrix}, \quad Q_{1;L}, Q_{2;L}: n\times t, $$,
\begin{small}
\begin{equation}\label{I4C00}
\tilde{I}^{(4)}(\hat{\Lambda})=  \int 
\exp\!\Big\{
-\frac{1}{2}\sum_{i=1}^{2n}
\left( Q_LQ_L^* \right)_{ii}^2+
{\rm Tr}\left( \big( \hat{\Lambda}^*+\hat{\Lambda} \big) Q_LQ_L^* \right)
\Big\} {\rm d}Q_{1;L}{\rm d}Q_{2;L},
\end{equation}
\end{small}
and $$
\begin{aligned}
\rho:=  \prod_{i=m+1}^{\gamma}
&|z_{0}-\theta_{i}|^{2 \sum_{j=1}^{\alpha_{i}} 
p_{i,j}\beta_{i,j} }
  \frac{1}{ \det(\widetilde{D}^{(11)})\det(\hat{D}^{(1,1)})
 } \det(D)    \\
  & \times \det\begin{bmatrix}
(V_1)_{I_2} & (V_2)_{I_2}  \\ -\overline{(V_2)}_{I_2} & \overline{(V_1)}_{I_2}
\end{bmatrix}
\det\begin{bmatrix}
(U_1)_{I_1} & (U_2)_{I_1}  \\ -\overline{(U_2)}_{I_1} & \overline{(U_1)}_{I_1}
\end{bmatrix} =1.
\end{aligned}
$$
Here the second identity in the latter  follows  from  $\widetilde{D}^{(11)}$ in  \eqref{Ut11R},  \eqref{4RUUU}, \eqref{UVDrelationequa} and the 
Schur complement  technique. 
 
 Now we need to evaluate the two matrix integrals   $\tilde{I}^{(4)}$ and $I^{(4,{\rm re},c)}$. For 
 $\tilde{I}^{(4)}$  in \eqref{I4C00},  treated as $n$  integrals separately according  to  its rows  it is simplified to 
\begin{small}
\begin{equation*}
\begin{aligned}
\tilde{I}^{(4)}\big(  z_0^{-1}\hat{\Lambda} \big)&=
\prod_{i=1}^n \int \cdots \int 
\exp\Big\{
-\big( \sum_{j=1}^{2t}\left| \widehat{q}_{1,j} \right|^2 \big)^2
+2z_0^{-1}\left(
\hat{z}_i+\overline{\hat{z}_i}
\right)\sum_{j=1}^{2t}\left|  \widehat{q}_{1,j} \right|^2
\Big\}        
\prod_{j=1}^{2t}{\rm d}\Re \widehat{q}_{1,j}{\rm d}\Im \widehat{q}_{1,j},
\end{aligned}
\end{equation*}
\end{small}
where
$$ \widehat{Q}_L:=\left[ Q_{1;L}, Q_{2;L} \right]:=\left[ \widehat{q}_{1,j} \right]
,\quad l=1,\cdots,n,\ j=1,\cdots,2t.
 $$
Apply the  spherical polar coordinates and we can get 
\begin{small}
\begin{equation}\label{midcorre4}
\tilde{I}^{(4)}\big(  z_0^{-1}\hat{\Lambda} \big)= \pi^{(2t+\frac{1}{2})n}
2^{(-t+\frac{1}{2})n}
\prod_{i=1}^n\Big( e^{ (\hat{z}_i +\overline{\hat{z}_i})^2} 
\mathrm{IE}_{2t-1}\big( z_0^{-1}\sqrt{2}\hat{z}_i +\overline{z_0}^{-1}\sqrt{2}\overline{\hat{z}_i}\big)\Big),
\end{equation}
\end{small}
where   $\mathrm{IE}_{n}(z) $ is defined in \eqref{IEF}.

For  $I^{(4,{\rm re},c)}\left( 2z_0^{-1}\hat{\Lambda} \right)$  defined  in \eqref{IC4R},  by the singular value decomposition:
\begin{small}
\begin{equation}\label{singularcorre4R}
Y=U\sqrt{R}U^t,\quad
R={\rm diag}\left( r_1,\cdots,r_{2n} \right),\quad
r_1\geq \cdots \geq r_{2n} \geq 0,
\end{equation}
\end{small}
 the Jacobian reads 
\begin{equation}\label{singularjacobian4}
{\rm d}Y=\frac{\pi^{n(n+1)}}{\prod_{i=1}^{2n-1}i!}
\left| \Delta (r) \right| {\rm d}R{\rm d}U,
\end{equation}
where 
$U$ is chosen from the unitary group  $\mathcal{U}(2n)$ with the Haar measure, and $\Delta (r)=\prod_{1\leq i<j\leq 2n}(r_j-r_i)$. 
 Use  Harish-Chandra-Itzykson-Zuber integration formula \eqref{HCIZ}
 and also  the Pfaffian  integration formula  \cite{de}
   \begin{small}
\begin{equation}\label{Cauchybinet4R}
 \int  \det\big( [ f_i(x_j) ]_{i,j=1}^{2n} \big)
{\rm sgn}\left( \Delta(x) \right)\prod_{i=1}^{2n} {\rm d}\mu(x_i)=
(2n)!{\rm Pf}\Big( \Big[ a_{i,j} \Big]_{i,j=1}^n \Big),
\end{equation}
\end{small}
where
$$ a_{i,j}=\int_{x\leq y}\left(  
f_i(x)f_j(y)-f_j(x)f_i(y)
\right){\rm d}\mu(x){\rm d}\mu(y), $$
we obtain 
\begin{small}
\begin{equation}\label{corremid14}
\begin{aligned}
I^{(4,{\rm re},c)}\left( 2z_0^{-1}\hat{\Lambda} \right)=&
\frac{\pi^{n(2n+1)}(-z_0)^n
}
{2^{n(2n-1)}\prod_{1\leq i<j \leq n}|\hat{z}_i-\hat{z}_j|^2|\hat{z}_i-\overline{\hat{z}}_j|^2
\prod_{i=1}^n (\overline{\hat{z}}_i-\hat{z}_i)
}                  \\
&\times
{\rm Pf}\begin{bmatrix}
\widetilde{f}(\hat{z}_i,\hat{z}_j) & \widetilde{f}(\hat{z}_i,\overline{\hat{z}}_j) \\
\widetilde{f}(\overline{\hat{z}}_i,\hat{z}_j)
& \widetilde{f}(\overline{\hat{z}}_i,\overline{\hat{z}}_j)
\end{bmatrix}_{i,j=1}^n,
\end{aligned}
\end{equation}
\end{small}
where
$$
\widetilde{f}(z,w)=\int\int_{0\leq x\leq y<\infty} (xy)^t\exp\{-\frac{1}{2}(x^2+y^2)\}
\left(  
e^{-2z_0(zx+wy)}-e^{-2z_0(wx+zy)}
\right){\rm d}x{\rm d}y.
$$
After a change of  variables $x\to (u-v)/\sqrt{2}$, $y\to (u+v)/\sqrt{2}$,  recalling \eqref{fn2} we have
\begin{equation*}\label{F4CR}
\widetilde{f}(z,w)=
2^{2-t}\pi f_{t}(\sqrt{2}z_{0}^{-1}z, \sqrt{2}z_{0}^{-1}w).
\end{equation*}
 Combining  \eqref{midcorre4edgeR},  \eqref{midcorre4} and \eqref{corremid14}, changing 
 $ \hat{z}_i\to \hat{z}_i/\sqrt{2},\quad i=1,\cdots,n, $
in \eqref{zk4C},  
  we  can thus  give a complete proof of Theorem  \ref{4-complex-correlation} at the real edge.

 At the complex edge,  although   $Z_0$ is no longer a scalar matrix,   we still have
\begin{equation*}
\begin{aligned}
&{\rm Tr}\left( Z_0^*Q-QJ_{q}^* \right)D\left( Z_0^*Q-QJ_{q}^* \right)^*  \\
&=2{\rm Tr}\left[
Q_1\left( \overline{z}_0\mathbb{I}_r-J_2^* \right),Q_2\left( \overline{z}_0\mathbb{I}_r-J_2^t \right)
\right]D
\left[
Q_1\left( \overline{z}_0\mathbb{I}_r-J_2^* \right),Q_2\left( \overline{z}_0\mathbb{I}_r-J_2^t \right)
\right]^*             \\
&=2{\rm Tr}\left[ Q_1,Q_2 \right]\widetilde{D}
\left[ Q_1,Q_2 \right]^*,
\end{aligned}
\end{equation*}
where
\begin{equation*}
\widetilde{D}=\begin{bmatrix}
z_0\mathbb{I}_r-J_2 & \\ & z_0\mathbb{I}_r-\overline{J_2}
\end{bmatrix}^*D\begin{bmatrix}
z_0\mathbb{I}_r-J_2 & \\ & z_0\mathbb{I}_r-\overline{J_2}
\end{bmatrix}.
\end{equation*}
Therefore,  we can still admit  a three-layer  structure according to   the Jordan block structure of  $J_2^{(0)}$. 
Based on this observation, following almost the same procedure as in the proof of  Theorem \ref{2-complex-correlation},  we can  complete the proof of Theorem  \ref{4-complex-correlation} at the complex edge. Here we omit the  detailed derivation.  
 \end{proof}


\section{Interpolating  point processes} \label{conclusion}
  In this section we  further   investigate the nature of the transition at, when  some  eigenvalues of the perturbed matrix $A_0$ go to the transition point $z_0$ at a certain rate,     there are interpolating  point processes.  For this, we need more definitions. 

Set \begin{equation}\label{EJ2adeno}
J_2^{(1)}={\rm diag}\Big(
\hat{\theta}_1\mathbb{I}_{\beta_{1,1}},\cdots,\hat{\theta}_m\mathbb{I}_{\beta_{m,1}}
\Big),
\end{equation}
and $t=\sum_{i=1}^m\beta_{i,1}$. Define two more matrix integrals over the space of $n\times t$ and $n\times 2t$ complex matrices   respectively
\begin{small}
\begin{multline}\label{EI2GI}
\widetilde{I}^{(2)}( \hat{Z} )=
 \int 
 \exp\Big\{
 -{\rm Tr}\Big(
 Y
 \big( (P^*P)_{I_1} \big)^{-\frac{1}{2}} \big( J_2^{(1)} \big)^*
 \big( \big( (P^*P)^{-1} \big)_{I_1}  \big)^{-1} J_2^{(1)}
 \big( (P^*P)_{I_1} \big)^{-\frac{1}{2}}
 Y^*
 \Big)
 \Big\}
 \\     \times
\exp\Big\{
-\frac{1}{2}\sum_{i=1}^n\big( YY^* \big)_{i,i}^2+\mathrm{Tr}\big( \big( \hat{Z}+\hat{Z}^* \big)YY^* \big)
\Big\} {\rm d}Y,
\end{multline}
\end{small} with $ I_1=\left\{ 1,\cdots,t \right\}$ and 
\begin{small}
\begin{multline}\label{EI4GIR}
\widetilde{I}^{(4,\mathrm{re})}\big( \hat{Z} \big)=
 \int 
\exp\Big\{
-\sum_{i=1}^n\big( YY^* \big)_{i,i}^2+\mathrm{Tr}\Big( \big( \hat{Z}+\hat{Z}^* \big)YY^* \Big)
\Big\} \times
\\
 \exp\bigg\{
 -2{\rm Tr}\bigg(
 Y
 \big( (P^*P)_{\widetilde{I}_1} \big)^{-\frac{1}{2}}
\begin{bmatrix}
\overline{J_2^{(1)}}   &  \\ &  J_2^{(1)}
 \end{bmatrix}
 \Big( \big( (P^*P)^{-1} \big)_{\widetilde{I}_1}  \Big)^{-1}
  \begin{bmatrix}
 J_2^{(1)}  &  \\ & \overline{J_2^{(1)}}
 \end{bmatrix}
  \big( (P^*P)_{\widetilde{I}_1} \big)^{-\frac{1}{2}}
 Y^*
 \bigg)
 \bigg\}
 {\rm d}Y,
\end{multline}
with $\widetilde{I}_1=\left\{ 1,2,\cdots,t \right\}\bigcup\left\{ n+1,n+2,\cdots,n+t \right\}.$
\end{small}
Here  it should be noted  that    $P$ in \eqref{EI2GI} can be  defined as in \eqref{J2detail}  or \eqref{J4detail}, depending on  the complex or quaternion ensembles,  but  $P$ in \eqref{EI4GIR} only     as in \eqref{J4detail}.

We also need more assumptions.

{\bf Assumption-D}:  For the Jordan form   \eqref{J2detail},  all $\alpha_i=p_{i,1}=1$ for $i=1,\ldots,\gamma$, 
and for a  given  integer $0\leq m\leq \gamma$,     $ \theta_{j}\neq z_0$ is fixed  for $j=m+1, \ldots, \gamma$ but                             
\begin{equation}\label{beta2edgethetaexpressionInter}
\begin{aligned}
&\theta_i=z_0+N^{-\frac{1}{4}}\hat{\theta}_i,\quad
i=1,2,\cdots,m.            
\end{aligned}
\end{equation}

\begin{theorem}\label{Extrathm2}
For the $ {\mathrm{GinUE}}_{N}(A_0)$ ensemble  satisfying {\bf Assumption-D},  let $|z_0|=1$, as $N\to \infty$ scaled eigenvalue correlations      hold uniformly for   all 
$\hat{z}_{k}$
in a compact subset of $\mathbb{C}$
\begin{multline} \label{E4edgecomplex}
\frac{1}{N^n}R_N^{(n)}\Big( A_0;z_0+ \frac{\hat{z}_1}{\sqrt{N}},\cdots, z_0+ \frac{\hat{z}_n}{\sqrt{N}} \Big)=
\left( \det\big( (P^*P)_{I_1} \big)\det\big(\big( (P^*P)^{-1} \big)_{I_1} \big) \right)^{-n}
\\ \times
(2\pi)^{-\frac{n}{2}}\pi^{-n(n+1+t)}
e^{-\frac{1}{2}{\rm Tr}\left( z_0^{-1}\hat{Z}+\overline{z_0^{-1}\hat{Z}} \right)^2}
\prod_{1\leq i<j\leq n}\left| \hat{z_i}-\hat{z}_j \right|^2       
\\\times
\widetilde{I}^{(2)}\left( z_0^{-1}\hat{Z} \right) I^{(2)}\left(  
 -\hat{\theta}_1,\cdots,-\hat{\theta}_m ;
z_0^{-1}\hat{Z},z_0^{-1}\hat{Z}
\right)
+O\big( N^{-\frac{1}{4}} \big),
\end{multline}
where $I^{(2)}$ is defined in \eqref{I2GI}.
\end{theorem}

\begin{proof}[A sketched proof]
Denote $U:=P^*P$, similar to 
 \eqref{J2decom} we 
rewrite $J_2$ as  a sum 
\begin{equation*}\label{EJ2decom}
J_2=J_2^{(0)}+N^{-\frac{1}{4}} J_2^{(1)}\bigoplus 0_{r_0},
\end{equation*}
where
\begin{equation*}\label{EJ2mdeno}
J_2^{(0)}=z_0\mathbb{I}_t
\bigoplus J_2^{(2)}
\end{equation*}
with 
\begin{equation*}\label{EJ2rdeno}
J_2^{(2)}={\rm diag}\left(
\theta_{m+1}\mathbb{I}_{\beta_{m+1,1}},\cdots,\theta_{\gamma}\mathbb{I}_{\beta_{\gamma,1}}
\right).
\end{equation*}
 Not as in the proof of Theorem \ref{2-complex-correlation}, now we have
\begin{equation*}
\begin{aligned}
&(z_0\mathbb{I}_r-J_2)^*U(z_0\mathbb{I}_r-J_2)       \\
&=(z_0\mathbb{I}_r-J_2^{(0)})^*U(z_0\mathbb{I}_r-J_2^{(0)})
-N^{-\frac{1}{4}}\Big(\big( J_2^{(1)} \big)^*\bigoplus 0_{r_0}\Big)U(z_0\mathbb{I}_r-J_2)      \\
&-N^{-\frac{1}{4}}(z_0\mathbb{I}_r-J_2^{(0)})^*U \Big(J_2^{(1)} \bigoplus 0_{r_0}\Big)
+N^{-\frac{1}{2}}\Big(\big( J_2^{(1)} \big)^*\bigoplus 0_{r_0}\Big)
 U\Big(J_2^{(1)} \bigoplus 0_{r_0}\Big).
\end{aligned}
\end{equation*}
Denote 
\begin{equation*}\label{EQUdeno}
U=\begin{bmatrix}
U_{11} & U_{12} \\
U_{21} & U_{22}
\end{bmatrix},\quad
Q=\left[ Q_1,Q_2 \right],
\end{equation*}
where $U_{11}$, $U_{22}$, Q and $Q_1$ are of size $t\times t$, $r_0\times r_0$, $n\times r$ and $n\times t$ respectively. In view of \eqref{correlationbeta2expansion}, we have
\begin{small}
\begin{equation*}
\begin{aligned}
&\widetilde{f}_2(Q)={\rm Tr}Q_2(z_0\mathbb{I}_{r_0}-J_2^{(2)})^*U_{22}(z_0\mathbb{I}_{r_0}-J_2^{(2)})Q_2^* 
\\
&-N^{-\frac{1}{4}}{\rm Tr}Q_1\big( J_2^{(1)} \big)^*U_{12}(z_0\mathbb{I}_{r_0}-J_2^{(2)})Q_2^*
-N^{-\frac{1}{4}}{\rm Tr}Q_2(z_0\mathbb{I}_{r_0}-J_2^{(2)})^*U_{21}J_2^{(1)}Q_1^*
\\
&+N^{-\frac{1}{2}}{\rm Tr}Q_1\big( J_2^{(1)} \big)^*U_{11}J_2^{(1)}Q_1^*
+\frac{1}{2}\sum_{i=1}^n\big( Q_1U_{11}Q_1^* \big)_{i,i}^2+
O\big( \| Q_2 \|\| Q \|^3+\| Q \|^6 \big),
\end{aligned}
\end{equation*}
\end{small}
and
\begin{small}
\begin{equation*}
\widetilde{h}_2( Q)={\rm Tr}\left(
\big( \overline{z_0^{-1}\hat{Z}}+z_0^{-1}\hat{Z} \big)Q_1U_{11}Q_1^*
\right)+N^{-\frac{1}{2}}O\big( \| Q_2 \|\| Q \|+N^{-\frac{1}{4}}\| Q \|^2 \big).
\end{equation*}
\end{small}

Repeat  almost the same  procedure  as in the proof of Theorem \ref{2-complex-correlation}, we can get the conclusion.
\end{proof}

 Similarly, following almost  the same procedure in the proof of Theorem  \ref{4-complex-correlation}, we can verify   
\begin{theorem}\label{Extrathm4}
For the $ {\mathrm{GinSE}}_{N}(A_0)$ ensemble  satisfying {\bf Assumption-D}, as $N\to \infty$ scaled eigenvalue correlations      hold uniformly for   all 
$\hat{z}_{k}$
in a compact subset of $\mathbb{C}$.

(i)If $z_0=\pm 1$, then
\begin{multline} \label{E4edgecomplex4}
\frac{1}{(2N)^n}R_N^{(n)}\Big( A_0;z_0+ \frac{\hat{z}_1}{\sqrt{2N}},\cdots, z_0+ \frac{\hat{z}_n}{\sqrt{2N}} \Big)=
\left( \det\big( (P^*P)_{\widetilde{I}_1} \big)\det\big( (P^*P)^{-1} \big)_{\widetilde{I}_1}  \right)^{-n}
\\ \times
\pi^{-\frac{n}{2}}\pi^{-2n(n+1+t)}2^{n(n-3+2t)}
e^{-\frac{1}{2}{\rm Tr}\left( \hat{Z}+\overline{\hat{Z}} \right)^2}
\prod_{1\leq i<j\leq n}\left| \hat{z_i}-\hat{z}_j \right|^2  \left| \hat{z_i}-\overline{\hat{z}}_j \right|^2   
\prod_{i=1}^n\left| \hat{z_i}-\overline{\hat{z}}_i \right|^2   
\\\times
\widetilde{I}^{(4,\mathrm{re})}\left( \sqrt{2} z_0^{-1}\hat{Z} \right) I^{(4,\mathrm{re})}\left(  
 -\hat{\theta}_1,\cdots,-\hat{\theta}_m ;
\sqrt{2}z_0^{-1}\hat{\Lambda}
\right)
+O\big( N^{-\frac{1}{4}} \big),
\end{multline}
where $I^{(4,\mathrm{re})}$ is defined in \eqref{I4GIR}.

(ii)If $\Im z_0>0$ and $ |z_0|=1$, then
\begin{multline} \label{E4edgecomplex4C}
\frac{1}{(2N)^n}R_N^{(n)}\Big( A_0;z_0+ \frac{\hat{z}_1}{\sqrt{2N}},\cdots, z_0+ \frac{\hat{z}_n}{\sqrt{2N}} \Big)=
\left( \det\big( (P^*P)_{I_1} \big)\det\big( (P^*P)^{-1} \big)_{I_1}  \right)^{-n}
\\ \times
\pi^{-\frac{n}{2}}\pi^{-2n(n+1)-nt}2^{\frac{n(n-1-t)}{2}-n}(-1)^{nt}
e^{-\frac{1}{2}{\rm Tr}\left( z_0^{-1}\hat{Z}+\overline{z_0^{-1}\hat{Z}} \right)^2}
\prod_{1\leq i<j\leq n}\left| \hat{z_i}-\hat{z}_j \right|^2   
\\\times
\widetilde{I}^{(2)}\left( z_0^{-1}\hat{Z} \right) 
I^{(4)}\left(  
 -\hat{\theta}_1,\cdots,-\hat{\theta}_m ;
\sqrt{2}z_0^{-1}\hat{Z},\sqrt{2}z_0^{-1}\hat{Z}
\right)
+O\big( N^{-\frac{1}{4}} \big),
\end{multline}
where $I^{(4)}$ is defined in \eqref{I4GI}, $\widetilde{I}^{(2)}$ is defined in \eqref{EI2GI}, while we need to replace $J_2^{(1)}$ with $2^{\frac{1}{4}}J_2^{(1)}$.
\end{theorem}

 \noindent{\bf Acknowledgements}  
 We  would like to thank     Elton P. Hsu for  his  encouragement and support, and also  to thank Y.V. Fyodorov,  J. Grela and Y. Wang for useful comments on the  first arXiv version. This  work 
 was   supported by  the National Natural Science Foundation of China \#12090012.
 


%




\end{document}